\def\revisionoff{1}
\documentclass{article}
\usepackage{graphicx} \usepackage{xcolor}
\usepackage{tikz}
\usetikzlibrary{calc}

\usepackage[margin=1in]{geometry}
\usepackage{amssymb,amsmath,amsthm}
\usepackage{longtable,booktabs}
\usepackage{seqsplit}
\usepackage{placeins}
\usepackage{subcaption}
\usepackage{needspace}
\usepackage[hidelinks]{hyperref}
\newtheorem{lemma}{Lemma}
\newtheorem{theorem}{Theorem}
\newtheorem{problem}{Problem}
\newtheorem{proposition}{Proposition}
\newtheorem{definition}{Definition}
\newtheorem{claim}{Claim}

\usepackage{algorithm}
\usepackage{algorithmicx}
\usepackage{algpseudocode}
\title{Simultaneous Matrix Multiplication by Tensor Contraction Aggregation}
\author{Edgar Solomonik\thanks{Siebel School of Computing and Data Science,
University of Illinois Urbana-Champaign, USA. \texttt{solomon2@illinois.edu}.}
\and Markus Bl\"aser\thanks{Department of Computer Science, Saarland University,
Saarbr\"ucken, Germany. \texttt{mblaeser@cs.uni-saarland.de}.}}
\date{}

\DeclareMathOperator{\vcr}{vec}

\newcommand{\ivl}[2]{[\![#1,#2]\!]}

\newif\ifshowrevision
\ifdefined\revisionoff
  \showrevisionfalse
\else
  \showrevisiontrue
\fi
\newif\ifrevisionregion
\definecolor{revisionblue}{RGB}{0,70,155}
\newenvironment{revision}
  {\revisionregiontrue\ifshowrevision\color{revisionblue}\captionsetup{font+={color=revisionblue}}\fi}
  {}

\makeatletter
\let\STCfloatboxreset\@floatboxreset
\def\@floatboxreset{\STCfloatboxreset
  \ifshowrevision\ifrevisionregion\color{revisionblue}\fi\fi}
\makeatother
\providecommand{\brankop}{\underline{R}}

\begin{document}

\maketitle

\begin{abstract}
We present an interpretation of the fast matrix multiplication algorithms of Coppersmith and Winograd as aggregations of tensor contractions. This point of view provides a direct connection to older constructions by Pan and to more recent work on bilinear algorithms for symmetric tensor contraction. Further, we provide a simple explicit approach for exact aggregation of tensor contractions, each one yielding an independent matrix-matrix product. These constructions and border rank adaptations thereof do not come close to improving the best known bounds on $\omega$ (the matrix multiplication exponent), but provide explicit moderate-size fast matrix multiplication algorithms that achieve a better exponent than Strassen's algorithm and have potential practical value. These finite constructions improve the multiplication-count bounds of the reference algorithms compared here for selected numbers and dimensions of independent square matrix products.
In particular, using $\langle m,m,m\rangle$ to denote the matrix multiplication tensor for matrices of dimension $m$, we obtain constructions that give bounds such as the border rank bound
\[
\underline R\!\left(\langle n^2,n^2,n^2\rangle^{\oplus9}\right)
\le (n+1)^6+6(n+1)^5,
\]
and the exact rank bound
\[
R\!\left(\langle n^3,n^3,n^3\rangle^{\oplus14}\right)\le (n+1)^9.
\]
We quantify the accuracy of such constructions in finite precision, and provide border-rank constructions that reduce cost with limiting retained fractions of accuracy at least $1/2$, $1/3$, or $1/4$ relative to the classical matrix-multiplication approach under the stated error model.

\end{abstract}

\section{Introduction}

Independent matrix--matrix products are a key subject in the theoretical study of the matrix multiplication exponent, $\omega$.
Sch\"onhage's asymptotic sum theorem~\cite{schonhage1981partial} gives upper bounds on $\omega$ based on constructions for independent matrix--matrix products, and the extensively studied method of Coppersmith and Winograd~\cite{coppersmith1990matrix} provides efficient border-rank constructions for such sets of products.
Exact-rank bounds for balanced restrictions of these tensor powers, and hence for balanced compatible families of independent products, were established by Flavi, Jelisiejew, and Micha{\l}ek~\cite{flavi2025symmetric}.
Their work builds on the group-tensor formulation of Alman and Vassilevska Williams~\cite{alman2018further}.
At the same time, independent matrix--matrix products are in themselves a common computational kernel, frequently arising in dense and sparse linear algebra, for instance when multiplying block-sparse matrices~\cite{borstnik2014sparse}.
In recent years, standard software interfaces and libraries have been developed for independent matrix--matrix multiplications~\cite{abdelfattah2021batched}.
In this paper, we develop a new, more practical point of view on the Coppersmith--Winograd construction by considering aggregation of tensor contractions.

\Needspace{8\baselineskip}
Aggregation of two matrix multiplications is a technique employed by Victor Pan in his matrix multiplication constructions~\cite{pan1978strassen,pan2017neighbourhood}.  
We recall his construction here as a trilinear form identity.
For $A,B,D,U,V,W\in\mathbb{C}^{n\times n}$, define
\begin{align}
S&=\sum_{i,j,h}(a_{ij}+u_{jh})(b_{jh}+v_{hi})(d_{hi}+w_{ij}),\\
p_{ih}&=\sum_j(a_{ij}+u_{jh}),\qquad
q_{ij}=\sum_h(b_{jh}+v_{hi}),\qquad
r_{jh}=\sum_i(d_{hi}+w_{ij}).
\end{align}
Then
\begin{align}
\operatorname{tr}(ABD)+\operatorname{tr}(UVW)
=S
&-\sum_{i,j}a_{ij}q_{ij}w_{ij}
-\sum_{j,h}u_{jh}b_{jh}r_{jh}
-\sum_{h,i}p_{ih}v_{hi}d_{hi}.
\label{eq:intro-pan-pair}
\end{align}
Hence, two independent products of $n\times n$ matrices may be computed using $n^3+3n^2$ bilinear multiplications.
Leveraging border rank, Pan reduces the count further to $n^3+2n^2$~\cite[Eq.~(3.1)]{pan1982implicit}.

In his 1982 work, Pan remarks that his construction computes $AB+BA$ with fewer nonscalar multiplications than his construction for $AB$~\cite[Remark~6.1, p.~33]{pan1982implicit}.
In recent years, a family of algorithms has also been introduced for symmetric tensor contraction~\cite{solomonik2021fast}.
For example, for symmetric $A,B\in\mathbb{R}^{n\times n}$, $C=AB+BA$ may be computed by forming the symmetrized aggregate
\begin{equation}
z_{ijk} = (a_{ij} + a_{ik} + a_{jk})(b_{ij} + b_{ik} + b_{jk}),\qquad i<j<k,
\end{equation}
and applying lower-order corrections, using a total of $n^3/6 + O(n^2)$ bilinear multiplications~\cite{solomonik2021fast}.
A similar sum over subsets works for higher-order tensor contractions~\cite{solomonik2021fast}.
Instead of symmetrization, we consider an analogous aggregation of independent tensor contractions.
In Section~\ref{sec:simultaneous-contraction}, we show that aggregation of independent tensor contractions is exactly the task performed by the filtered powers of the little and big Coppersmith--Winograd tensors.

We give a direct tensor-contraction formulation of exact decompositions for compatible sets of independent matrix--matrix products.
We achieve this using tensor products of a simple centering embedding from $\mathbb C^n$ into the zero-sum subspace of $\mathbb C^{n+1}$.
We further augment the centering embedding with a weighting, yielding a weighted embedding from $\mathbb{C}^{n}$ to $\mathbb{C}^{n+2}$ that eliminates both linearly reducible and replicated cross terms.
For square matrix products, each contraction is a trilinear form corresponding to a balanced tripartition of $3N$ indices, and our approach can simultaneously compute any set of such contractions, so long as no recombination of parts from different selected tripartitions is itself a tripartition.
This combinatorial condition is the same one addressed by Coppersmith and Winograd using Salem--Spencer sets~\cite{coppersmith1990matrix,salemspencer1942}.

Besides the exact-rank constructions for independent square matrix products, we also consider aggregation of contractions corresponding to incomplete and unbalanced tripartitions.
We show that this space includes the big Coppersmith--Winograd tensor construction~\cite{coppersmith1990matrix} (Section~\ref{sec:cw-contractions}).
Further, using a standard border-rank degeneration approach~\cite{bini1980approximate}, we loosen the compatibility conditions on tripartitions, enabling aggregation of a larger number of independent matrix--matrix products.

We quantify the parameter conditioning of these graded constructions, estimating the fraction of working-precision digits retained.
Our bounds quantify the sensitivity of the bilinear algorithms obtained by the border-rank (graded) constructions to perturbations of their factor maps and bilinear products.
These characterizations enable us to target constructions with a specified limiting fraction of retained accuracy, such as half the digits guaranteed by the classical error bound as working precision increases.

With the general aggregation framework established, we then use numerical search to find sets of tripartitions that are compatible for aggregation.
The asymptotic exponential growth rate of classical conflict-free tripartition families is characterized by Coppersmith and Winograd~\cite{coppersmith1990matrix}, while numerical search enables us to find constructions of moderate size.
Our constructions improve on the prior exact and border-rank multiplication-count bounds compared here for certain moderate numbers $p$ of independent matrix--matrix products and matrix dimensions $q$.

In one of the simplest new constructions, we show that 14 matrix multiplications of matrices of dimension $n^3$ can be aggregated and computed with $(n+1)^9$ bilinear multiplications, instead of the naive $14n^9$ (Appendix~\ref{apx:construction-certificates}).
We catalogue and compare the new constructions to prior art in Section~\ref{sec:frontier}.
Figure~\ref{fig:equalsize} shows an overview of these results, comparing selected prior finite-base constructions with the new variants at finite $p$ and $q$.

The largest exact construction we consider computes $p=105$  independent products of dimension $q=n^4$, while graded constructions give $p=125,133,141$ with limiting fractions of retained precision at least $1/2,1/3,1/4$, respectively (Section~\ref{sec:d12-search}).
Table~\ref{tab:intro-speedups} gives multiplication-count reduction factors of the new construction over LITA, Strassen, and classical multiplication at $q=11^4=14641$.
For example, the exact fourteen-product construction improves the padding-free recursive estimate for LITA~\cite{khoruzhii2026lita} at $q=7^3=343$ by a factor of $1.029$ (Section~\ref{sec:plot-models}).

\begin{table}[!htbp]
\centering\small
\begin{tabular}{lrc rrr}
\toprule
Construction & $p$ & retained precision & over LITA & over Strassen & over classical\\
\midrule
Exact & 105 & exact & $1.44\times$ & $2.30\times$ & $14.61\times$\\
Border rank & 125 & $\theta\ge1/2$ & $1.71\times$ & $2.73\times$ & $17.31\times$\\
Border rank & 133 & $\theta\ge1/3$ & $1.83\times$ & $2.91\times$ & $18.48\times$\\
Border rank & 141 & $\theta\ge1/4$ & $1.93\times$ & $3.08\times$ & $19.52\times$\\
\bottomrule
\end{tabular}
\caption{Multiplication-count reduction factors at $q=11^4=14641$, the largest dimension $n^4\le16000$ with integer $n$.
The consecutive reference counts are $pq^{\log_{32}14197}$ (LITA~\cite{khoruzhii2026lita}),
$pq^{\log_2 7}$ (Strassen~\cite{strassen1969gaussian}), and $pq^3$ (classical).
The first two are padding-free recursive models; the third is the classical count.
The precision column lists limiting fractions of classical reference
accuracy from Section~\ref{sec:stability};
\eqref{eq:finite-precision-fraction} gives the comparison at fixed
working precision. All entries are multiplication-count ratios rounded to two decimal places.}
\label{tab:intro-speedups}
\end{table}

\FloatBarrier

\section{Notation}
\label{sec:notation}

We will be interested in characterizing the complexity of (multiple) matrix products and tensor contractions.
We generally refer to $A$ and $B$ as the input matrices/tensors and to $C$ as the output.
Any tensor contraction can be written as a product of (possibly rectangular) matrices by combining modes (dimensions) of the tensors shared between $(A,B)$, $(B,C)$, and $(A,C)$.
We associate matrix multiplication with a trilinear form by pairing its output with an independent matrix of variables $Z$:
\begin{equation}
\sum_{i,j=0}^{n-1}(XY)_{ij}Z_{ij}
=\sum_{i,j,k=0}^{n-1}X_{ik}Y_{kj}Z_{ij}.
\end{equation}
This trilinear form is induced by the same tensor as the bilinear map of matrix multiplication.
The trilinear form is more convenient for deriving low-rank decompositions of the tensor, which can be used to apply the bilinear map efficiently.

For positive integers $a,b,c$, we write
\begin{equation}
\langle a,b,c\rangle
=
\sum_{i=0}^{a-1}\sum_{j=0}^{b-1}\sum_{k=0}^{c-1}
X_{ij}Y_{jk}Z_{ki}
\end{equation}
for the trilinear form representing the matrix multiplication tensor corresponding to the product of an
$a\times b$ matrix and a $b\times c$ matrix. Here $Z$ is indexed in transposed order relative to the entrywise pairing above. We identify a trilinear form with its coefficient tensor in the chosen bases.

A rank-one trilinear form is a
product of one linear form in each variable space. For a trilinear form $T$, we write $R(T)$ for the
least number of rank-one forms whose sum is $T$, and $\underline R(T)$ for the
border rank, the least $r$ for which $T$ is a limit of tensors of rank at most
$r$.
  The direct sum $T_1\oplus\cdots\oplus T_p$ uses disjoint variable spaces;
in particular, $\langle a,b,c\rangle^{\oplus p}$ denotes $p$ independent
matrix products of the same dimensions.

We write multi-indices in bold, for example $\mathbf i=(i_0,\ldots,i_{d-1})$ for a tensor of order $d$.
We use the interval notation $\ivl 0{d-1} = \mathbb{Z} \cap [0, d-1]$.
Throughout the paper, $F=\ivl{0}{d-1}$ denotes the set of positions in a $d$-component multi-index, with $d$ specified for each construction. Thus $r\in F$ labels a position, while $i_r$ is its index value.
All set complements are relative to this coordinate set: $S^c=F\setminus S$.

\begin{definition}[Tensor Contraction]
For $s,t,v\in\mathbb{Z}_{\geq0}$, the contraction of tensors $A\in(\mathbb{C}^n)^{\otimes(s+v)}$ and $B\in(\mathbb{C}^n)^{\otimes(v+t)}$ into a tensor $C\in (\mathbb{C}^n)^{\otimes(s+t)}$ may be written as
\begin{equation}
C_{\mathbf i\mathbf j}
 =\sum_{\mathbf k\in\ivl{0}{n-1}^v}A_{\mathbf i\mathbf k}B_{\mathbf k\mathbf j},
\qquad
\mathbf i\in\ivl{0}{n-1}^s,\quad\mathbf j\in\ivl{0}{n-1}^t.
\end{equation}
The contraction is a bilinear map corresponding to the trilinear form,
\begin{equation}
\mathcal{T}^{s,t,v}_n(X,Y,Z)=\sum_{\mathbf i,\mathbf j,\mathbf k} X_{\mathbf{i} \mathbf{k}} Y_{\mathbf{k} \mathbf{j}} Z_{\mathbf{i} \mathbf{j}}.
\end{equation}
\end{definition}
We use restriction notation to select coordinates of a multi-index.
\begin{definition}[Index restriction]
\label{def:ief}
Let $S=\{j_0<\cdots<j_{k-1}\}\subseteq F$.
For $\mathbf i=(i_0,\ldots,i_{d-1})$, define
\begin{equation}
\mathbf i|_S=(i_{j_0},\ldots,i_{j_{k-1}}).
\end{equation}
The restriction to the empty set is the empty tuple.
\end{definition}

The principal notation used in the constructions and their error analysis is
summarized in Table~\ref{tab:notation-guide}.
\begin{table}[!htbp]
\centering\small
\renewcommand{\arraystretch}{1.12}
\begin{tabular}{@{}p{0.25\linewidth}@{\hspace{0.02\linewidth}}p{0.57\linewidth}@{\hspace{0.02\linewidth}}p{0.14\linewidth}@{}}
\toprule
Notation & Meaning & Defined in\\
\midrule
$F=\ivl{0}{d-1}$ & Coordinate positions in a $d$-component multi-index & \S\ref{sec:notation}\\
$\mathcal S;\ U_t,V_t,W_t$ & Selected tripartitions and the three sets of omitted coordinates for product $t$ & \S\ref{sec:aggregation}\\
$\eta_r(i,j,k)$ & Number of the three selected blocks that omit coordinate $r$ & \S\ref{sec:aggregation}\\
$\mu,\ \nu,\ \Omega_{S,F}$ & Centering; padding; insertion of weighted coordinates $S$ & Defs.~\ref{def:tensor-centering}--\ref{def:weight}\\
$\Phi_{S,F}$ & Combined centering, padding, and weighting embedding & Def.~\ref{def:weight}\\
$\Gamma_S^P$ & Centered lifting on the remaining coordinates $F\setminus P$ & \eqref{eq:centered-lifting}\\
$\mathcal R_{ijk}$ & Coordinates omitted by all three selected blocks & Def.~\ref{def:mixed-form}\\
$\phi_{\mathcal S}$ & Target sum of independent contractions & \eqref{eq:phi-S-labeled}\\
$\Phi_{\mathcal S},\ \Phi_{\mathcal S,\mathcal Q}(\lambda)$ & Aggregate with corrections $\mathcal P(\mathcal S)$; scaled aggregate with selected corrections & \eqref{eq:master-filter}, \eqref{eq:partial-filter}\\
$\mathcal A_{\mathcal S}(P)$ & Aggregate with coordinates $P$ omitted; $P=\emptyset$ gives the main aggregate & Def.~\ref{def:centered-aggregate}\\
$\mathcal P(\mathcal S),\ \mathcal B(\mathcal S)$ & Correction subsets for all replicated terms; unwanted triples remaining after these corrections & Defs.~\ref{def:correction-support}, \ref{def:survivor-set}\\
$\mathcal Q_{\le0}$ & Correction sets required by replicated mixed terms of nonpositive degree & Def.~\ref{def:degree-corrections}\\
$a_t,b_t,c_t,\ h$ & Signed block exponents with desired degree zero; factor-growth exponent after normalizing the smallest surviving degree to one & Defs.~\ref{def:integral-grading}, \ref{def:precision-order}\\
$c_{\mathcal Q}(R),\ \mathcal J$ & Residual multiplier after the corrections; remaining mixed triples & \eqref{eq:partial-correction-coefficient}, Def.~\ref{def:aggregate-index-sets}\\
$\rho_{\mathcal S,\mathcal Q}(n)$ & Multiplication budget of the evaluated aggregates & \eqref{eq:partial-correction-budget}\\
$e_{ijk}$ & Degree $a_i+b_j+c_k$ of a remaining mixed term & Def.~\ref{def:parameter-dependent-map}\\
$\mathcal M,\ \mathcal M_\lambda$ & Target bilinear map and its parameter-dependent approximation & \eqref{eq:target-bilinear-map}, Def.~\ref{def:parameter-dependent-map}\\
$\Xi=(A,B,C),\ \Xi(\lambda)$ & Bilinear algorithm and its parameter-dependent factor matrices & \eqref{eq:cp-evaluation}, \eqref{eq:parameterized-algorithm}\\
$\lambda,\ \epsilon$ & Approximation parameter; relative input and rounding-error bound & \S\ref{sec:approximation-error}, Lem.~\ref{lem:rounded-bilinear}\\
$\beta(\lambda)$ & Approximation-error bound including residual multipliers & Def.~\ref{def:approximation-bound}\\
$\sigma(\Xi),\ \sigma_c$ & Spectral-norm product of a bilinear algorithm; classical reference value & Def.~\ref{def:spectral-product}, Lem.~\ref{lem:classical-spectral-product}\\
$\sigma_0,\ \beta'(0)$ & Leading round-off amplification and approximation coefficients & Def.~\ref{def:leading-sigma}, \eqref{eq:leading-error-balance}\\
$e(\epsilon,x,y),\ e_c(\epsilon,x,y)$ & Relative-error bounds for the chosen parameter rule and the classical algorithm & Def.~\ref{def:retained-accuracy}\\
$\theta(\epsilon,x,y),\ \theta$ & Fraction of classical-reference accuracy and its limit as $\epsilon\to0$ & Def.~\ref{def:retained-accuracy}\\
$\lambda_{\rm bal}(\epsilon)$ & Parameter rule giving $\theta=1/(1+h)$ & Lem.~\ref{lem:retained-accuracy}\\
\bottomrule
\end{tabular}
\caption{Principal notation. The referenced sections and definitions specify each construction; this table is a lookup guide.}
\label{tab:notation-guide}
\end{table}

\section{Motivation: Symmetric Tensor Contractions}
\label{sec:mot}
Our starting point is a family of fast bilinear algorithms for symmetric tensor contractions~\cite{solomonik2021fast}.
These algorithms reduce bilinear rank by forming a fully symmetrized trilinear aggregate over the index space of the contraction and extracting the desired terms from its partial sums.
We use the same treatment of lower-order terms to construct trilinear aggregation algorithms for simultaneous matrix multiplication.

Consider a product between a symmetric matrix $A\in\mathbb{C}^{n\times n}$ and a vector $b\in\mathbb{C}^n$.
The trilinear form corresponding to the matrix vector product $Ab$ is naively a sum of $n^2$ distinct terms, involving $\binom{n+1}{2}$ variables $X_{ij}$ with $j\leq i$, i.e.,
\begin{equation}
T_\text{symv}(X,Y,Z) = \sum_{i=0}^{n-1} \bigg[X_{ii}Y_iZ_i + \bigg(\sum_{j=0}^{i-1} X_{ij}Y_iZ_j + \sum_{j=0}^{i-1} X_{ij}Y_jZ_i\bigg)\bigg].
\end{equation}
The symmetric matrix--vector tensor has rank at most $\binom{n+1}{2}$,
\begin{equation}
T_\text{symv}(X,Y,Z) = \sum_{i=0}^{n-1} \bigg[(X_{ii}-\sum_{j=0}^{i-1}X_{ij}-\sum_{j=i+1}^{n-1} X_{ji})Y_iZ_i  + \sum_{j=0}^{i-1} X_{ij}(Y_i+Y_j)(Z_i+Z_j) \bigg].
\end{equation}

The sum over pairs $(i,j)$ with $j<i$ in this low-rank expansion corresponds to a `symmetrized' form of the tensor contraction, i.e., $X_{ij}(Y_i+Y_j)$ is symmetric in $(i,j)$, while $X_{ij}Y_i$ is not.

The symmetrization produces unwanted terms, all of which have lower order.
More generally, consider a contraction of an order $s+v$ symmetric tensor with a $v+t$ symmetric tensor that is fully symmetrized (i.e., contract $A$ and $B$ to obtain a partially symmetric tensor in the first $s$ and the last $t$ indices, then sum over all permutations of $s+t$ indices into parts of $s$ and $t$ indices to obtain a fully symmetric tensor $C$ of order $s+t$).
When $s=t=v=1$, this is $AB+BA$ with symmetric matrices $A$ and $B$.
For fixed $s,t,v$, let $w=s+t+v$ and use $d=w$ in the coordinate and complement notation. The corresponding trilinear form is
\begin{align}
&\sum_{(U,V,W)\in B_{s,t,v}}\sum_{\mathbf i\in\ivl{0}{n-1}^w}
X_{\mathbf i|_{U^c}}
Y_{\mathbf i|_{V^c}}
Z_{\mathbf i|_{W^c}} \\
B_{s,t,v}&=\{(U,V,W) \mid U\sqcup V \sqcup W = \ivl{0}{w-1}, |U|=t, |V|=s, |W|=v\}. \label{eq:bstv}
\end{align}
For each set of $w$ distinct index values, symmetry leaves ${w\choose s,t,v}$ distinct monomials, one per tripartition of those values. Including repeated-index tuples gives
$\binom{n}{w}{w\choose s,t,v}+O(n^{w-1})=n^w/(s!t!v!)+O(n^{w-1})$
distinct monomials. We obtain a decomposition of rank at most $n^w/w!+O(n^{w-1})$ by first computing the aggregate
\begin{align}
\sum_{\mathbf i\in\ivl{0}{n-1}^w}
\bigg(\sum_{U \subseteq \ivl{0}{w-1}, |U|=t} X_{\mathbf i|_{U^c}}\bigg)
\bigg(\sum_{V \subseteq \ivl{0}{w-1}, |V|=s} Y_{\mathbf i|_{V^c}}\bigg)
\bigg(\sum_{W \subseteq \ivl{0}{w-1}, |W|=v} Z_{\mathbf i|_{W^c}}\bigg).
\end{align}
Its summand is invariant under permutations of $\mathbf i$, so the distinct-index contribution requires only $\binom{n}{w}$ rank-one terms, one for each increasing tuple $i_0<\cdots<i_{w-1}$, with coefficient $w!$. Repeated-index tuples contribute $O(n^{w-1})$ additional terms.
When expanding this trilinear product, terms for which $(U,V,W)$ is not a tripartition of $\ivl{0}{w-1}$ must satisfy at least one of the following two conditions.
\begin{definition}[Replicated Term]
\label{def:rep}
$|U \cap V \cap W| \geq 1$ \\
In this case, some index in $\mathbf{i}$ is omitted in $X$, $Y$, and $Z$, and hence there are at most $n^{w-1}$ distinct monomial terms, each appearing at least $n$ times.
\end{definition}
\begin{definition}[Linearly Reducible Term] $|(U\cap V)\setminus W| \geq 1$ (or similar with $U,V,W$ permuted) \\
\label{def:lr}
In this case, some index in $\mathbf{i}$ appears only in one type of variable ($X$, $Y$, or $Z$). We may sum that tensor of variables over the corresponding index to obtain a trilinear form with $n^{w-1}$ terms (again a tensor contraction with one less index).
\end{definition}
Subtracting these lower-order cross terms gives the stated rank bound.
Further examples and complete equations for computing the cross terms are given in~\cite{solomonik2021fast}.
In the following sections, we review well-known constructions of fast matrix multiplication algorithms, and show that, when written as high-order tensor contractions, they may be similarly understood in terms of replicated and linearly reducible terms.

\section{Background on Fast Matrix Multiplication}
\label{sec:fmm-background}

\subsection{Direct sums of independent matrix--matrix products}
\label{sec:direct-sums}

For bilinear maps with disjoint input and output spaces, the direct sum is the
map that evaluates all of them at once.
Hence the problem of independent matrix multiplications is a direct sum of matrix multiplication tensors.
The elementary upper bounds
\begin{equation}
R(T_1\oplus\cdots\oplus T_p)\leq\sum_{i=1}^p R(T_i),
\qquad
\underline R(T_1\oplus\cdots\oplus T_p)
\leq\sum_{i=1}^p\underline R(T_i)
\end{equation}
come from applying separate algorithms (rank decompositions).  Rank additivity was conjectured by
Strassen, but has been proven false for general tensors \cite{shitov2019directsum}.
Rank additivity of matrix-multiplication remains an open problem. For instance,
$R(\langle2,2,2\rangle^{\oplus2})=14$ already follows from the Alder--Strassen bound~\cite{alder1981associative}: the algebra $\mathbb C^{2\times2}\times\mathbb C^{2\times2}$ has dimension eight and two maximal two-sided ideals, giving the lower bound $2\cdot8-2=14$, while two applications of Strassen's algorithm attain it. See also Rupniewski~\cite{rupniewski2022directsum} for more general direct-sum results.

\subsection{Sch\"onhage's Asymptotic Sum Theorem}

Sch\"onhage's asymptotic sum inequality~\cite{schonhage1981partial}
is the standard route from a border-rank bound for simultaneous matrix
multiplication to a bound on the exponent $\omega$.
Thus trilinear aggregation algorithms for simultaneous matrix multiplication yield bounds on the matrix multiplication exponent.
The asymptotic sum inequality gives much stronger bounds on $\omega$ than using $m$ independent products merely as blocks of
one product of dimension $nm^{1/3}$.  For exact-rank constructions, the inequality also has a direct algorithmic interpretation described next.

Suppose the construction input to the theorem is an exact bilinear construction of rank $r$ that reduces $m$ independent matrix multiplications of dimension $n$ to $r$ scalar products.
At the next recursive level, each scalar product becomes a matrix--matrix product of dimension $n$.
We may compute these $r$ products by applying the same construction $\lceil r/m\rceil$ times, obtaining an exact construction for $m$ independent matrix multiplications of dimension $n^2$ with rank at most $r\lceil r/m\rceil$.
Iterating this recursion gives a concrete sequence of algorithms, with additional recursive levels used as the input dimension increases, and can yield a better effective exponent $\omega$.
Inefficiency from the use of the ceiling may be addressed by first taking a sufficiently high tensor power of the construction.
We now state the asymptotic sum theorem in slightly simplified form.
\begin{theorem}
(Sch\"onhage's Asymptotic Sum Theorem~\cite{schonhage1981partial})
Let $\mathcal{M} = \bigoplus_{i=1}^m \langle n, n, n \rangle$ be a direct sum of $m$ matrix multiplication tensors.
If the border rank is bounded by $\underline{R}(\mathcal{M}) \le r$, then the exponent of matrix multiplication $\omega$ satisfies
\[  mn^\omega \le r. \]
\end{theorem}
\begin{proof}
We provide an algorithmic proof for the simplified case where the direct sum consists of $m$ identical terms $\langle n, n, n \rangle$ and the rank bound $r$ is divisible by $m$ in Appendix~\ref{apx:ast}.
For the general proof involving multinomial coefficients and arbitrary dimensions, we refer the reader to~\cite{schonhage1981partial} or the survey by Bl\"aser~\cite{blaser2013fast}.
\end{proof}

\subsection{The Coppersmith-Winograd Tensor}

We first recall the 
tensor constructions of Coppersmith and Winograd~\cite{coppersmith1990matrix}.
\begin{definition}[Little CW tensor]
For $x,y,z\in \mathbb{C}^{n+1}$, the little CW tensor is defined by the trilinear form,
\begin{equation}
\phi_{\mathrm{CW}}(x,y,z) = \sum_{i=1}^n \bigl(x_iy_iz_0 + x_iy_0z_i + x_0y_iz_i\bigr).
\end{equation}
\end{definition}
For $n\ge2$, the little CW tensor has border rank exactly $n+2$~\cite[Section~1.2]{conner2021:cc}. The upper bound, valid for all $n\ge1$, is shown by
\begin{align}
\phi_{\mathrm{CW}}(x,y,z) = \lim_{\lambda \to 0}\bigg[&\lambda^{-2}
\sum_{i=1}^n (x_0+\lambda x_i)(y_0+\lambda y_i)(z_0 + \lambda z_i) \\
-&\lambda^{-3}(x_0+\lambda^2 \sum_{i=1}^nx_i)
(y_0+\lambda^2 \sum_{j=1}^ny_j)
(z_0+\lambda^2 \sum_{k=1}^nz_k)\\
-&(\lambda^{-2}n-\lambda^{-3})x_0y_0z_0
\bigg].
\end{align}
For $n=1$, the little CW tensor instead has border rank two.

\begin{definition}[Big CW tensor]
For $x,y,z\in \mathbb{C}^{n+2}$, the big CW tensor is defined by the trilinear form,
\begin{equation}
\phi_{\mathrm{BCW}}(x,y,z) = \sum_{i=1}^n \bigl(x_iy_iz_0 + x_iy_0z_i + x_0y_iz_i\bigr)
+ x_{0}y_{0}z_{n+1}
+ x_{0}y_{n+1}z_{0}
+ x_{n+1}y_{0}z_{0}.
\end{equation}

\end{definition}

The big CW tensor has border rank exactly $n+2$: its flattening in any variable space has matrix rank $n+2$, giving the lower bound, and the following degeneration gives the upper bound
\begin{align}
\phi_{\mathrm{BCW}}(x,y,z) = \lim_{\lambda \to 0}\bigg[&\lambda^{-2}
\sum_{i=1}^n (x_0+\lambda x_i)(y_0+\lambda y_i)(z_0 + \lambda z_i) \\
-&\lambda^{-3}(x_0+\lambda^2 \sum_{i=1}^nx_i)
(y_0+\lambda^2 \sum_{j=1}^ny_j)
(z_0+\lambda^2 \sum_{k=1}^nz_k)\\
-&(1/\sigma_{\lambda,n})
(x_0- \sigma_{\lambda,n} x_{n+1})
(y_0- \sigma_{\lambda,n} y_{n+1})
(z_0- \sigma_{\lambda,n} z_{n+1})
 \bigg], \\
 \sigma_{\lambda,n} = &\frac{\lambda^3}{n\lambda-1}.
\end{align}

\section{Simultaneous Tensor Contraction}
\label{sec:simultaneous-contraction}
\subsection{The Coppersmith-Winograd Tensor as Aggregated Tensor Contractions}
\label{sec:cw-contractions}

In this section, we show that powers of the little CW and big CW tensors have natural interpretations in terms of tensor contractions associated with tripartitions.
As a bilinear map, we may interpret the little CW tensor as taking inputs 
$a^{(0)}\in\mathbb{C}$, $a^{(1)}\in\mathbb{C}^n$,
$b^{(0)}\in\mathbb{C}$, $b^{(1)}\in\mathbb{C}^n$, and computing 
\begin{equation}
c^{(0)} = (a^{(1)})^Tb^{(1)},\quad c^{(1)} = a^{(0)}b^{(1)} + a^{(1)}b^{(0)}.
\end{equation}
Following notation we used for tripartitions before, we may label the variables in the trilinear forms as 
$X=\{X^{(\{\})}, X^{(\{0\})}\}$,
$Y=\{Y^{(\{\})}, Y^{(\{0\})}\}$,
$Z=\{Z^{(\{\})}, Z^{(\{0\})}\}$, and write the CW tensor as
\begin{align}
\phi_{\mathrm{CW}}(X,Y,Z) = \sum_{i=0}^{n-1} \bigl(
X^{(\{\})}_iY^{(\{\})}_iZ^{(\{0\})}
+X^{(\{\})}_iY^{(\{0\})}Z^{(\{\})}_i
+X^{(\{0\})}Y^{(\{\})}_iZ^{(\{\})}_i\bigr).\label{eq:phicw}
\end{align}
Here, we label each component of  $X$, $Y$, $Z$ by the (possibly empty) subset of indices that is excluded from the variable.
The index space is $\{0\}$, and hence the only two such subsets are $\{\}$ and $\{0\}$.

Recall our definition of tripartitions $B_{s,t,v}$ in \eqref{eq:bstv}, and define all tripartitions as
\begin{equation}
P_w = \bigcup_{\substack{s,t,v\in\mathbb{Z}_{\geq0}\\s+t+v=w}} B_{s,t,v}.
\end{equation}
For higher tensor powers, let $d\ge1$.
We have the following relation (recall Definition~\ref{def:ief}).
\begin{lemma}
\begin{align}
\phi^{\otimes d}_{\mathrm{CW}}(X,Y,Z) = \sum_{(U,V,W)\in P_d}\ \ 
\sum_{\mathbf{i}\in\ivl{0}{n-1}^{d}}
X^{(U)}_{\mathbf i|_{U^c}}
Y^{(V)}_{\mathbf i|_{V^c}}
Z^{(W)}_{\mathbf i|_{W^c}}.\label{eq:cwpower}
\end{align}
\end{lemma}
\begin{proof}
The claim follows by induction on $d$.
$P_d$ may be obtained by adding an item to one of the three parts of each tripartition in $P_{d-1}$.
These are exactly the terms we get in $\phi^{\otimes d}_{\mathrm{CW}}$, since $\phi^{\otimes d}_{\mathrm{CW}}=\phi^{\otimes (d-1)}_{\mathrm{CW}} \otimes \phi_{\mathrm{CW}}$, $\phi^{\otimes (d-1)}_{\mathrm{CW}}$ contains the terms in $P_{d-1}$ by induction, and \eqref{eq:phicw} maps $\phi_{\mathrm{CW}}$ to an item in $U$, $V$, or $W$.
\end{proof}
Coppersmith and Winograd consider a restriction of this tensor by zeroing out blocks of variables $X^{(U)}$, $Y^{(V)}$, $Z^{(W)}$ for some choices of $U$, of $V$, and of $W$.
Ultimately, they obtain a set of independent matrix--matrix products with border rank at most $(n+2)^d$.
This bound follows by tensoring the $(n+2)$-term little CW degeneration $d$ times; zeroing variable blocks cannot increase the resulting number of rank-one terms.
Appendix~\ref{apx:N1br} gives the smaller bound $n^3+6n^2$ for the balanced cube, whose six matrix products share variable blocks.

The big CW tensor similarly gives contractions indexed by incomplete tripartitions. In each tensor-power term, write $F=G\sqcup H$, where the factors on $G$ come from the little CW tensor and those on $H$ from the three additional scalar terms. The two tripartitions $U\sqcup V\sqcup W=G$ and $\bar U\sqcup\bar V\sqcup\bar W=H$ specify these choices.
As a bilinear map, the big CW tensor may be interpreted as taking inputs 
$a^{(0)}, \hat{a}^{(0)}, b^{(0)}, \hat{b}^{(0)} \in \mathbb{C}$ and $a^{(1)}, b^{(1)} \in\mathbb{C}^n$, and computing
\begin{equation}
c^{(0)} = 
a^{(0)}\hat{b}^{(0)} + \hat{a}^{(0)}b^{(0)} + (a^{(1)})^Tb^{(1)}, \quad \hat{c}^{(0)}=a^{(0)}b^{(0)}, \quad c^{(1)} = a^{(0)}b^{(1)} + a^{(1)}b^{(0)}.
\end{equation}
We write $X^{(S,\bar S)}$ for a variable block of order $d-|S|$ that omits coordinates $S$. Within $S$, the subset $\bar S$ selects the additional scalar variable (index $n+1$), and $S\setminus\bar S$ selects the original scalar variable (index $0$). In a tensor-power term, the $X$-block thus has $S=U\cup H$ and $\bar S=\bar U\subseteq H$.
For the big CW tensor, we then have 
$X=\{X^{(\{\},\{\})}, X^{(\{0\},\{\})}, X^{(\{0\},\{0\})}\}$,
$Y=\{Y^{(\{\},\{\})}, Y^{(\{0\},\{\})}, Y^{(\{0\},\{0\})}\}$, and
$Z=\{Z^{(\{\},\{\})}, Z^{(\{0\},\{\})}, Z^{(\{0\},\{0\})}\}$,
with
\begin{align}
\phi_{\mathrm{BCW}}(X,Y,Z) =& 
\phi_{\mathrm{CW}}(
\{X^{(\{\},\{\})}, X^{(\{0\},\{\})}\},
\{Y^{(\{\},\{\})}, Y^{(\{0\},\{\})}\},
\{Z^{(\{\},\{\})}, Z^{(\{0\},\{\})}\}) \\
&+
X^{(\{0\},\{\})}Y^{(\{0\},\{\})}Z^{(\{0\},\{0\})}  \\
&+
X^{(\{0\},\{\})}Y^{(\{0\},\{0\})}Z^{(\{0\},\{\})} \\
&+
X^{(\{0\},\{0\})}Y^{(\{0\},\{\})}Z^{(\{0\},\{\})} .
\end{align}
\begin{lemma}\label{lem:bcw-power}
For $X=\{X^{(S,\bar S)} \vert \bar S\subseteq S\subseteq F\}$, with $X^{(S,\bar S)}\in(\mathbb{C}^{n})^{\otimes(d-|S|)}$
and $Y$, $Z$ defined similarly,
\begin{equation}
\phi_{\mathrm{BCW}}^{\otimes d}(X,Y,Z) = 
\sum_{G \sqcup H = F} \ \sum_{U \sqcup V \sqcup W = G} \ \sum_{\bar{U} \sqcup \bar{V} \sqcup \bar{W} = H} \ 
\sum_{\mathbf{i}\in\ivl{0}{n-1}^{G}}
X^{(U\cup H,\bar{U})}_{\mathbf i|_{G\setminus U}}
Y^{(V\cup H,\bar{V})}_{\mathbf i|_{G\setminus V}}
Z^{(W\cup H,\bar{W})}_{\mathbf i|_{G\setminus W}},
\end{equation}
where $\sqcup$ is disjoint union and we use the shorthand notation,
\begin{equation}
\sum_{A\sqcup B \sqcup C = D} = \sum_{\substack{(A,B,C):\ A \sqcup B \sqcup C = D}}. 
\end{equation}
\end{lemma}
\begin{proof}
Each tensor contraction term in $\phi_{\mathrm{BCW}}^{\otimes d}$ is a tensor product of $|G|\leq d$ terms from $\phi_{\mathrm{CW}}$ (indexed by $G$) and $|H|$ scalar-valued terms introduced in $\phi_{\mathrm{BCW}}$ (indexed by $H$), such as $X^{(\{0\},\{\})}Y^{(\{0\},\{\})}Z^{(\{0\},\{0\})}$.
The latter $|H|$ scalar-valued terms correspond to a tripartition of the indices omitted from all three tensor variables in the term, with the block of variables indexed accordingly.
Only coordinates in $G$ are summed, since those in $H$ are omitted from all three blocks.
\end{proof}

A selection of these contractions is conflict-free if no two share a variable block in any of the three spaces and zeroing all unselected blocks leaves exactly the selected contractions. Indexing the selected contractions by $t$, the restriction gives independent matrix products $\langle n^{|V_t|},n^{|W_t|},n^{|U_t|}\rangle$ and has border rank at most $(n+2)^d$.

\subsection{Aggregation of Tensor Contractions}
\label{sec:aggregation}

We approach the simultaneous matrix products selected from powers of the little CW tensor in a bottom-up manner, by considering trilinear aggregation of tensor contractions.

We define a trilinear aggregation of tensor contractions by a subset of balanced tripartitions $\mathcal{S} \subseteq B_{N,N,N}$.
We are then interested in a factorization of the trilinear form corresponding to $|\mathcal{S}|$ simultaneous matrix--matrix products of matrices of dimension $n^N$, all viewed as different contractions of order $2N$ tensors.
With $d=3N$,
\begin{align}
\phi_\mathcal{S}(X,Y,Z) &=
 \sum_{(U,V,W)\in \mathcal{S}}
\sum_{\mathbf{i}\in\ivl{0}{n-1}^{d}}
X^{(U)}_{\mathbf i|_{U^c}}
Y^{(V)}_{\mathbf i|_{V^c}}
Z^{(W)}_{\mathbf i|_{W^c}}.\label{eq:ctrsim}
\end{align}
For $j\in\{1,2,3\}$, let $\pi_j(\mathcal S)$ denote the projection of $\mathcal S$ onto its $j$th coordinate.
We seek to obtain $\phi_{\mathcal S}$ from the aggregated trilinear form,
\begin{equation}
\sum_{\mathbf{i}\in\ivl{0}{n-1}^{d}}
\bigg(\sum_{U\in\pi_1(\mathcal S)}X^{(U)}_{\mathbf i|_{U^c}}\bigg)
\bigg(\sum_{V\in\pi_2(\mathcal S)}
Y^{(V)}_{\mathbf i|_{V^c}}\bigg)
\bigg(\sum_{W\in\pi_3(\mathcal S)}
Z^{(W)}_{\mathbf i|_{W^c}}\bigg).\label{eq:ctragg}
\end{equation}

Index the tripartitions in $\mathcal S$ by $t$, writing $(U_t,V_t,W_t)$.

A mixed choice $(U_i,V_j,W_k)$ selects one factor from each of the three
sums in \eqref{eq:ctragg}.  For each coordinate $r\in F$, its local multiplicity
$\eta_r(i,j,k)\in\{0,1,2,3\}$ is the number of sets among $U_i,V_j,W_k$
that contain $r$.
Since $U_i,V_j,W_k$ specify omitted coordinates, the index in coordinate
$r$ occurs in exactly $3-\eta_r$ of the selected tensors.  When it occurs
in two tensors, summation over that index contracts the corresponding
pair of fibers, as in the trilinear form for matrix multiplication.
When it occurs in all three tensors, the summation is over their
entrywise (Hadamard) product.  An index occurring in only one tensor
can be summed within that tensor, while an index absent from all three
simply replicates the same value.  We use the following names for these
four cases:

\[
\begin{array}{c|c|c}
\eta_r & \text{number of tensors containing the index} & \text{index type}\\ \hline
0 & 3 & \text{Hadamard index}\\
1 & 2 & \text{pairwise contraction index}\\
2 & 1 & \text{linearly reducible index}\\
3 & 0 & \text{replicated index}
\end{array}
\]

For a balanced mixed choice, $\sum_{r\in F}\eta_r(i,j,k)=|U_i|+|V_j|+|W_k|=d$.
Thus, in the absence of linearly reducible and replicated coordinates, every $\eta_r$ is one and the selected sets form a tripartition.

\begin{definition}[Compatible combination and conflict-free family]
\label{def:compat}
Let $\mathcal S$ be a family of tripartitions of the same coordinate set $F$.
A mixed choice $(U_i,V_j,W_k)$ is \emph{compatible} if
\begin{equation}
U_i\sqcup V_j\sqcup W_k=F.
\end{equation}
The family $\mathcal S$ is \emph{conflict-free} if every compatible choice is
diagonal, i.e., compatibility implies $i=j=k$.
\end{definition}

\begin{figure}
\centering
\includegraphics[width=\textwidth]{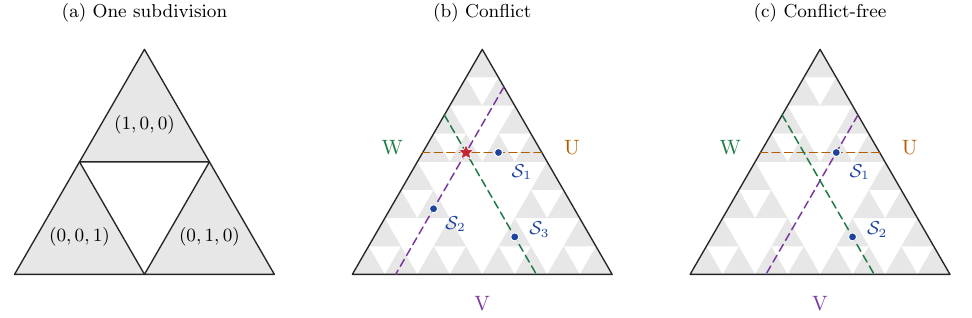}
\caption{Sierpiński’s triangle, obtained by repeatedly removing central subtriangles, visualizes conflicts between tripartitions. (a) One subdivision, labelled by membership in \(U,V,W\). Blue points mark selected tripartitions; mixed projections create an additional contraction in (b) (star), but not in (c).}
\label{fig:sierp}
\end{figure}

Figure~\ref{fig:sierp} visualizes tripartitions and their conflicts using Sierpiński's triangle.
In panel (b), mixing parts from the selected tripartitions produces the
additional starred tripartition; panel (c) illustrates a conflict-free selection.

\begin{problem}[Cross-term elimination]
\label{prob:cte}
For $N\ge1$, determine
\[
M_N=\max\left\{|\mathcal S|:\mathcal S\subseteq B_{N,N,N}
\text{ is conflict-free}\right\}.
\]
\end{problem}

\begin{lemma}
\label{lem:conflict-injective}
If $\mathcal S$ is conflict-free, then each of the three projections
$t\mapsto U_t$, $t\mapsto V_t$, and $t\mapsto W_t$ is injective.
\end{lemma}
\begin{proof}
If $U_i=U_j$ for $i\ne j$, then $(U_i,V_j,W_j)$ is a non-diagonal compatible
choice.  The arguments for $V$ and $W$ are identical.
\end{proof}

Consequently, a balanced conflict-free family has only replicated or linearly reducible cross terms.  The following subsections develop embeddings that annihilate these
terms by expanding the local dimension from $n$ to $n+2$.
Section~\ref{sec:compat-color} later
shows how replicated terms can instead be cancelled by lower-order aggregates
while using only dimension $n+1$.

For a conflict-free family $\mathcal S\subseteq B_{N,N,N}$, the parts $U_t$ are distinct, so $|\mathcal S|\le\binom{3N}{N}$.
Together with the Salem--Spencer construction used by Coppersmith and Winograd~\cite{salemspencer1942,coppersmith1990matrix}, this upper bound gives $M_N=\binom{3N}{N}^{1-o(1)}$ as $N\to\infty$.

\subsection{Centering}
\label{sec:centering}

We introduce a modewise zero-sum lifting and call it \emph{centering}.
It automatically ensures that linearly reducible terms vanish.

In particular, we embed tensors $X^{(U)}, Y^{(V)}, Z^{(W)} \in (\mathbb{C}^{n})^{\otimes 2N}$ into tensors $\bar{X}^{(U)}, \bar{Y}^{(V)}, \bar{Z}^{(W)}\in(\mathbb{C}^{n+1})^{\otimes 2N}$ such that the sum, equivalently the average, along every fiber is zero.
We achieve this embedding by a tensor product of simple embeddings from $\mathbb{C}^{n}$ to $\mathbb{C}^{n+1}$.
Let $1_m=(1,\ldots,1)^T\in\mathbb R^m$ denote the all-ones vector.

\begin{definition}[Centering embedding]
A centering embedding is a real matrix
\[
Q\in\mathbb R^{(n+1)\times n}
\]
whose columns form an orthonormal basis for the hyperplane
\begin{equation}
1_{n+1}^{\perp}=\{x\in\mathbb R^{n+1}:1_{n+1}^Tx=0\}.
\end{equation}
Equivalently, the matrix
\[
\begin{bmatrix}
\frac{1}{\sqrt{n+1}}1_{n+1} & Q
\end{bmatrix}
\in\mathbb R^{(n+1)\times(n+1)}
\]
is orthogonal.

Thus
\begin{equation}
1_{n+1}^TQ=0,\qquad Q^TQ=I_n.
\end{equation}
We use the same matrix \(Q\) as a complex-linear map
\(\mathbb C^n\to\mathbb C^{n+1}\).  Throughout, the tensor contractions are bilinear algebraic contractions, so no complex conjugation is used.
\end{definition}

A simple choice for a centering embedding is to choose $Q$ based on a transposed Helmert matrix~\cite{lancaster1965helmert},
\begin{equation}
\forall k\in \ivl 0n, l\in \ivl 0{n-1}, \quad q_{kl} = \begin{cases}
\frac{1}{\sqrt{(l+1)(l+2)}} & k\leq l,\\
\displaystyle -\frac{l+1}{\sqrt{(l+1)(l+2)}} & k=l+1,\\
0 & k> l+1.
\end{cases}
\end{equation}
The Helmert matrix is computationally convenient because, apart from a subdiagonal part, each column is a scalar multiple of the corresponding column of the transposed prefix sum matrix, and hence it can be applied to a vector with $O(n)$ operations.

We can recover tensor contractions after applying the centering embedding.
For example, for any matrices \(A,B\in \mathbb C^{n\times n}\), embed them as
\begin{equation}
\bar A=QAQ^T,\qquad \bar B=QBQ^T.
\end{equation}
Then
\begin{equation}
\bar A\bar B=QABQ^T,
\end{equation}
so, after computing \(\bar C=\bar A\bar B\), we recover
\begin{equation}
C=Q^T\bar C Q.
\end{equation}
As an example of how the embedding may be useful, recall the symmetric matrix--vector product algorithm \(c=Ab\) from Section~\ref{sec:mot}.  We may obtain an algorithm with \(\binom{n+2}{2}\) products by computing
\begin{equation}
\bar A=QAQ^T,\qquad \bar b=Qb,
\end{equation}
then
\begin{equation}
\bar c_i=\sum_j \bar A_{ij}(\bar b_i+\bar b_j),
\end{equation}
and finally recovering
\begin{equation}
c=Q^T\bar c.
\end{equation}
Since $\sum_j\bar A_{ij}=0$, the unwanted term $\bar b_i\sum_j\bar A_{ij}$ vanishes without an explicit correction. This example illustrates how the embedding automates cross-term elimination. As we also see later for simultaneous tensor contraction, computing the cross-term corrections explicitly can instead give a smaller multiplication count.

To apply centering to higher order tensors, we apply the linear embedding given by $Q$ along each mode of the tensor.
The resulting tensor has zero sum along every fiber.

\begin{definition}[Centering map and its tensor products]
\label{def:tensor-centering}
For a centering matrix $Q\in\mathbb R^{(n+1)\times n}$, define
\begin{equation}
\mu:\mathbb C^n\to\mathbb C^{n+1},\qquad \mu(x)=Qx.
\end{equation}
For a tensor whose modes are indexed by a finite set $S$, centering in
all modes is the map
\begin{equation}
\mu^{\otimes |S|}=Q^{\otimes |S|}:
(\mathbb C^n)^{\otimes |S|}\to(\mathbb C^{n+1})^{\otimes |S|}.
\end{equation}
The input and output modes are indexed by $S$, in increasing order. In index notation,
\begin{equation}
(\mu^{\otimes |S|}(T))_{\mathbf j}
=\sum_{\mathbf a\in\ivl{0}{n-1}^{S}}
T_{\mathbf a}\prod_{r\in S}Q_{j_r a_r},
\qquad \mathbf j\in\ivl{0}{n}^{S}.
\end{equation}
\end{definition}
The tensor centering embedding preserves tensor contraction trilinear forms.
\begin{lemma}\label{lem:center_tri}
For any $s,t,v\in\mathbb{Z}_{\geq0}$, let
\[
X\in(\mathbb{C}^{n})^{\otimes(s+v)},\qquad
Y\in(\mathbb{C}^{n})^{\otimes(v+t)},\qquad
Z\in(\mathbb{C}^{n})^{\otimes(s+t)}.
\]
Let $\bar X=\mu^{\otimes(s+v)}(X)$,
$\bar Y=\mu^{\otimes(v+t)}(Y)$, and
$\bar Z=\mu^{\otimes(s+t)}(Z)$. Then
\[
\mathcal{T}^{s,t,v}_n(X,Y,Z)
=
\mathcal{T}^{s,t,v}_{n+1}(\bar{X},\bar{Y},\bar{Z}).
\]
\end{lemma}
\begin{proof}
By Definition~\ref{def:tensor-centering},
\begin{align}
\bar{X}_{\mathbf{l}} = \sum_{\mathbf{k} \in \ivl{0}{n-1}^{s+v}} X_{\mathbf{k}} \prod_{m=0}^{s+v-1} Q_{l_m k_m},
\end{align}
and similar for $\bar{Y}$, $\bar{Z}$, so
\begin{align}
\mathcal{T}^{s,t,v}_{n+1}(\bar{X},\bar{Y},\bar{Z}) &= 
\sum_{\boldsymbol{\alpha} \boldsymbol{\beta} \mathbf{j} \mathbf{k}} \prod_{m=0}^{s-1}\underbrace{\bigg(\sum_{\ell=0}^{n} Q_{\ell\alpha_m}Q_{\ell\beta_m}\bigg)}_{\delta_{\alpha_m\beta_m}}X_{\boldsymbol{\alpha} \mathbf{k}} Y_{\mathbf{k} \mathbf{j}} Z_{\boldsymbol{\beta} \mathbf{j}},
\end{align}
where the analogous $Q^TQ=I_n$ identities have already been applied over the $j$ and $k$ indices. Applying the remaining identities over the $\alpha$ and $\beta$ indices gives the stated equality.
\end{proof}

\begin{lemma}
If a term in a trilinear form satisfies the linear reducibility condition (Definition~\ref{def:lr}), it vanishes when evaluated on centered tensors.
Specifically, if an index $f$ appears in exactly one variable, the term is zero for centered $X$, $Y$, and $Z$.
\end{lemma}
\begin{proof}
Assume the index $f$ appears only in the tensor $X$.
We can factor the summation over $f$ from the rest of the trilinear form,
\begin{align}
\sum_{\text{all indices}} (\dots) &= \sum_{\text{indices} \setminus \{f\}} \left( \sum_{f} X_{\dots f \dots} \right) Y_{\dots} Z_{\dots}.
\end{align}
The inner sum is the marginal sum of $X$ along mode $f$.
Since $X$ is a centered tensor, this marginal sum is zero.
Therefore, the entire term vanishes.
\end{proof}

\subsection{Padding and Weighting}
\label{sec:padding-weighting}

We now handle the replicated cross terms by an embedding into $\mathbb{C}^{n+2}$.
Beyond centering, this embedding simply pads the tensor with zeros, which are used in the trilinear form to annihilate the replicated terms via a suitable choice of weighting in the trilinear form.
\begin{definition}[Padding]
\label{def:padding}
Define the padding map $\nu:\mathbb C^{n+1}\to\mathbb C^{n+2}$ by
\begin{align}
\nu(x)&=\begin{bmatrix}x\\0\end{bmatrix}.
\end{align}
Thus $\nu\circ\mu$ is the padded centering embedding.
\end{definition}
The padded centering embedding becomes interesting when combined with a weighting.
\begin{definition}[Weight vector]
\label{def:weight-vector}
Define $w\in\mathbb{C}^{n+2}$ by
\begin{equation}
w_i = \begin{cases} 1 & \text{if } 0 \le i \le n \\ -(n+1)^{1/3} & \text{if } i = n+1 \end{cases}. 
\end{equation}
\end{definition}
The efficacy of the weighting follows from the following claim.
\begin{claim}\label{cla:emb}
The weight vector $w$ and padded centering embedding $\nu\circ\mu$ satisfy the following properties.
For any $x, y \in \mathbb{C}^n$,
\begin{enumerate}
    \item metric preservation: $\sum_{i} w_i \nu(\mu(x))_i \nu(\mu(y))_i = x^Ty$,
    \item linear reducibility annihilation: $\sum_{i} w_i^2 \nu(\mu(x))_i = 0$,
    \item replication annihilation: $\sum_{i} w_i^3 = 0$.
\end{enumerate}
\end{claim}
\begin{proof}
Since $w_i=1$ for $0\le i\le n$ and $\nu(\mu(x))_{n+1}=0$,
\begin{equation}
\sum_{i=0}^{n+1} w_i\nu(\mu(x))_i\nu(\mu(y))_i
=
\sum_{i=0}^{n}(Qx)_i(Qy)_i
=x^TQ^TQy
=x^Ty.
\end{equation}
Since $Q$ maps into $1_{n+1}^{\perp}$,
\begin{equation}
\sum_{i=0}^{n+1}w_i^2\nu(\mu(x))_i
=
\sum_{i=0}^{n}(Qx)_i
=0.
\end{equation}
Finally,
\begin{equation}
\sum_{i=0}^{n+1}w_i^3
=
\sum_{i=0}^{n}1-(n+1)
=0.
\end{equation}
\end{proof}
Hence, to eliminate both linearly reducible and replicated terms, we aggregate weighted, padded, and centered tensors.

\begin{definition}[Weighting and the combined embedding]
\label{def:weight}
Let $w$ be the weight vector from Definition~\ref{def:weight-vector}.
For finite coordinate sets $S\subseteq F$, define
\[
\Omega_{S,F}:(\mathbb C^{n+2})^{\otimes |F\setminus S|}
\longrightarrow(\mathbb C^{n+2})^{\otimes |F|}
\]
by
\begin{equation}
\Omega_{S,F}(T)_{\mathbf i}
=T_{\mathbf i|_{F\setminus S}}\prod_{r\in S}w_{i_r},
\qquad \mathbf i\in\ivl{0}{n+1}^{F}.
\end{equation}
The input modes are indexed by $F\setminus S$, and the output modes by
$F$, both in increasing order. The map inserts the coordinates in $S$
with the corresponding weights. Define the combined embedding by
\begin{equation}
\Phi_{S,F}=\Omega_{S,F}\circ(\nu\circ\mu)^{\otimes |F\setminus S|}.
\end{equation}
It centers and pads the retained modes, then inserts the weighted coordinates.
For its transpose, write $k=|F\setminus S|$ and
$F\setminus S=\{r_0<\cdots<r_{k-1}\}$. With respect to the algebraic
pairing,
\begin{equation}
\bigl(\Phi_{S,F}^{\star}(A)\bigr)_{\boldsymbol\alpha}
=\sum_{\mathbf i\in\ivl{0}{n+1}^{F}}
 A_{\mathbf i}\left(\prod_{r\in S}w_{i_r}\right)
 \left(\prod_{j=0}^{k-1}\nu(\mu(e_{\alpha_j}))_{i_{r_j}}\right),
\qquad \boldsymbol\alpha\in\ivl{0}{n-1}^k.
\label{eq:Psi-star}
\end{equation}
Here $e_0,\ldots,e_{n-1}$ are the standard basis vectors of $\mathbb C^n$.
\end{definition}

For a tripartition $(U,V,W)$ of $F$, metric preservation in Claim~\ref{cla:emb}, applied in each coordinate, gives
\begin{equation}
\sum_{\mathbf i\in\ivl{0}{n-1}^{F}}
X^{(U)}_{\mathbf i|_{U^c}}Y^{(V)}_{\mathbf i|_{V^c}}Z^{(W)}_{\mathbf i|_{W^c}}
=\sum_{\mathbf i\in\ivl{0}{n+1}^{F}}
\Phi_{U,F}(X^{(U)})_{\mathbf i}
\Phi_{V,F}(Y^{(V)})_{\mathbf i}
\Phi_{W,F}(Z^{(W)})_{\mathbf i},
\end{equation}
a decomposition of the contraction into $(n+2)^d$ rank-one terms.
The right-hand side is also defined for sets that do not form a tripartition.
If $|U|+|V|+|W|=d$, such a choice omits some coordinate in at least two blocks, so the second or third identity of Claim~\ref{cla:emb} makes the sum zero.
Consequently, for a balanced conflict-free family, aggregating the embedded blocks preserves the selected products and eliminates all mixed products.

\subsection{Full Bilinear Algorithm for Simultaneous Contraction}

We now combine the embeddings introduced in the preceding subsections into a full algorithm for simultaneous matrix multiplication.
The algorithm is relatively simple and exact, and hence may be of practical interest, though the corrections introduced in Section~\ref{sec:corrections} give a strictly smaller rank bound.
Throughout this subsection, let
\begin{equation}
d=3N.
\end{equation}
\Needspace{11\baselineskip}
Exact-rank bounds for these balanced families of compatible tripartitions are already known.
Flavi, Jelisiejew, and Micha{\l}ek~\cite[Corollary~6.11 and Proposition~6.13]{flavi2025symmetric}
bound the rank of the full balanced little-CW restriction, including its shared variable blocks.
Their result implies that every conflict-free family $\mathcal S\subseteq B_{N,N,N}$ satisfies
\begin{equation}
R\!\left(\langle n^N,n^N,n^N\rangle^{\oplus|\mathcal S|}\right)
\le (n+2)^{3N}
-\binom{3N}{2N+1}(n+1)^{N-1}.
\label{eq:prior-balanced-exact}
\end{equation}
The algorithm below gives an explicit aggregation formula with the simpler known budget $(n+2)^{3N}$.

Recall the combined embedding $\Phi_{S,F}$ and its transpose from Definition~\ref{def:weight}.
Here $\langle A,B\rangle=\sum_{\mathbf i}A_{\mathbf i}B_{\mathbf i}$, without complex conjugation. We have
\begin{equation}
\bigl\langle \Phi_{S,F}^{\star}(T),Z\bigr\rangle
=\bigl\langle T,\Phi_{S,F}(Z)\bigr\rangle
\qquad \forall Z\in(\mathbb C^n)^{\otimes |F\setminus S|}.
\label{eq:Psi-star-transpose}
\end{equation}

Algorithm~\ref{alg:Psi-bilinear-simultaneous-contractions} first embeds and sums the $A$ input blocks, and does the same for the $B$ input blocks. It multiplies the two aggregates entrywise, then applies $\Phi_{W_s,F}^{\star}$ to recover each output $C^{(s)}$. Thus by one aggregate evaluation, the algorithm computes $m$ independent contractions.

\begin{algorithm}[t]
\caption{Weighted padded bilinear algorithm for simultaneous contractions}
\label{alg:Psi-bilinear-simultaneous-contractions}
\begin{algorithmic}[1]
\Require
\[
\mathcal S=\{(U_s,V_s,W_s)\}_{s=1}^{m}\subseteq B_{N,N,N},
\qquad
U_i\sqcup V_j\sqcup W_\ell=F\Longrightarrow i=j=\ell.
\]
\Require
\[
A^{(s)}\in(\mathbb C^n)^{\otimes 2N}
\text{ indexed by }F\setminus U_s,
\qquad
B^{(s)}\in(\mathbb C^n)^{\otimes 2N}
\text{ indexed by }F\setminus V_s.
\]
\Ensure
\[
C^{(s)}\in(\mathbb C^n)^{\otimes 2N}
\text{ indexed by }F\setminus W_s,
\]
with
\[
C^{(s)}_{\boldsymbol{\alpha}}
=
\sum_{\substack{\mathbf{i}\in\ivl{0}{n-1}^{d}\\
\mathbf i|_{W_{s}^c}=\boldsymbol{\alpha}}}
A^{(s)}_{\mathbf i|_{U_{s}^c}}
B^{(s)}_{\mathbf i|_{V_{s}^c}}.
\]

\ForAll{\(\mathbf{k}\in\ivl{0}{n+1}^{d}\)}
    \State
    \[
    T_{\mathbf k}\gets
    \left(\sum_{s=1}^{m}\Phi_{U_s,F}(A^{(s)})_{\mathbf k}\right)
    \left(\sum_{s=1}^{m}\Phi_{V_s,F}(B^{(s)})_{\mathbf k}\right).
    \]
\EndFor

\For{\(s=1,\ldots,m\)}
    \State
    \[
    C^{(s)}
    \gets
    \Phi_{W_s,F}^{\star}(T)
    \]
\EndFor

\State \Return \(\{C^{(s)}\}_{s=1}^{m}\)
\end{algorithmic}
\end{algorithm}

\begin{theorem}[Correctness of Algorithm~\ref{alg:Psi-bilinear-simultaneous-contractions}]
\label{thm:Psi-bilinear-correctness}
If
\begin{equation}
\mathcal S=\{(U_s,V_s,W_s)\}_{s=1}^{m}\subseteq B_{N,N,N}
\end{equation}
satisfies
\begin{equation}
U_i\sqcup V_j\sqcup W_\ell=F\Longrightarrow i=j=\ell,
\end{equation}
then Algorithm~\ref{alg:Psi-bilinear-simultaneous-contractions} computes the
\(m\) simultaneous contractions
\begin{equation}
C^{(s)}_{\boldsymbol{\alpha}}
=
\sum_{\substack{\mathbf{i}\in\ivl{0}{n-1}^{d}\\
\mathbf i|_{W_{s}^c}=\boldsymbol{\alpha}}}
A^{(s)}_{\mathbf i|_{U_{s}^c}}
B^{(s)}_{\mathbf i|_{V_{s}^c}}.
\end{equation}
In particular, this gives a bilinear algorithm of rank at most
\begin{equation}
(n+2)^d=(n+2)^{3N}.
\end{equation}
\end{theorem}

\begin{proof}
Let
\[
Z^{(s)}\in(\mathbb C^n)^{\otimes 2N}
\]
be indexed by \(F\setminus W_s\).  By the definition of
the transpose of $\Phi_{S,F}$,
\begin{align}
\sum_{s=1}^{m}\bigl\langle C^{(s)},Z^{(s)}\bigr\rangle
&=\sum_{\mathbf k\in\ivl{0}{n+1}^{F}}T_{\mathbf k}
  \sum_{s=1}^{m}\Phi_{W_s,F}(Z^{(s)})_{\mathbf k}\notag\\
&=\sum_{\mathbf k\in\ivl{0}{n+1}^{F}}
 \left(\sum_{i=1}^{m}\Phi_{U_i,F}(A^{(i)})_{\mathbf k}\right)
 \left(\sum_{j=1}^{m}\Phi_{V_j,F}(B^{(j)})_{\mathbf k}\right)
 \left(\sum_{\ell=1}^{m}\Phi_{W_\ell,F}(Z^{(\ell)})_{\mathbf k}\right).
\label{eq:Psi-bilinear-trilinear-reduction}
\end{align}
Thus multiplying each output entry $C^{(s)}_{\boldsymbol\alpha}$ by an independent variable $Z^{(s)}_{\boldsymbol\alpha}$ and summing over $s$ and $\boldsymbol\alpha$ gives exactly the weighted padded aggregated trilinear form.

Expanding \eqref{eq:Psi-bilinear-trilinear-reduction} gives one term for
each triple \((i,j,\ell)\).  If
\begin{equation}
U_i\sqcup V_j\sqcup W_\ell\neq F,
\end{equation}
then, since all three sets have cardinality \(N\) and \(d=3N\), the
corresponding term is either linearly reducible or replicated in the sense of
Definitions~\ref{def:lr} and~\ref{def:rep}.  By the filtering identities of
Claim~\ref{cla:emb}, namely
\begin{equation}
\sum_{q=0}^{n+1}w_q^2\nu(\mu(x))_q=0,
\qquad
\sum_{q=0}^{n+1}w_q^3=0,
\end{equation}
such terms vanish after centering, padding, and weighting.

The only remaining terms satisfy
\begin{equation}
U_i\sqcup V_j\sqcup W_\ell=F.
\end{equation}
By conflict-freeness, this forces \(i=j=\ell\).  For these diagonal terms, every
coordinate is retained by exactly two of the three tensors and weighted by the
third, so Claim~\ref{cla:emb} gives, in each coordinate,
\begin{equation}
\sum_{q=0}^{n+1}w_q\nu(\mu(e_a))_q\nu(\mu(e_b))_q
=
\delta_{ab},
\end{equation}
where $e_a,e_b$ are the standard basis vectors of $\mathbb C^n$.

Applying this identity independently in all \(d\) coordinates gives
\begin{equation}
\sum_{s=1}^{m}\bigl\langle C^{(s)},Z^{(s)}\bigr\rangle
=
\sum_{s=1}^{m}
\sum_{\mathbf{i}\in\ivl{0}{n-1}^{d}}
A^{(s)}_{\mathbf i|_{U_{s}^c}}
B^{(s)}_{\mathbf i|_{V_{s}^c}}
Z^{(s)}_{\mathbf i|_{W_{s}^c}}.
\end{equation}
Regrouping by
\begin{equation}
\boldsymbol{\alpha}=\mathbf i|_{W_{s}^c}
\end{equation}
yields
\begin{equation}
\sum_{s=1}^{m}\bigl\langle C^{(s)},Z^{(s)}\bigr\rangle
=
\sum_{s=1}^{m}
\sum_{\boldsymbol{\alpha}\in\ivl{0}{n-1}^{2N}}
\left(
\sum_{\substack{\mathbf{i}\in\ivl{0}{n-1}^{d}\\
\mathbf i|_{W_{s}^c}=\boldsymbol{\alpha}}}
A^{(s)}_{\mathbf i|_{U_{s}^c}}
B^{(s)}_{\mathbf i|_{V_{s}^c}}
\right)
Z^{(s)}_{\boldsymbol{\alpha}}.
\end{equation}
Since the \(Z^{(s)}\) are arbitrary, the displayed coefficient is precisely
the output tensor \(C^{(s)}\).

Finally, the algorithm performs one scalar multiplication for each
\[
\mathbf{k}\in\ivl{0}{n+1}^{d},
\]
and all other operations are linear.  Hence the bilinear rank is at most
\begin{equation}
|\ivl{0}{n+1}^{d}|=(n+2)^d=(n+2)^{3N}.
\end{equation}
\end{proof}

\subsection{Centered corrections and exact rank}
\label{sec:compat-color}
\label{sec:corrections}
\label{sec:compat}

For a fixed-shape conflict-free family, the weighted embedding gives an exact
algorithm with rank at most $(n+2)^d$.
We now use the smaller centered dimension $n+1$ and cancel replicated cross
terms by lower-order aggregates.

Let
\begin{equation}
\mathcal S=\{\mathcal S_1,\ldots,\mathcal S_m\},\qquad
\mathcal S_t=(U_t,V_t,W_t),\qquad
U_t\sqcup V_t\sqcup W_t=F,
\end{equation}
where $t=1,\ldots,m$ and the part sizes may depend on $t$.

Assign independent input blocks $X^{(t)},Y^{(t)}$ and output variables $Z^{(t)}$ to each $\mathcal S_t$. In particular, for $t\ne t'$, $X^{(t)}$ and $X^{(t')}$ are distinct blocks even when $U_t=U_{t'}$, and similarly for $Y$ and $Z$. The resulting target direct sum is
\begin{align}
\phi_{\mathcal S}(X,Y,Z)
 =\sum_{t=1}^{m}\ \sum_{\mathbf i\in\ivl{0}{n-1}^{F}}
 X^{(t)}_{\mathbf i|_{U_{t}^c}}
 Y^{(t)}_{\mathbf i|_{V_{t}^c}}
 Z^{(t)}_{\mathbf i|_{W_{t}^c}} .
\label{eq:phi-S-labeled}
\end{align}
The summand indexed by $t$ is the matrix multiplication tensor
\[
\left\langle n^{|V_t|},n^{|W_t|},n^{|U_t|}\right\rangle.
\]

\begin{definition}[Centered lifting]
Let $P\subseteq S\subseteq F$, and let $T\in(\mathbb C^n)^{\otimes |F\setminus S|}$ have modes indexed by $F\setminus S$.
Define $\Gamma_S^P(T)\in(\mathbb C^{n+1})^{\otimes |F\setminus P|}$ by centering $T$ and replicating it along $S\setminus P$:
\begin{equation}
\Gamma_S^P(T)_{\mathbf j}
=\bigl(\mu^{\otimes |F\setminus S|}(T)\bigr)_{\mathbf j|_{F\setminus S}},
\qquad \mathbf j\in\ivl{0}{n}^{F\setminus P}.
\label{eq:centered-lifting}
\end{equation}
We abbreviate $\Gamma_S^\emptyset$ by $\Gamma_S$.
\end{definition}

To cancel cross terms that share replicated coordinates $P$, we aggregate
all lower-order liftings whose input blocks omit these coordinates.
\begin{definition}[Centered aggregate]
\label{def:centered-aggregate}
For $P\subseteq F$, define
\begin{align}
\mathcal A_{\mathcal S}(P)
={}&\sum_{\mathbf j\in\ivl{0}{n}^{F\setminus P}}
\left(\sum_{\substack{1\le t\le m\\P\subseteq U_t}}
\Gamma_{U_t}^P(X^{(t)})_{\mathbf j}\right)
\left(\sum_{\substack{1\le t\le m\\P\subseteq V_t}}
\Gamma_{V_t}^P(Y^{(t)})_{\mathbf j}\right) 
\left(\sum_{\substack{1\le t\le m\\P\subseteq W_t}}
\Gamma_{W_t}^P(Z^{(t)})_{\mathbf j}\right),
\label{eq:centered-correction-aggregate}
\end{align}
For nonempty $P$, the correction
$(-1)^{|P|}(n+1)^{|P|}\mathcal A_{\mathcal S}(P)$ is called a
\emph{singleton}, \emph{pair}, \emph{triple}, or \emph{quadruple}
correction when $|P|=1,2,3,4$, respectively. The name specifies the number of deleted coordinates.
\end{definition}
Its rank satisfies
\begin{equation}
R\bigl(\mathcal A_{\mathcal S}(P)\bigr)
\le (n+1)^{d-|P|}.
\label{eq:correction-rank}
\end{equation}

\begin{definition}[Mixed trilinear form]
\label{def:mixed-form}
For three selected tensor blocks $X^{(i)},Y^{(j)},Z^{(k)}$, whose modes
are indexed by $U_i^c,V_j^c,W_k^c$, respectively, define
\begin{equation}
\phi_{ijk}
=
\sum_{\boldsymbol\ell\in\ivl{0}{n}^{F}}
\Gamma_{U_i}(X^{(i)})_{\boldsymbol\ell}
\Gamma_{V_j}(Y^{(j)})_{\boldsymbol\ell}
\Gamma_{W_k}(Z^{(k)})_{\boldsymbol\ell}.
\label{eq:surviving-form}
\end{equation}
Its set of replicated coordinates is $\mathcal R_{ijk}=U_i\cap V_j\cap W_k$.

\end{definition}

Expanding the inner sums in Definition~\ref{def:centered-aggregate} gives
\begin{equation}
\mathcal A_{\mathcal S}(P)
=(n+1)^{-|P|}
\sum_{\substack{i,j,k\\P\subseteq\mathcal R_{ijk}}}\phi_{ijk},
\label{eq:correction-aggregate-expansion}
\end{equation}
since summing over each coordinate in $P$ contributes a factor $n+1$.
For any family $\mathcal Q$ of nonempty subsets of $F$, define
\begin{equation}
c_{\mathcal Q}(R)=1+
\sum_{\substack{P\in\mathcal Q\\P\subseteq R}}(-1)^{|P|}
\qquad(R\subseteq F).
\label{eq:partial-correction-coefficient}
\end{equation}
Then
\begin{equation}
\mathcal A_{\mathcal S}(\emptyset)
+\sum_{P\in\mathcal Q}(-1)^{|P|}(n+1)^{|P|}\mathcal A_{\mathcal S}(P)
=\sum_{i,j,k}c_{\mathcal Q}(\mathcal R_{ijk})\phi_{ijk}.
\label{eq:general-correction-identity}
\end{equation}

Retain the local multiplicity notation $\eta_r(i,j,k)$ from
Section~\ref{sec:aggregation}.
A mixed choice is annihilated by centering whenever $\eta_r=2$ for some $r$.
If it has no coordinate of multiplicity two and $\mathcal R_{ijk}\ne\emptyset$, its
replicated coordinates can be removed by inclusion--exclusion.
\begin{definition}[Correction support]
\label{def:correction-support}
The correction support is
\begin{equation}
\mathcal P(\mathcal S)
=
\left\{
\begin{array}{c|l}
P\subseteq F
& P\ne\emptyset,\ \text{and there exist }i,j,k\text{ such that}\vphantom{\displaystyle\sum_{r=1}^{d}}\\[-1mm]
& P\subseteq \mathcal R_{ijk},\quad
(i,j,k)\notin\{(t,t,t):1\le t\le m\},\quad
\eta_r(i,j,k)\ne2\ \text{for all }r
\end{array}
\right\}.
\label{eq:correction-support}
\end{equation}
\end{definition}
Let
\begin{equation}
\kappa_s=\bigl|\{P\in\mathcal P(\mathcal S):|P|=s\}\bigr|
\end{equation}
and define the multiplication budget for the main aggregate
$\mathcal A_{\mathcal S}(\emptyset)$ and the corrections
$\mathcal A_{\mathcal S}(P)$ indexed by $P\in\mathcal P(\mathcal S)$
\begin{equation}
\rho_{\mathcal S}(n)
=(n+1)^d+
\sum_{P\in\mathcal P(\mathcal S)}(n+1)^{d-|P|}
=(n+1)^d+\sum_{s=1}^d\kappa_s(n+1)^{d-s}.
\label{eq:rho-S}
\end{equation}

\begin{definition}[Mixed products after replicated-term cancellation]
\label{def:survivor-set}
The mixed choices not eliminated by centering and the corrections
indexed by $\mathcal P(\mathcal S)$ are
\begin{equation}
\mathcal B(\mathcal S)
=
\left\{(i,j,k)\notin\{(t,t,t):1\le t\le m\}:
\eta_r(i,j,k)\le1\ \text{for every }r\in F\right\}.
\label{eq:survivor-set}
\end{equation}
\end{definition}
For $(i,j,k)\in\mathcal B(\mathcal S)$, the remaining term is
$\phi_{ijk}$ of Definition~\ref{def:mixed-form}.
When the three selected part sizes sum to $d$, every local multiplicity in a
choice in $\mathcal B(\mathcal S)$ equals one and $\phi_{ijk}$ is itself a matrix multiplication
trilinear form.  If the selected part sizes sum to less than $d$, coordinates
with $\eta_r=0$ remain and $\phi_{ijk}$ is a centered trilinear form rather than a
matrix multiplication tensor.

For each replicated intersection, inclusion--exclusion subtracts the
aggregate with one coordinate deleted, adds those with two coordinates
deleted, and continues over its subsets. This preserves the desired
products and cancels every mixed product with a replicated coordinate.
The following identity gives both the remaining terms and the cost of
these corrections.

\begin{lemma}[Centered correction identity]
\label{lem:corr}
Define
\begin{equation}
\Phi_{\mathcal S}
=
\sum_{P\in\{\emptyset\}\cup\mathcal P(\mathcal S)}
(-1)^{|P|}(n+1)^{|P|}
\mathcal A_{\mathcal S}(P).
\label{eq:master-filter}
\end{equation}
For the target form $\phi_{\mathcal S}$ defined in
\eqref{eq:phi-S-labeled},
\begin{equation}
\Phi_{\mathcal S}
=
\phi_{\mathcal S}
+
\sum_{(i,j,k)\in\mathcal B(\mathcal S)}\phi_{ijk},
\qquad
R(\Phi_{\mathcal S})\le\rho_{\mathcal S}(n).
\label{eq:centered-filtered-identity}
\end{equation}
\end{lemma}

\begin{proof}
Apply \eqref{eq:general-correction-identity} with
$\mathcal Q=\mathcal P(\mathcal S)$.
For a diagonal choice, $\mathcal R_{ttt}=\emptyset$ and $Q^TQ=I_n$
recovers the desired product. If some $\eta_r=2$, then
$\phi_{ijk}=0$ by centering. For every remaining choice with
$\mathcal R_{ijk}\ne\emptyset$, all its nonempty subsets belong to
$\mathcal P(\mathcal S)$, so
\begin{equation}
c_{\mathcal P(\mathcal S)}(\mathcal R_{ijk})
=\sum_{P\subseteq\mathcal R_{ijk}}(-1)^{|P|}=0.
\end{equation}
The other non-diagonal choices are exactly $\mathcal B(\mathcal S)$.
Summing \eqref{eq:correction-rank} over the evaluated aggregates gives
the rank bound.
\end{proof}

For a fixed-shape family $\mathcal S\subseteq B_{s,t,v}$, every mixed choice
has
\begin{equation}
\sum_{r\in F}\eta_r(i,j,k)=|U_i|+|V_j|+|W_k|=t+s+v=d.
\end{equation}
Hence a choice with $\eta_r\le1$ for all $r$ must have $\eta_r=1$ for every
$r$.  In this case $\mathcal B(\mathcal S)$ consists exactly
of the non-diagonal compatible combinations of Definition~\ref{def:compat}.

If this fixed-shape family is conflict-free, no mixed products remain, so
the identity is already an exact algorithm.

Splitting the weighted aggregate of Algorithm~\ref{alg:Psi-bilinear-simultaneous-contractions}
by the coordinates $P$ at the padded index $n+1$ gives
\begin{equation}
\Phi^{\rm wt}_{\mathcal S}
=\sum_{P\subseteq F}(-1)^{|P|}(n+1)^{|P|}\mathcal A_{\mathcal S}(P).
\label{eq:weighted-inclusion-exclusion}
\end{equation}
Padding leaves only blocks that omit $P$, and their three weight factors
multiply to $(-1)^{|P|}(n+1)^{|P|}$. Thus the weighted algorithm evaluates
all subsets, while the corrected algorithm retains only those needed.

Since $\mathcal P(\mathcal S)$ consists of nonempty subsets of $F$,
\begin{equation}
\rho_{\mathcal S}(n)
\le
\sum_{P\subseteq F}(n+1)^{d-|P|}
=(n+2)^d.
\label{eq:correction-vs-padding}
\end{equation}
Thus the corrected algorithm uses at most the rank budget of the
weighted embedding, with lower counts when fewer subsets require correction.

\section{Border-rank bounds by parameter scaling}
\label{sec:grading}

We now develop border-rank constructions that eliminate mixed products by
scaling the variable blocks by powers of $\lambda$. This suppresses the
unwanted products in $\mathcal B(\mathcal S)$ of
Definition~\ref{def:survivor-set}, including conflicts between
tripartitions, and can also suppress replicated terms without computing
their corrections. Desired products have degree zero, while each
unwanted product is either canceled by corrections or has positive
degree and tends to zero with $\lambda$. The resulting limit gives a
border-rank bound~\cite{bini1980approximate,strassen1991degeneration}.

\begin{definition}[Grading]
\label{def:integral-grading}
A grading of $\mathcal S$ assigns
$(a_t,b_t,c_t)\in\mathbb Q^3$ for $t=1,\ldots,m$ such that
\begin{equation}
a_t+b_t+c_t=0
\label{eq:shifted-diagonal}
\end{equation}
and
\begin{equation}
a_i+b_j+c_k>0\qquad((i,j,k)\in\mathcal B(\mathcal S)).
\label{eq:integral-grading}
\end{equation}
The grading is integral when all exponents are integers.
\end{definition}
Rational exponents permit rescaling the smallest positive degree of the
remaining mixed products to one in Section~\ref{sec:approximation-error}.
Replacing $\lambda$ by a suitable positive integer power gives the same
family of tensors with integer exponents.

\begin{lemma}
\label{lem:kcolor-injective}
If $\mathcal S$ admits a grading, then each of the maps
$t\mapsto U_t$, $t\mapsto V_t$, and $t\mapsto W_t$ is injective.
\end{lemma}
\begin{proof}
If $U_i=U_j$ for $i\ne j$, then both $(i,j,j)$ and
$(j,i,i)$ lie in $\mathcal B(\mathcal S)$. Applying
\eqref{eq:integral-grading} to these choices and using $a_t+b_t+c_t=0$ gives
$a_i-a_j>0$ and $a_j-a_i>0$, a contradiction. The other two projections
are analogous.
\end{proof}

Replicated mixed terms of positive degree tend to zero with $\lambda$.
We therefore select corrections only for those of nonpositive degree.

\begin{definition}[Corrections for nonpositive degrees]
\label{def:degree-corrections}
For a grading of $\mathcal S$, define
\begin{equation}
\mathcal Q_{\le0}
=\bigcup_{a_i+b_j+c_k\le0}
  \{P:\emptyset\ne P\subseteq\mathcal R_{ijk}\},
\label{eq:target-correction-family}
\end{equation}
where the union is over non-diagonal triples with
$\eta_r(i,j,k)\ne2$ for every $r\in F$.
\end{definition}
Correction sets can be shared, so this choice may also cancel terms of
positive degree. A construction is \emph{correction-free} when its
selected correction family is empty.

More generally, let $\mathcal Q\subseteq\mathcal P(\mathcal S)$ be any selected
correction family. Using the
multiplication budget $\rho_{\mathcal S}(n)$ from \eqref{eq:rho-S}, the
budget for the main aggregate and the selected corrections is
\begin{equation}
\rho_{\mathcal S,\mathcal Q}(n)
=(n+1)^d+\sum_{P\in\mathcal Q}(n+1)^{d-|P|}
\le\rho_{\mathcal S}(n).
\label{eq:partial-correction-budget}
\end{equation}

For $\lambda>0$, substitute $\lambda^{a_t}X^{(t)}$,
$\lambda^{b_t}Y^{(t)}$, and $\lambda^{c_t}Z^{(t)}$ for the variable blocks
in \eqref{eq:centered-correction-aggregate}, and denote the resulting
aggregate by $\mathcal A_{\mathcal S,\lambda}(P)$. Then
\begin{equation}
\Phi_{\mathcal S,\mathcal Q}(\lambda)
=\mathcal A_{\mathcal S,\lambda}(\emptyset)
 +\sum_{P\in\mathcal Q}(-1)^{|P|}(n+1)^{|P|}
   \mathcal A_{\mathcal S,\lambda}(P).
\label{eq:partial-filter}
\end{equation}
Applying \eqref{eq:general-correction-identity} to the scaled blocks gives
\begin{equation}
\Phi_{\mathcal S,\mathcal Q}(\lambda)
=\phi_{\mathcal S}+
\sum_{\substack{(i,j,k)\ \text{non-diagonal}\\
\eta_r(i,j,k)\ne2\ \text{for all }r\in F}}
 c_{\mathcal Q}(\mathcal R_{ijk})\lambda^{a_i+b_j+c_k}\phi_{ijk}.
\label{eq:filtered-identity}
\end{equation}
When $\mathcal Q=\mathcal P(\mathcal S)$, this is the scaled version of
\eqref{eq:master-filter}, with only the mixed products in
$\mathcal B(\mathcal S)$ remaining.

\begin{definition}[Remaining mixed products]
\label{def:aggregate-index-sets}
For a tripartition family $\mathcal S$ and selected corrections
$\mathcal Q\subseteq\mathcal P(\mathcal S)$, let
\begin{equation}
\mathcal J=\left\{(i,j,k)\ \text{non-diagonal}:
\eta_r(i,j,k)\ne2\ \text{for every }r\in F,\quad
c_{\mathcal Q}(\mathcal R_{ijk})\ne0\right\}.
\label{eq:residual-index-set}
\end{equation}
When $\mathcal Q=\mathcal P(\mathcal S)$, we have
$\mathcal J=\mathcal B(\mathcal S)$ of Definition~\ref{def:survivor-set}.
\end{definition}

The following theorem shows that this scaling makes the unwanted
products tend to zero while using at most $\rho_{\mathcal S,\mathcal Q}(n)$
rank-one terms for the main aggregate and its corrections.
\begin{theorem}[Border rank by parameter scaling]
\label{thm:graded}
Let $\mathcal Q\subseteq\mathcal P(\mathcal S)$. If a grading of
$\mathcal S$ has $a_i+b_j+c_k>0$ for every remaining mixed product
$(i,j,k)\in\mathcal J$ of Definition~\ref{def:aggregate-index-sets}, then
\begin{equation}
\underline R\left(\bigoplus_{t=1}^{m}
\left\langle n^{|V_t|},n^{|W_t|},n^{|U_t|}\right\rangle\right)
\le\rho_{\mathcal S,\mathcal Q}(n)\le\rho_{\mathcal S}(n).
\label{eq:graded-border-rank}
\end{equation}
\end{theorem}
\begin{proof}
By \eqref{eq:correction-rank}, the main aggregate and selected corrections
use at most $\rho_{\mathcal S,\mathcal Q}(n)$ rank-one terms for every
$\lambda>0$. In \eqref{eq:filtered-identity}, each desired product has
coefficient one, and every remaining mixed product has positive degree.
Hence $\Phi_{\mathcal S,\mathcal Q}(\lambda)\to\phi_{\mathcal S}$ as
$\lambda\to0$, giving the border-rank bound.
\end{proof}

The choice $\mathcal Q=\mathcal Q_{\le0}$ satisfies this hypothesis because
every replicated term of nonpositive degree has
\begin{equation}
c_{\mathcal Q_{\le0}}(\mathcal R_{ijk})
=\sum_{P\subseteq\mathcal R_{ijk}}(-1)^{|P|}=0.
\label{eq:degree-correction-condition}
\end{equation}
All remaining replicated terms have positive degree, as do the products
in $\mathcal B(\mathcal S)$ by the grading condition.

\section{Stability and retained precision}
\label{sec:stability}

We now study the approximation error introduced by $\lambda$ and the
sensitivity of the bilinear algorithm to round-off error. The first
depends on the remaining mixed products. For the second, we distinguish
input sensitivity, which depends on the bilinear map, from amplification
of errors in the intermediate vectors, which depends on its factorization.
The product of the three factor norms bounds the latter. We compare this
product with that of classical multiplication and balance approximation
error against round-off error at a fixed working precision.

Bini and Lotti analyzed stability of fast matrix multiplication~\cite{bini1980stability},
and Bini, Lotti, and Romani studied approximate bilinear
algorithms~\cite{bini1980approximate}. Subsequent analyses quantify the
effects of recursion, coefficient magnitudes, and
scaling~\cite{demmel2007stable,ballard2016stability,dai2023nuclear}.
We use the product of spectral norms appearing in the bilinear norm
estimates of Ju and Solomonik~\cite[Theorem~9.1 and Corollary~9.2]{ju2020convolution}
to bound the propagation of intermediate rounding errors.
For centered constructions with distinct part labels, the row Gram
matrices of the factors are block diagonal, with each diagonal block equal to a multiple of the identity matrix. Their spectral norms
therefore follow directly from these scalars, which depend on the
tripartition set, corrections, and parameter exponents.

After grouping the modes into matrix indices, write
\begin{equation}
X^{(t)}\in\mathbb C^{n^{|V_t|}\times n^{|W_t|}},\qquad
Y^{(t)}\in\mathbb C^{n^{|W_t|}\times n^{|U_t|}},\qquad
x=(\vcr X^{(t)})_{t=1}^{m},\quad y=(\vcr Y^{(t)})_{t=1}^{m}.
\end{equation}
The target is the bilinear map
\begin{equation}
\mathcal M(x,y)=\bigl(\vcr(X^{(t)}Y^{(t)})\bigr)_{t=1}^{m}.
\label{eq:target-bilinear-map}
\end{equation}
\begin{sloppypar}
In the following analysis, vector norms are Euclidean, matrix norms are spectral, and
$\|\mathcal L\|=\sup_{\|x\|=\|y\|=1}\|\mathcal L(x,y)\|$ for a
bilinear map $\mathcal L$. The norm inequalities below hold over
$\mathbb R$ and $\mathbb C$; the centering matrix $Q$ is real with
$Q^TQ=I$ and $\mathbf1^TQ=0$.
\end{sloppypar}

\subsection{Approximation error}
\label{sec:approximation-error}

For the selected corrections $\mathcal Q$, the remaining mixed products
$\mathcal J$ of Definition~\ref{def:aggregate-index-sets} contribute to the
approximation error with coefficients \eqref{eq:partial-correction-coefficient}.

For $\mathcal J\ne\emptyset$, let $g>0$ be the smallest value of
$a_i+b_j+c_k$ among $(i,j,k)\in\mathcal J$. In the error analysis we
divide all block exponents by $g$, keeping the correction family fixed.
The resulting tensor at parameter $\lambda$ is the
original tensor in \eqref{eq:partial-filter} at parameter $\lambda^{1/g}$.
Thus we may normalize
the smallest degree among the remaining mixed products to one.

\begin{definition}[Normalized grading and growth exponent]
\label{def:precision-order}
For $\mathcal J\ne\emptyset$, a grading is normalized if
\begin{equation}
\min_{(i,j,k)\in\mathcal J}(a_i+b_j+c_k)=1.
\label{eq:shifted-survivor}
\end{equation}
Its growth exponent is
\begin{equation}
h=-\min_{1\le t\le m}a_t-\min_{1\le t\le m}b_t
  -\min_{1\le t\le m}c_t.
\label{eq:grading-height}
\end{equation}
For exact constructions, we use zero exponents and $h=0$.
\end{definition}

\begin{definition}[Parameter-dependent bilinear map]
\label{def:parameter-dependent-map}
For a normalized grading from Definition~\ref{def:precision-order}, let
\begin{equation}
e_{ijk}=a_i+b_j+c_k\qquad((i,j,k)\in\mathcal J).
\label{eq:residual-exponent}
\end{equation}
For $0<\lambda\le1$, let $z$ stack the variables $Z^{(t)}$ in the output
order of \eqref{eq:target-bilinear-map}. Define $\mathcal M_\lambda$ by
\begin{equation}
\bigl\langle\mathcal M_\lambda(x,y),z\bigr\rangle
=\phi_{\mathcal S}(X,Y,Z)
 +\sum_{(i,j,k)\in\mathcal J}
 c_{\mathcal Q}(\mathcal R_{ijk})\lambda^{e_{ijk}}\phi_{ijk}.
\label{eq:parameter-dependent-map}
\end{equation}
This is the form in \eqref{eq:filtered-identity}, evaluated with the
normalized grading.
\end{definition}

\begin{definition}[Approximation bound]
\label{def:approximation-bound}
For the construction in Definition~\ref{def:parameter-dependent-map}, define
\begin{equation}
\beta(\lambda)=\sum_{(i,j,k)\in\mathcal J}
 |c_{\mathcal Q}(\mathcal R_{ijk})|
 (n+1)^{|\mathcal R_{ijk}|}\lambda^{e_{ijk}},
 \qquad 0\le\lambda\le1.
\label{eq:approximation-bound}
\end{equation}
\end{definition}

\begin{lemma}[Approximation error]
\label{lem:degeneration-conditioning}
For $\mathcal M_\lambda$ from Definition~\ref{def:parameter-dependent-map}
and $0<\lambda\le1$,
\begin{equation}
\|\mathcal M_\lambda-\mathcal M\|\le\beta(\lambda).
\label{eq:bias-conditioning}
\end{equation}
\end{lemma}
\begin{proof}
For $(i,j,k)\in\mathcal J$, set
$\bar X=\mu^{\otimes |U_i^c|}(X^{(i)})$,
$\bar Y=\mu^{\otimes |V_j^c|}(Y^{(j)})$, and
$\bar Z=\mu^{\otimes |W_k^c|}(Z^{(k)})$.
By \eqref{eq:centered-lifting},
$\Gamma_{U_i}(X^{(i)})_{\boldsymbol\ell}=\bar X_{\boldsymbol\ell|_{U_i^c}}$,
and similarly for $Y,Z$. Coordinates in $\mathcal R_{ijk}$ contribute
$n+1$ each. Group the coordinates retained by all three tensors into
$\mathbf u$, and the coordinates retained by each pair into
$\mathbf a,\mathbf b,\mathbf c$. After ordering the modes accordingly,
\begin{equation}
\phi_{ijk}=(n+1)^{|\mathcal R_{ijk}|}
 \sum_{\mathbf u,\mathbf a,\mathbf b,\mathbf c}
 \bar X_{\mathbf u,\mathbf a,\mathbf b}
 \bar Y_{\mathbf u,\mathbf b,\mathbf c}
 \bar Z_{\mathbf u,\mathbf a,\mathbf c}.
\label{eq:mixed-form-expansion}
\end{equation}
For each matrix slice, use
$|\langle AB,C\rangle|\le\|A\|_F\|B\|_F\|C\|_F$.
Bound one slice norm by its full tensor norm and apply Cauchy--Schwarz
to the other two slice norms. The isometry of $\mu$ then gives
\begin{equation}
|\phi_{ijk}|\le(n+1)^{|\mathcal R_{ijk}|}
 \|X^{(i)}\|_F\|Y^{(j)}\|_F\|Z^{(k)}\|_F.
\label{eq:mixed-form-norm}
\end{equation}
For $0<\lambda\le1$, apply the triangle inequality to
\eqref{eq:parameter-dependent-map} and bound each block norm by the norm
of its stacked vector. Taking the supremum over unit vectors gives
$\|\mathcal M_\lambda-\mathcal M\|\le\beta(\lambda)$.
\end{proof}

\subsection{Round-off analysis of the bilinear algorithms}
\label{sec:coefficient-growth}

The approximation bound does not account for errors made while evaluating
the algorithm in finite precision. We distinguish sensitivity to the
input data from the additional errors introduced by the factorization.

A bilinear algorithm is represented by a triple of factor matrices
$\Xi=(A,B,C)$ with $r$ columns. Writing their columns as $a_\ell,b_\ell,c_\ell$,
its evaluation is
\begin{equation}
z=C\bigl((A^Tx)\odot(B^Ty)\bigr)
 =\sum_{\ell=1}^r c_\ell(a_\ell^Tx)(b_\ell^Ty).
\label{eq:cp-evaluation}
\end{equation}
Here $\odot$ denotes entrywise multiplication: $(u\odot v)_\ell=u_\ell v_\ell$.
\begin{definition}[Spectral-norm product]
\label{def:spectral-product}
For a bilinear algorithm $\Xi=(A,B,C)$, define
\begin{equation}
\sigma(\Xi)=\|A\|_2\|B\|_2\|C\|_2.
\label{eq:coefficient-growth}
\end{equation}
\end{definition}

The factorization specifies two linear input maps, the entrywise product,
and a linear output map. It does not prescribe how the linear maps are
applied. Their implementations may use different summation orders,
compensated summation, or factorizations of the coefficient matrices
rather than direct matrix--vector multiplication. These choices can affect
round-off error~\cite{ballard2016stability,ju2020convolution}.

We analyze the propagation of errors in the inputs and outputs of these
stages. Specifically, we hold $A,B,C$ exact and model each stage as an
exact operation followed by rounding its result at the working precision.
The errors in the intermediate vectors $A^Tx$, $B^Ty$, and their entrywise
product depend on the chosen factorization. In contrast, perturbing the
original inputs and then evaluating exactly depends only on the bilinear
map, and rounding the final result is relative to that result.

The next lemma shows how these norms bound errors introduced in the
intermediate vectors. It keeps the contribution of the original input
perturbations separate, using the norm of the bilinear map rather than
replacing it by $\sigma(\Xi)$. The latter replacement also gives a valid
upper bound, as in~\cite[Theorem~9.1]{ju2020convolution}, but does not
distinguish input sensitivity from errors caused by the algorithm.

\begin{lemma}[Input sensitivity and intermediate round-off]
\label{lem:rounded-bilinear}
Let $\Xi=(A,B,C)$ evaluate the bilinear map $\mathcal L$ in
\eqref{eq:cp-evaluation}, and let $z=\mathcal L(x,y)$. For $0<\epsilon<1$,
suppose
\begin{equation}
\widehat x=x+\delta x,\quad \widehat y=y+\delta y,\qquad
\|\delta x\|_2\le\epsilon\|x\|_2,\quad
\|\delta y\|_2\le\epsilon\|y\|_2.
\label{eq:input-perturbations}
\end{equation}
Let $\operatorname{fl}$ satisfy
$\|\operatorname{fl}(v)-v\|_2\le\epsilon\|v\|_2$ for every rounded vector,
and compute
\begin{equation}
\widehat u=\operatorname{fl}(A^T\widehat x),\quad
\widehat v=\operatorname{fl}(B^T\widehat y),\quad
\widehat w=\operatorname{fl}(\widehat u\odot\widehat v),\quad
\widehat z=\operatorname{fl}(C\widehat w).
\label{eq:rounded-evaluation}
\end{equation}
Then
\begin{equation}
\|\widehat z-z\|_2
\le\bigl((1+\epsilon)^3-1\bigr)
 \bigl[\|\mathcal L\|+(1+\epsilon)^3\sigma(\Xi)\bigr]
 \|x\|_2\|y\|_2.
\label{eq:factor-perturbation}
\end{equation}
\end{lemma}
\begin{proof}
First fix the supplied inputs $\widehat x,\widehat y$. Set
$u=A^T\widehat x$, $v=B^T\widehat y$,
$\delta u=\widehat u-u$, $\delta v=\widehat v-v$, and
$\delta w=\widehat w-\widehat u\odot\widehat v$. Then
\begin{equation}
\widehat w-u\odot v
 =\delta u\odot v+u\odot\delta v+\delta u\odot\delta v+\delta w.
\label{eq:intermediate-error-expansion}
\end{equation}
The rounding bounds give
$\|\delta u\|_2\le\epsilon\|u\|_2$,
$\|\delta v\|_2\le\epsilon\|v\|_2$, and
$\|\delta w\|_2\le\epsilon(1+\epsilon)^2\|u\|_2\|v\|_2$.
Using $\|u\odot v\|_2\le\|u\|_2\|v\|_2$ yields
\begin{align}
\|C\widehat w-\mathcal L(\widehat x,\widehat y)\|_2
&\le \|C\|_2\bigl(2\epsilon+\epsilon^2+
       \epsilon(1+\epsilon)^2\bigr)\|u\|_2\|v\|_2\notag\\
&\le\bigl((1+\epsilon)^3-1\bigr)
 \|C\|_2\|A\|_2\|B\|_2\|\widehat x\|_2\|\widehat y\|_2.
\label{eq:intermediate-roundoff}
\end{align}
This is the error introduced by the three intermediate roundings.
The original input perturbations instead satisfy, by bilinearity,
\begin{equation}
\|\mathcal L(\widehat x,\widehat y)-\mathcal L(x,y)\|_2
\le(2\epsilon+\epsilon^2)\|\mathcal L\|\|x\|_2\|y\|_2.
\label{eq:intrinsic-input-error}
\end{equation}
Finally, rounding $C\widehat w$ and adding and subtracting
$\mathcal L(\widehat x,\widehat y)$ gives
\begin{align}
\|\widehat z-z\|_2
&\le(1+\epsilon)\|C\widehat w-\mathcal L(\widehat x,\widehat y)\|_2\notag\\
&\quad +(1+\epsilon)\|\mathcal L(\widehat x,\widehat y)-z\|_2
       +\epsilon\|z\|_2.
\label{eq:final-roundoff-separation}
\end{align}
Use $\|\widehat x\|_2\|\widehat y\|_2\le
(1+\epsilon)^2\|x\|_2\|y\|_2$ and
$\|z\|_2\le\|\mathcal L\|\|x\|_2\|y\|_2$.
The coefficient of $\|\mathcal L\|$ is
$(1+\epsilon)(2\epsilon+\epsilon^2)+\epsilon=(1+\epsilon)^3-1$,
which proves the bound.
\end{proof}
To first order, the two contributions to the bound are
$3\epsilon\|\mathcal L\|\|x\|_2\|y\|_2$ and
$3\epsilon\sigma(\Xi)\|x\|_2\|y\|_2$, respectively.
With exact original inputs and no rounding of the final output, only
\eqref{eq:intermediate-roundoff} is needed. The assumptions exclude
overflow and underflow and also apply to complex vectors with the stated
normwise bounds.

Within this error model, the intermediate-error bound is sharp in general,
but need not have the
best constant for a particular factorization. To see the possible
amplification, choose unit inputs with
$|a_\ell^Tx|=\|a_\ell\|_2$ and $|b_\ell^Ty|=\|b_\ell\|_2$,
and allow a relative error of magnitude $\epsilon$ in only product
$\ell$. With the other operations exact, the resulting error is
\begin{equation}
\|\delta z\|_2
=\epsilon\|a_\ell\|_2\|b_\ell\|_2\|c_\ell\|_2.
\label{eq:single-product-roundoff}
\end{equation}
For one rank-one term this product equals $\sigma(\Xi)$; allowing all
three intermediate roundings with the same sign attains the coefficient
$(1+\epsilon)^3-1$. Section~\ref{sec:aggregation-factor-norms} establishes
sharpness of the parameter-growth exponent for the centered constructions.

For the target map $\mathcal M$ of \eqref{eq:target-bilinear-map},
$\|\mathcal M\|=1$. Indeed,
\begin{equation}
\|\mathcal M(x,y)\|_2^2
=\sum_{t=1}^m\|X^{(t)}Y^{(t)}\|_F^2
\le\sum_{t=1}^m\|X^{(t)}\|_F^2\|Y^{(t)}\|_F^2
\le\|x\|_2^2\|y\|_2^2,
\label{eq:target-map-norm}
\end{equation}
and equality is attained by inputs selecting a single scalar product.
For $\mathcal M_\lambda$, Lemma~\ref{lem:degeneration-conditioning} gives
$\|\mathcal M_\lambda\|\le1+\beta(\lambda)$. Thus the input-sensitivity
term stays bounded as $\lambda\to0$, even when the intermediate-error
bound grows.

\begin{lemma}[Classical spectral-norm product]
\label{lem:classical-spectral-product}
For classical evaluation of the independent products in
\eqref{eq:target-bilinear-map}, the spectral-norm product is
\begin{equation}
\sigma_c=n^{(\max_t|U_t|+\max_t|V_t|+\max_t|W_t|)/2}.
\label{eq:classical-sigma}
\end{equation}
For a balanced family $\mathcal S\subseteq B_{N,N,N}$, the matrix
dimension is $q=n^N$, so $\sigma_c=q^{3/2}=n^{3N/2}$, independently of
the number of products.
\end{lemma}
\begin{proof}
Write $\Xi=(A,B,C)$ for the classical algorithm. Its columns are
indexed by the scalar products
$X^{(t)}_{\alpha\beta}Y^{(t)}_{\beta\gamma}$: the column of $A$ selects
$X^{(t)}_{\alpha\beta}$, that of $B$ selects
$Y^{(t)}_{\beta\gamma}$, and that of $C$ adds the product to output
entry $(\alpha,\gamma)$ of block $t$. Every column is a standard basis
vector. Each input entry of $X^{(t)}$ is used $n^{|U_t|}$ times, each
input entry of $Y^{(t)}$ is used $n^{|V_t|}$ times, and each output entry
accumulates $n^{|W_t|}$ products. Consequently,
\begin{equation}
AA^*=\bigoplus_{t=1}^m n^{|U_t|}I,\qquad
BB^*=\bigoplus_{t=1}^m n^{|V_t|}I,\qquad
CC^*=\bigoplus_{t=1}^m n^{|W_t|}I,
\end{equation}
with identity blocks of the corresponding input or output dimensions.
Taking the square root of the largest coefficient in each matrix gives
\eqref{eq:classical-sigma}.
\end{proof}

The ratio $\sigma(\Xi)/\sigma_c$ compares the bounds for intermediate
rounding errors under the same assumptions.
Tensor products multiply $\sigma(\Xi)$,
so $k$ recursive uses of a fixed square base raise its normalized
ratio to the $k$th power.

\subsection{Factor norms of the aggregation constructions}
\label{sec:aggregation-factor-norms}

The construction of Definition~\ref{def:parameter-dependent-map} has
factor matrices
\begin{equation}
\Xi(\lambda)=(A(\lambda),B(\lambda),C(\lambda)),\qquad
\mathcal M_\lambda(x,y)
=C(\lambda)\bigl((A(\lambda)^Tx)\odot(B(\lambda)^Ty)\bigr).
\label{eq:parameterized-algorithm}
\end{equation}

We now compute the spectral-norm product of
Definition~\ref{def:spectral-product} from the row Gram matrices
$A(\lambda) A(\lambda)^*$, $B(\lambda) B(\lambda)^*$, and
$C(\lambda) C(\lambda)^*$. Each is block diagonal in the original tensor
blocks, and each diagonal block is a scalar multiple of the identity.
For a balanced tripartition set without corrections or parameter
scaling, all these scalars equal $(n+1)^N$. Corrections and parameter
exponents change the scalars while preserving this diagonal structure.
Thus the spectral norms are the square roots of the largest block
coefficients.

Recall the selected correction family $\mathcal Q$ from
Theorem~\ref{thm:graded}. The main aggregate has coefficient
one, and the correction indexed by $P\in\mathcal Q$ has coefficient
$(-1)^{|P|}(n+1)^{|P|}$. We distribute $(n+1)^{|P|/3}$ to each of its
three factors and put the sign in $C(\lambda)$. The blocks use the
scalings $\lambda^{a_t},\lambda^{b_t},\lambda^{c_t}$, respectively.
Uniform reciprocal scalings of the three factor matrices leave
$\sigma(\Xi(\lambda))$ unchanged, although scaling different columns
independently can change it. This distribution agrees with the weighted expansion
\eqref{eq:weighted-inclusion-exclusion}, up to column signs.
The part projections are injective for conflict-free families by
Lemma~\ref{lem:conflict-injective}, and for graded families by
Lemma~\ref{lem:kcolor-injective}.

\begin{proposition}[Spectral norms of the centered factors]
\label{prop:spectral-factors}
For a set of complete tripartitions with injective part projections
and the factors of \eqref{eq:parameterized-algorithm}, the row Gram matrices
are block diagonal, with each diagonal block equal to a multiple of the
identity matrix. Their spectral norms satisfy
\begin{align}
\|A(\lambda)\|_2^2
 &=\max_{1\le t\le m}\lambda^{2a_t}(n+1)^{|U_t|}
       \left(1+\sum_{\substack{P\in\mathcal Q\\P\subseteq U_t}}
          (n+1)^{-|P|/3}\right),\notag\\
\|B(\lambda)\|_2^2
 &=\max_{1\le t\le m}\lambda^{2b_t}(n+1)^{|V_t|}
       \left(1+\sum_{\substack{P\in\mathcal Q\\P\subseteq V_t}}
          (n+1)^{-|P|/3}\right),\notag\\
\|C(\lambda)\|_2^2
 &=\max_{1\le t\le m}\lambda^{2c_t}(n+1)^{|W_t|}
       \left(1+\sum_{\substack{P\in\mathcal Q\\P\subseteq W_t}}
          (n+1)^{-|P|/3}\right).
\label{eq:spectral-factor-norms}
\end{align}
\end{proposition}
\begin{proof}
For $P=\emptyset$ or $P\in\mathcal Q$ with $P\subseteq U_t$, the unscaled lifting applies $Q$ in
$U_t^c$ and is constant in $U_t\setminus P$. In its Gram matrix, each
retained coordinate gives $Q^TQ=I$, while each constant coordinate gives
$\mathbf1^T\mathbf1=n+1$. Hence this matrix is
$(n+1)^{|U_t|-|P|}I$.

To compute a cross block for distinct $U_t,U_s$ with
$P\subseteq U_t\cap U_s$, choose a coordinate in their symmetric
difference. Interchanging $s,t$ if necessary, let this coordinate be
$r\in U_t\setminus U_s$. It is not deleted by $P$. The $t$th lifting
is constant in that coordinate, whereas the $s$th lifting uses a column
of $Q$. The inner product therefore contains the factor
\begin{equation}
\sum_{j_r=0}^{n}1\cdot Q_{j_r b}=0
\qquad\text{for every column }b,
\end{equation}
since the columns of $Q$ have zero sum. Thus every cross block vanishes.

Multiplying each lifting by $(n+1)^{|P|/3}\lambda^{a_t}$ makes its
Gram block equal to
$(n+1)^{|U_t|-|P|/3}\lambda^{2a_t}I$.
The main and correction aggregates occupy disjoint column blocks of
$A(\lambda)$, so their contributions add to give
\begin{equation}
A(\lambda) A(\lambda)^*
=\bigoplus_{t=1}^{m}
\lambda^{2a_t}(n+1)^{|U_t|}
\left(1+\sum_{\substack{P\in\mathcal Q\\P\subseteq U_t}}
 (n+1)^{-|P|/3}\right)I_{n^{d-|U_t|}}.
\end{equation}
Taking the largest scalar block coefficient gives the first equality.
The same argument gives the other two.
\end{proof}

For a balanced tripartition set $\mathcal S\subseteq B_{N,N,N}$, each part has size $N$ and $q=n^N$.
An ungraded, correction-free aggregate has
$A(1)A(1)^*=B(1)B(1)^*=C(1)C(1)^*=(n+1)^NI$.
The weighted embedding instead has Gram blocks
$(n+1+(n+1)^{2/3})^NI$, since the weight vector of
Definition~\ref{def:weight-vector} has squared norm $n+1+(n+1)^{2/3}$.
For any selected correction family, we bound the sum over its correction
subsets by the sum over all subsets. For any finite set $E$ and scalar $z$,
grouping subsets by their cardinality and applying the binomial theorem gives
\begin{equation}
\sum_{P\subseteq E}z^{|P|}
=\sum_{s=0}^{|E|}\binom{|E|}{s}z^s
=(1+z)^{|E|}.
\label{eq:subset-binomial}
\end{equation}
The coefficient $\binom{|E|}{s}$ counts the subsets of size $s$.
Applying \eqref{eq:subset-binomial} with $E=U_t$, $|U_t|=N$, and
$z=(n+1)^{-1/3}$, and including the main aggregate separately, yields
\begin{equation}
1\le 1+\sum_{\substack{P\in\mathcal Q\\P\subseteq U_t}}(n+1)^{-|P|/3}
\le\sum_{P\subseteq U_t}(n+1)^{-|P|/3}=(1+(n+1)^{-1/3})^N.
\label{eq:correction-subset-bound}
\end{equation}
The same calculation applies to $V_t,W_t$. Thus every exact construction
from a balanced tripartition set satisfies
\begin{equation}
\left(\frac{n+1}{n}\right)^{3N/2}
\le\frac{\sigma(\Xi(1))}{\sigma_c}
\le\left(\frac{n+1+(n+1)^{2/3}}{n}\right)^{3N/2},
\qquad
\lim_{n\to\infty}\frac{\sigma(\Xi(1))}{\sigma_c}=1
\quad(N\text{ fixed}).
\label{eq:exact-sigma-ratio}
\end{equation}
The upper bound equals $\sigma(\Xi(1))/\sigma_c$ for the weighted embedding
of Definition~\ref{def:weight} applied to the same tripartition set
$\mathcal S$. The weighted embedding uses local dimension $n+2$, whereas the
selected aggregates use dimension $n+1$ with only the corrections indexed
by $\mathcal Q$.
By \eqref{eq:weighted-inclusion-exclusion}, the corrected factors retain
selected column blocks of the weighted factors, up to column signs.
Deleting the other blocks reduces the multiplication count and cannot
increase their spectral norms, whether the resulting construction is
exact or approximate. For an exact construction from complete
tripartitions of unequal shapes, the same argument replaces $3N$ by
$\max_t|U_t|+\max_t|V_t|+\max_t|W_t|$.

The smallest exponent in each factor determines its growth as
$\lambda\to0$. Their sum is $-h$ from \eqref{eq:grading-height}, so the
spectral-norm product grows as $\lambda^{-h}$. The parameter choice also
requires its leading coefficient.
\begin{definition}[Leading spectral coefficient]
\label{def:leading-sigma}
For the factorization $\Xi(\lambda)$ and growth exponent $h$, define the constant
$\sigma_0$, when the limit exists and is positive, by
\begin{equation}
\sigma_0=\lim_{\lambda\to0}\lambda^h\sigma(\Xi(\lambda)).
\label{eq:leading-sigma}
\end{equation}
\end{definition}
For the first maximum in \eqref{eq:spectral-factor-norms}, denote by
$K_{A,t}$ the coefficient of $\lambda^{2a_t}$, and set
$a_* = \min_t a_t$ and $K_A=\max_{a_t=a_*}K_{A,t}$, with
$K_B,K_C$ defined analogously. Then
\begin{equation}
\lambda^{-2a_*}\|A(\lambda)\|_2^2
=\max_t K_{A,t}\lambda^{2(a_t-a_*)}=K_A
\label{eq:small-parameter-spectral-max}
\end{equation}
for all sufficiently small positive $\lambda$, as every term with $a_t>a_*$
tends to zero and is eventually smaller than the positive constant $K_A$.
Taking square roots gives $\|A(\lambda)\|_2=\sqrt{K_A}\lambda^{a_*}$,
and multiplying the three norm identities yields
\begin{equation}
\sigma(\Xi(\lambda))=\sigma_0\lambda^{-h},\qquad
\sigma_0=(K_AK_BK_C)^{1/2}
\label{eq:exact-small-parameter-sigma}
\end{equation}
for all sufficiently small positive $\lambda$, with the construction and $n$
fixed. 

For these factors, the exponent $h$ also describes attainable
amplification of intermediate rounding errors under the stated model.
Fix a column $\ell$ belonging to the main aggregate. In its $t$th input
block, the column of $A(\lambda)$ has norm
$\lambda^{a_t}(n/(n+1))^{|U_t^c|/2}$, since
$QQ^T=I-1_{n+1}1_{n+1}^T/(n+1)$; the corresponding formulas hold for
$B(\lambda)$ and $C(\lambda)$. Each block is nonzero, so
\begin{equation}
0<\lim_{\lambda\to0}\lambda^h
 \|a_\ell(\lambda)\|_2\|b_\ell(\lambda)\|_2\|c_\ell(\lambda)\|_2<\infty.
\label{eq:sharp-growth-column}
\end{equation}
By \eqref{eq:single-product-roundoff}, rounding just this product can
therefore cause an error proportional to $\epsilon\lambda^{-h}$ for
suitable unit inputs. Thus the power of $\lambda$ in the bound is sharp
for fixed construction and local dimension; this does not assert that
the bound's constant is attained for every factorization or input.

In the balanced case,
\begin{equation}
\left(\frac{n+1}{n}\right)^{3N/2}
\le \frac{\sigma(\Xi(\lambda))}{\sigma_c\lambda^{-h}}
\le \left(\frac{n+1+(n+1)^{2/3}}{n}\right)^{3N/2}.
\label{eq:normalized-sigma-ratio}
\end{equation}
The dimension-dependent factor therefore tends to one; the parameter
loss is isolated in $\lambda^{-h}$. Formula~\eqref{eq:spectral-factor-norms}
also evaluates the complete-tripartition constructions with unequal part sizes.
The incomplete big-CW restrictions use the Gram calculation in
Appendix~\ref{apx:incomplete-sigma}.

\subsection{Parameter choice and retained accuracy}
\label{sec:parameter-choice}

Making $\lambda$ smaller reduces the approximation error but can increase
round-off amplification in the intermediate vectors. With smallest
remaining degree one, these contributions have orders $\lambda$ and
$\epsilon\lambda^{-h}$. The contribution independent of the factorization is of order
$\epsilon$. Thus for $h>0$, the balance determining the limiting loss of precision is between
the approximation and intermediate-rounding errors. All limits in this
subsection keep the construction, local dimension, and input vectors fixed.

\begin{definition}[Fraction of retained accuracy]
\label{def:retained-accuracy}
For $x,y$ with $\mathcal M(x,y)\ne0$ and a parameter choice
$\lambda=\lambda(\epsilon)\in(0,1]$, define the relative-error bounds
\begin{align}
e(\epsilon,x,y)
 &=\frac{\|x\|_2\|y\|_2}{\|\mathcal M(x,y)\|_2}
 \Bigl[(1+\epsilon)^3\beta(\lambda)
 +\bigl((1+\epsilon)^3-1\bigr)
       \bigl(1+(1+\epsilon)^3\sigma(\Xi(\lambda))\bigr)\Bigr],
 \label{eq:retained-precision-error}\\
e_c(\epsilon,x,y)
 &=\frac{\|x\|_2\|y\|_2}{\|\mathcal M(x,y)\|_2}
 \bigl((1+\epsilon)^3-1\bigr)
       \bigl(1+(1+\epsilon)^3\sigma_c\bigr).
 \label{eq:classical-reference-error}
\end{align}
For $0<e(\epsilon,x,y),e_c(\epsilon,x,y)<1$, the fraction of
classical-reference accuracy retained by the bound is
\begin{equation}
\theta(\epsilon,x,y)=
\frac{\log(1/e(\epsilon,x,y))}{\log(1/e_c(\epsilon,x,y))}.
\label{eq:finite-precision-fraction}
\end{equation}
The limiting fraction is $\theta=\lim_{\epsilon\to0}\theta(\epsilon,x,y)$,
when this limit exists.
\end{definition}
The two bounds include the same model for input and output rounding.
Their algorithm-dependent terms use $\sigma(\Xi(\lambda))$ and
$\sigma_c$, respectively. The logarithms measure the digits guaranteed
by these bounds; their ratio need not equal the limiting fraction at a
particular working precision.

\begin{lemma}[Parameter choice and retained accuracy]
\label{lem:retained-accuracy}
Under the rounding assumptions of Lemma~\ref{lem:rounded-bilinear}, for
$0<\lambda\le1$ and $\mathcal M(x,y)\ne0$,
\begin{align}
\frac{\|\widehat z-\mathcal M(x,y)\|_2}{\|x\|_2\|y\|_2}
&\le (1+\epsilon)^3\beta(\lambda)+\bigl((1+\epsilon)^3-1\bigr)
       \bigl(1+(1+\epsilon)^3\sigma(\Xi(\lambda))\bigr).
\label{eq:conditioned-error}
\end{align}
For a construction from Theorem~\ref{thm:graded} with
$\mathcal J\ne\emptyset$, use its normalized grading from
Definition~\ref{def:precision-order} and let
\begin{equation}
\lambda_{\rm bal}(\epsilon)=\min\left\{1,
\left(\frac{3\epsilon\sigma_0}{\beta'(0)}\right)^{1/(h+1)}\right\}.
\label{eq:parameter-choice}
\end{equation}
Then
\begin{equation}
e(\epsilon,x,y)=O\!\left(\epsilon^{1/(h+1)}\right),\qquad
\theta=\frac{1}{1+h}.
\label{eq:target-retained-fraction}
\end{equation}
For an exact construction, $\lambda=1$ instead gives $\theta=1$.
\end{lemma}
\begin{proof}
Apply Lemma~\ref{lem:rounded-bilinear} to $\mathcal L=\mathcal M_\lambda$,
then add the approximation error. By \eqref{eq:target-map-norm} and
Lemma~\ref{lem:degeneration-conditioning},
$\|\mathcal M_\lambda\|\le1+\beta(\lambda)$; substituting this bound gives
\eqref{eq:conditioned-error}.
The smallest degree among the remaining mixed products is one, with a
positive coefficient in $\beta$, so $\beta'(0)>0$ and
$\beta(\lambda)=\beta'(0)\lambda+o(\lambda)$.
Equation~\eqref{eq:exact-small-parameter-sigma} gives
$\sigma(\Xi(\lambda))=\sigma_0\lambda^{-h}$ for sufficiently small
positive $\lambda$, with $\sigma_0>0$. If $h=0$, the nonnegative
differences $a_t-\min_s a_s$, $b_t-\min_s b_s$, and $c_t-\min_s c_s$ sum to zero
for every $t$, making every mixed degree zero, contrary to normalization.
Thus $h>0$, and the leading error terms are
\begin{equation}
\beta'(0)\lambda+3\epsilon\sigma_0\lambda^{-h}.
\label{eq:leading-error-balance}
\end{equation}
Equating them gives \eqref{eq:parameter-choice}, before clipping at one.
As $\epsilon\to0$, the choice tends to zero, so the small-parameter
norm identity applies and these terms are positive constants times
$\epsilon^{1/(h+1)}$. The contribution independent of the factorization
is $O(\epsilon)$ and hence of smaller order; the higher-order
terms in \eqref{eq:conditioned-error} are also smaller. The classical
bound satisfies
$e_c(\epsilon,x,y)\sim
3(1+\sigma_c)\|x\|_2\|y\|_2\epsilon/\|\mathcal M(x,y)\|_2$.
Consequently,
\begin{equation}
\log\frac1{e(\epsilon,x,y)}
 =\frac1{1+h}\log\frac1\epsilon+O(1),\qquad
\log\frac1{e_c(\epsilon,x,y)}
 =\log\frac1\epsilon+O(1).
\label{eq:accuracy-log-asymptotics}
\end{equation}
Their ratio tends to $1/(1+h)$.
For an exact construction, $\beta=0$ and $\lambda=1$ make both bounds
positive constants times $\epsilon$, so the limiting fraction is one.
\end{proof}

For normalized heights $h\le1,2,3$, respectively, the limiting retained
fraction is at least $1/2,1/3,1/4$. 
The residual coefficients and replicated-coordinate factors in $\beta$
also affect the finite-precision bound. At a specified precision, use the full
functions $\beta(\lambda_{\rm bal}(\epsilon))$ and
$\sigma(\Xi(\lambda_{\rm bal}(\epsilon)))$ in
Definition~\ref{def:retained-accuracy}, evaluating $\sigma$ from the full
maxima in \eqref{eq:spectral-factor-norms}.

For a fixed tripartition set and correction family $\mathcal Q$, the
remaining triples $\mathcal J$ are determined by
Definition~\ref{def:aggregate-index-sets}. Choosing $\lambda$ balances
error for one assignment of exponents; choosing the relative exponents
can increase the limiting fraction by reducing $h$. The following linear
program determines the best such fraction for this fixed correction family.

\begin{proposition}[Choice of parameter exponents]
\label{prop:shifted-normalization}
Fix $\mathcal S$ and the evaluated corrections $\mathcal Q$, and suppose
the remaining triples $\mathcal J$ admit a grading. The largest limiting
fraction given by the normalized gradings is $1/(h_*+1)$, where
\begin{equation}
\begin{aligned}
h_* = \min\ &h\\
\text{subject to }&a_t,b_t\ge0,\quad c_t\ge-h,\quad a_t+b_t+c_t=0
 &&(1\le t\le m),\\
&a_i+b_j+c_k\ge1 &&((i,j,k)\in\mathcal J).
\end{aligned}
\label{eq:height-lp}
\end{equation}
The minimum is attained by rational exponents. At an optimum,
$\min_t a_t=\min_t b_t=0$ and $\min_t c_t=-h_*$.
\end{proposition}
\begin{proof}
Subtracting $\min_t a_t$ from every $a_t$ and adding it to every $c_t$
preserves all product degrees and the spectral-norm product. The same
holds for $b_t$. Thus we may impose $a_t,b_t\ge0$ with minima zero; then
\eqref{eq:grading-height} reduces to $h=-\min_t c_t$. The constraints on
$\mathcal J$ ensure that every remaining degree is at least one.

At an optimum, a positive minimum of $a_t$ or $b_t$ could be removed as
above to decrease $h$, and $\min_t c_t>-h$ would also allow $h$ to decrease.
When $\mathcal J\ne\emptyset$, the smallest degree among the remaining mixed products must be one:
otherwise dividing all exponents and $h$ by that degree lowers the
objective. Lemma~\ref{lem:retained-accuracy} therefore identifies the
largest limiting fraction with $1/(h_*+1)$. The linear program has rational
coefficients and a bounded objective, so a feasible instance has a
rational optimum. If $\mathcal J=\emptyset$, zero exponents give $h_*=0$.
\end{proof}

\section{Computational search and comparisons}
\label{sec:balanced-constructions}
\label{sec:frontier}

The numerical comparisons evaluate multiplication counts as the matrix
dimension, number of independent products, and required fraction of retained
accuracy vary. We compare explicit corrections with padding and weighting,
both for the same tripartition set and for the number of products that can be
accommodated. Table~\ref{tab:construction-parameters} summarizes the selected
constructions. Appendices~\ref{apx:construction-bounds} and~\ref{apx:unequal-cw}
give additional bounds, and Appendix~\ref{apx:construction-certificates}
indexes the complete construction data and their figure appearances.

\subsection{Parameters and search}
\label{sec:plot-models}
\label{sec:fixed-dimension-comparison}

A balanced family in $B_{N,N,N}$ gives $p=|\mathcal S|$ independent
products of dimension $q=n^N$, with $d=3N$ aggregate indices.
The local dimension $n$ determines the embedding overhead, while the
family determines the correction counts and feasible parameter exponents.
We evaluate the selected multiplication budget $r=\rho_{\mathcal S,\mathcal Q}(n)$
from \eqref{eq:partial-correction-budget}, using all corrections in $\mathcal P(\mathcal S)$ for exact constructions. For a reference algorithm with count
$r_{\rm ref}$, its reduction factor is $r_{\rm ref}/r$.
In particular, relative to classical multiplication we use
\begin{equation}
\xi(p,q,r)=\frac{pq^3}{r},\qquad
\omega_{\rm eff}(p,q,r)=\frac{\log(r/p)}{\log q}.
\label{eq:intro-throughput}
\end{equation}
The plotted quantity is $\xi$, with one curve for each family.

We compare exact constructions, approximations with limiting retained
fraction $\theta\ge1/2,1/3,1/4$ of
Definition~\ref{def:retained-accuracy}, and unrestricted border rank.
The value of $\theta$ is computed from the smallest remaining degree
and normalized height; Proposition~\ref{prop:shifted-normalization} gives
the linear program for fixed tripartitions and corrections. At fixed
working precision, \eqref{eq:finite-precision-fraction} compares the
full approximation and round-off bound with the classical
reference. We compute $\sigma$ of Definition~\ref{def:spectral-product}
from \eqref{eq:spectral-factor-norms} and use the leading-term
parameter rule~\eqref{eq:parameter-choice}.
For $N=2,3,4$, we evaluate every integer $n\ge2$ with $n^N\le16000$.
Markers denote the dimensions $q=n^N$ at these integer values of $n$;
connecting segments join the samples. The largest sampled dimensions are
$15876$, $15625$, and $14641$, respectively.

For a prior square base of dimension $b$ and multiplication count $r_b$,
we use the padding-free recursive model
\begin{equation}
\xi\!\left(p,q,pq^{\log_b r_b}\right)=q^{3-\log_b r_b}.
\label{eq:prior-count-model}
\end{equation}
It agrees with $k$ tensor levels at $q=b^k$ and provides an optimistic
comparison at other dimensions. The base algorithms and their sources are listed in
Table~\ref{tab:plot-baselines}. The model applies the base separately to
each of the $p$ products, giving the reference count $pq^{\log_b r_b}$.

\paragraph{Search algorithm.}
We searched for constructions using several algorithms: randomized greedy
selection, replacement of tripartitions, search over discretized exponents,
and linear and mixed-integer programming. We retain the feasible
constructions found by these algorithms. The following describes the joint
search over tripartition sets and exponents used for the comparisons.

For balanced families, the candidates are the
$(3N)!/(N!)^3$ tripartitions in $B_{N,N,N}$. A local update removes a random
subset of the current family and scans the candidates in random order to
refill it. A candidate is accepted only when the three part projections remain
injective and a feasible pair of grid exponents $(a_t,b_t)$ is available,
with $c_t=-a_t-b_t$. The procedure periodically starts from the empty family
and occasionally accepts a lower-scoring family to explore other choices.
Stored constructions also supply starting families.

For a nonexact construction and a requested retained fraction
$0<\vartheta<1$, we use $\mathcal Q=\mathcal Q_{\le0}$ from
Definition~\ref{def:degree-corrections} and maximize
$pn^{3N}/\rho_{\mathcal S,\mathcal Q}(n)$. For each candidate, we recompute
$c_{\mathcal Q}(\mathcal R_{ijk})$ for every mixed product not annihilated
by centering. Let $g$ be the smallest remaining degree and
$h_0=-\min_t a_t-\min_t b_t-\min_t c_t$. A candidate with remaining mixed
products is eligible when
\begin{equation}
\theta=\frac{g}{g+h_0}\ge\vartheta.
\label{eq:search-retained-fraction}
\end{equation}
Updating an exponent pair changes the required correction sets, each
of which is counted only once. With the tripartitions fixed, linear and
mixed-integer programming provide further exponent choices. After each
such refinement, we recompute $\mathcal Q_{\le0}$ and check the attained
fraction before accepting the candidate.

The weighted search instead maximizes $p$, since one aggregate always costs
$(n+2)^{3N}$. Weighting eliminates all replicated products, so only
nonreplicated mixed products constrain its grading. Each tripartition set is
evaluated with both embeddings, and each construction remains eligible at all
weaker accuracy requirements. For incomplete tripartitions, the code also
supports local replacement with objective
$\sum_t n^{\omega|G_t|/3}$ at trial values of $(n,\omega)$, and mixed-integer
selection from an explicitly enumerated candidate set.

\paragraph{Exploration budget.}
The systematic weighted exploration covers $N=1,2,3,4$, exactness, and target
fractions $1/2,1/3,1/4$, together with an unrestricted class explored at height
bounds $6$ and $12$. Each height bound uses two grid spacings, $1$ and $1/2$, and two
random restarts, with $160$ local updates per start. For corrected constructions,
we use $N=2,3,4$, the same two grids, and target fractions
$1/2,1/3,1/4,1/7,1/13$; the last two explore the unrestricted comparison.
The local dimensions used in the objective are $(3,12,126)$, $(3,8,25)$, and $(3,7,11)$,
respectively. A standard start has $80$ local updates; additional optimization of selected
families uses up to $600$.

\paragraph{Verification and selection.}
Every saved construction uses exactly represented rational exponents and is checked
against all centered mixed products. An independent set-based enumeration
checks the selected tripartitions and inclusion--exclusion coefficients. The
factor-generation tests compare the resulting trilinear forms with the target
and the predicted residual; small cases additionally admit a full dense
coefficient-tensor comparison. The figure checks independently recompute
multiplication counts, attained accuracy fractions, and plotted coordinates from the
saved tripartitions, exponents, and corrections. Reproducing the comparisons
therefore does not require repeating the search. All saved feasible constructions
remain available, and the selected ones are evaluated at every integer local
dimension in the plotted range.

If a grading suppresses every non-diagonal product surviving centering,
including replicated products, no corrections are needed:
\begin{equation}
\underline R\left(\bigoplus_{t=1}^{m}
\left\langle n^{|V_t|},n^{|W_t|},n^{|U_t|}\right\rangle\right)
\le(n+1)^d.
\label{eq:correction-free-rank}
\end{equation}

\subsection{Construction parameters}
\label{sec:construction-parameters}

For each family below, let $r_p(n)$ denote its
multiplication budget for all $p$ products of dimension $q=n^N$.
For an exact construction this is an upper bound on
$R(\langle q,q,q\rangle^{\oplus p})$; for a parameter-dependent
construction it bounds the border rank and counts the bilinear products
at each nonzero parameter value.
In \eqref{eq:partial-correction-budget}, $(n+1)^{3N}$ counts the main aggregate and
$\kappa_s (n+1)^{3N-s}$ counts the corrections omitting $s$ coordinates.
Thus singleton, pair, triple, and quadruple corrections have $s=1,2,3,4$,
respectively (Definition~\ref{def:centered-aggregate}). The correction counts,
normalized heights, and limiting fractions of retained accuracy are
summarized in Table~\ref{tab:construction-parameters}. Appendix~\ref{apx:construction-certificates}
indexes the complete construction data and identifies the figures in which
each construction appears; the identifiers below are the file stems in that index.

\begin{table}[!htbp]
\centering\small
\begin{tabular}{rrllcc}
\toprule
$N$ & $p$ & type & $(\kappa_1,\kappa_2,\ldots)$ & $h$ & $\theta$\\
\midrule
2 & 6 & exact & $(4)$ & $0$ & $1$\\
2 & 6 & border rank & $\emptyset$ & $3$ & $\frac{1}{4}$\\
2 & 7 & border rank & $\emptyset$ & $15$ & $\frac{1}{16}$\\
2 & 8 & border rank & $(5)$ & $1$ & $\frac{1}{2}$\\
2 & 9 & border rank & $(6)$ & $2$ & $\frac{1}{3}$\\
2 & 9 & border rank & $(5)$ & $\frac{8}{3}$ & $\frac{3}{11}$\\
2 & 10 & border rank & $(4)$ & $7$ & $\frac{1}{8}$\\
\addlinespace[2pt]
3 & 14 & exact & $\emptyset$ & $0$ & $1$\\
3 & 28 & exact & $(9,35,4)$ & $0$ & $1$\\
3 & 32 & border rank & $(9,27)$ & $1$ & $\frac{1}{2}$\\
3 & 35 & border rank & $(9,27)$ & $2$ & $\frac{1}{3}$\\
3 & 36 & border rank & $(9,27)$ & $3$ & $\frac{1}{4}$\\
3 & 37 & border rank & $(9,29)$ & $3$ & $\frac{1}{4}$\\
3 & 39 & border rank & $(9,26)$ & $12$ & $\frac{1}{13}$\\
\addlinespace[2pt]
4 & 38 & exact & $\emptyset$ & $0$ & $1$\\
4 & 105 & exact & $(12,66,124,3)$ & $0$ & $1$\\
4 & 125 & border rank & $(12,66,143,5)$ & $1$ & $\frac{1}{2}$\\
4 & 133 & border rank & $(12,66,130,6)$ & $2$ & $\frac{1}{3}$\\
4 & 141 & border rank & $(12,66,144,6)$ & $3$ & $\frac{1}{4}$\\
4 & 150 & border rank & $(12,66,131,6)$ & $11$ & $\frac{1}{12}$\\
\bottomrule
\end{tabular}
\caption{Parameters of the square-product constructions used in the comparisons.
Each computes $p$ products of dimension $q=n^N$.
The entry $\kappa_s$ counts correction aggregates that omit $s$ coordinates;
$\emptyset$ means that no corrections are used.
The normalized height $h$ and limiting retained fraction $\theta=1/(1+h)$
are defined in Section~\ref{sec:stability}.
The two $N=2,p=9$ rows trade one singleton correction for a weaker accuracy requirement.
The data files and figure appearances for every row are indexed in Appendix~\ref{apx:construction-certificates}.}
\label{tab:construction-parameters}
\end{table}

\paragraph{Fourteen exact products.}
The $N=3$, $p=14$ family has fourteen distinct labels in each variable
space. Centering annihilates every non-diagonal choice, so one aggregate
on nine indices suffices:
\begin{equation}
 r_{14}(n)=(n+1)^9.
\end{equation}
There are no corrections or parameter exponents. This gives fourteen independent matrix--matrix products
with the lowest embedding overhead among the principal $N=3$ families,
which explains its advantage at intermediate dimensions.
Construction data: Appendix~\ref{apx:construction-certificates}, \texttt{ex14}.

\paragraph{Twenty-eight exact products.}
A second $N=3$ family doubles the number of independent matrix--matrix products by admitting replicated
mixed terms. Nine singleton, thirty-five pair, and four triple
corrections remove these terms, giving
\begin{equation}
 r_{28}(n)=(n+1)^9+9(n+1)^8+35(n+1)^7+4(n+1)^6.
\label{eq:exact-28}
\end{equation}
No compatible non-diagonal product remains. The correction terms grow
more slowly than the main aggregate, so the larger support eventually
costs less than two fourteen-product aggregates. In fact,
$r_{28}(n)<2r_{14}(n)$ for integer $n\ge11$.
Thus permitting replicated terms can reduce the count per output even
when the algorithm is required to be exact.
Construction data: Appendix~\ref{apx:construction-certificates}, \texttt{ex28}.

\paragraph{Thirty-two products with half retained precision.}
The $N=3$, $p=32$ family has budget
\begin{equation}
 r_{32}(n)=(n+1)^9+9(n+1)^8+27(n+1)^7.
\label{eq:d9-corr-32}
\end{equation}
Its twenty-seven remaining mixed products, including three replicated
products, are suppressed by a grading with $h=1$, giving $\theta=1/2$.
Compared with the exact twenty-eight-product family, it needs eight
fewer pair corrections and no triple corrections.
Relaxing exactness therefore has two benefits here: more products are
aggregated, and the total multiplication budget is smaller.
Construction data: Appendix~\ref{apx:construction-certificates}, \texttt{br32}.

\paragraph{Thirty-five products with one-third retained precision.}
The $N=3$, $p=35$ family has the same correction budget as the half-accuracy
construction:
\begin{equation}
 r_{35}(n)=(n+1)^9+9(n+1)^8+27(n+1)^7.
\end{equation}
Its eighty-three remaining mixed products, including thirty-six replicated
products, admit normalized exponents of height $2$, giving $\theta=1/3$.
The three additional outputs therefore require no increase in the
multiplication budget. A full-correction variant attains the
slightly larger fraction $8/23$ with thirty-one pair corrections;
its tripartitions and exponents are supplied for that stronger requirement.
Construction data: Appendix~\ref{apx:construction-certificates}, \texttt{br35}
and \texttt{br35full}, respectively.

\paragraph{Thirty-six products with quarter retained precision.}
A further $N=3$ family adds one more product at the same cost:
\begin{equation}
 r_{36}(n)=(n+1)^9+9(n+1)^8+27(n+1)^7.
\end{equation}
Here $h=3$ gives $\theta=1/4$. Of its 171 remaining mixed products,
116 are replicated. Thus the half-, third-, and quarter-accuracy choices
have distinct supports and gradings but identical rank polynomials,
aggregating $32$, $35$, and $36$ products, respectively.
Construction data: Appendix~\ref{apx:construction-certificates}, \texttt{br36}.

\paragraph{One hundred and five exact products.}
\label{sec:d12-search}
For $N=4$, the corrected exact family has
\begin{equation}
 r_{105}(n)=(n+1)^{12}+12(n+1)^{11}+66(n+1)^{10}+124(n+1)^9+3(n+1)^8.
\end{equation}
All singleton and pair corrections occur, but only $124$ of the $220$
triples and three of the $495$ quadruples are needed.
The correction-free family at the same order aggregates $38$ products
with budget $(n+1)^{12}$. Thus the corrected family provides almost three
times as many outputs, but must first overcome a substantial correction
cost.
At a fixed matrix dimension, increasing $N$ also reduces the local
dimension $n$; this delays the advantage over the $N=3$ families.
Construction data: Appendix~\ref{apx:construction-certificates}, \texttt{ex105} and \texttt{ex38}.

\paragraph{One hundred and forty-one products with quarter retained precision.}
The largest graded family at quarter retained precision has $N=4$, $p=141$, and budget
\begin{equation}
 r_{141}(n)=(n+1)^{12}+12(n+1)^{11}+66(n+1)^{10}+144(n+1)^9+6(n+1)^8.
\end{equation}
A grading of height three suppresses its $497$ remaining mixed
products, of which $109$ are replicated, giving $\theta=1/4$.
The $125$- and $133$-product families satisfy the stronger bounds
$h\le1$ and $h\le2$, respectively.
All four large families share the leading aggregate and all singleton
and pair corrections. Their budgets differ only in the $(n+1)^9$ and $(n+1)^8$
terms, while their output counts range from $105$ to $141$.
Consequently, relaxing the retained-precision requirement increases the
number of products at nearly the same total multiplication cost.
Construction data: Appendix~\ref{apx:construction-certificates}, \texttt{br125}, \texttt{br133}, and \texttt{br141}.

\paragraph{Other accuracy requirements.}
At quarter retained precision, a $37$-product family improves the count
per product of the $36$-product family for $n\ge3$. Without the quarter-accuracy
requirement, $39$ products are possible with fewer pair corrections. Their budgets are
\begin{align}
r_{37}(n)&=(n+1)^9+9(n+1)^8+29(n+1)^7,
&\theta&=\frac14,\\
r_{39}(n)&=(n+1)^9+9(n+1)^8+26(n+1)^7,
&\theta&=\frac1{13}.
\label{eq:new-d9-budgets}
\end{align}
A twelfth-order construction without an accuracy requirement gives
\begin{equation}
\begin{aligned}
r_{150}(n)&=(n+1)^{12}+12(n+1)^{11}+66(n+1)^{10} +131(n+1)^9+6(n+1)^8,\qquad \theta=\frac1{12}.
\end{aligned}
\label{eq:new-d12-budget}
\end{equation}
At order six, the selected unrestricted ten-product construction uses
only four singleton corrections, with $h=7$ and $\theta=1/8$.
A nine-product variant uses five singleton corrections and attains
$\theta=3/11>1/4$; the version attaining $1/3$ uses six.
These choices target small multiplication counts at the stated limiting
accuracy requirements, not the full finite-precision error bound.
Construction data: Appendix~\ref{apx:construction-certificates}, \texttt{br37}, \texttt{ub9}, \texttt{ub12}, \texttt{br10}, \texttt{br9q}, and \texttt{br9}.

\paragraph{Accuracy at fixed precision.}
The normalized height $h$ determines the growth $\lambda^{-h}$ of
$\sigma$, while the leading
approximation coefficient $\beta'(0)$ of
Definition~\ref{def:approximation-bound} includes only
mixed products of smallest positive degree, weighted by
$|c_{\mathcal Q}(\mathcal R_{ijk})|(n+1)^{|\mathcal R_{ijk}|}$.
Table~\ref{tab:sigma-approximate} gives these quantities and the
normalized leading coefficient $\sigma_0/\sigma_c$
(Definition~\ref{def:leading-sigma}) for the approximate families. The correction sets also enter the factor norms, so both
$\sigma_0$ and $\beta'(0)$ are recomputed for each selected family.
The limiting fraction is determined by $h$; the remaining mixed products
and factor norms additionally affect the parameter choice and accuracy
at fixed working precision. Reducing the count at a prescribed limiting
fraction does not minimize the error bound at fixed precision: omitting
replicated corrections can substantially increase $\beta'(0)$. Full-correction
variants therefore remain relevant when these constants are decisive.
For exact families, the common bound~\eqref{eq:exact-sigma-ratio}
shows that the extra amplification decreases to one as $n$ increases.

The parameter rule balances the first approximation term against the
leading intermediate-rounding term. This minimizes their sum when $h=1$,
but need not do so for other heights. The higher-order approximation terms
can be included by minimizing the full bound~\eqref{eq:conditioned-error}
with respect to $\lambda$.
This changes the finite-precision parameter choice while retaining the
same limiting fractions displayed in the count comparisons.

Neither the limiting retained fraction nor balancing, or even minimizing,
these error bounds necessarily selects the most stable or fastest construction
at a fixed working precision. The bounds are conservative, and multiplication
counts omit encoding, decoding, and data movement. Searches intended for an
implementation should therefore account for error constants and higher-order
approximation terms, the evaluation order, and measured execution costs at the precision
and dimensions of interest.

\subsection{Centering, corrections, and weighting}
\label{sec:embedding-comparison}

We separate the cost of evaluating a fixed support from the number of
products that a weighted aggregate can accommodate. For the first comparison,
the centered, padded, and weighted embedding of
Section~\ref{sec:padding-weighting} uses $(n+2)^{3N}$ multiplications.
The construction with selected corrections uses
$\rho_{\mathcal S,\mathcal Q}(n)$ on the \emph{same support}, with the
\emph{same grading}. Figure~\ref{fig:embedding-savings} shows the fractional
saving, normalized to the weighted count:
\begin{equation}
1-\frac{\rho_{\mathcal S,\mathcal Q}(n)}{(n+2)^{3N}}.
\label{eq:weighted-paired-reduction}
\end{equation}
For a graded support, weighting eliminates all replicated terms, whereas
explicit corrections need eliminate only those required by the prescribed
accuracy. The weighted version therefore meets at least the same limiting
accuracy requirement; equality of finite-precision errors is not assumed.
Appendix~\ref{apx:weighted-comparison} gives the weighted specialization.

The exact $p=14$ and $p=38$ families are outliers in
Figure~\ref{fig:embedding-savings}: centering alone suffices, with no
corrections. Their savings are $1-((n+1)/(n+2))^{3N}$, approximately
$28.8\%$ at $(N,n)=(3,25)$ and $61.7\%$ at $(4,11)$, respectively.
The larger saving at $N=4$ reflects both the higher order and the smaller
local dimension; the correction-free $N=2$ families have much smaller
savings at $n=126$. Most remaining families require corrections whose
costs offset much of the reduction from dimension $n+2$ to $n+1$.
These bars compare the same support with and without weighting; they do
not establish the lowest count per product, since families with
corrections can accommodate more products.

\begin{figure}[!htbp]
\centering
\includegraphics[width=\textwidth]{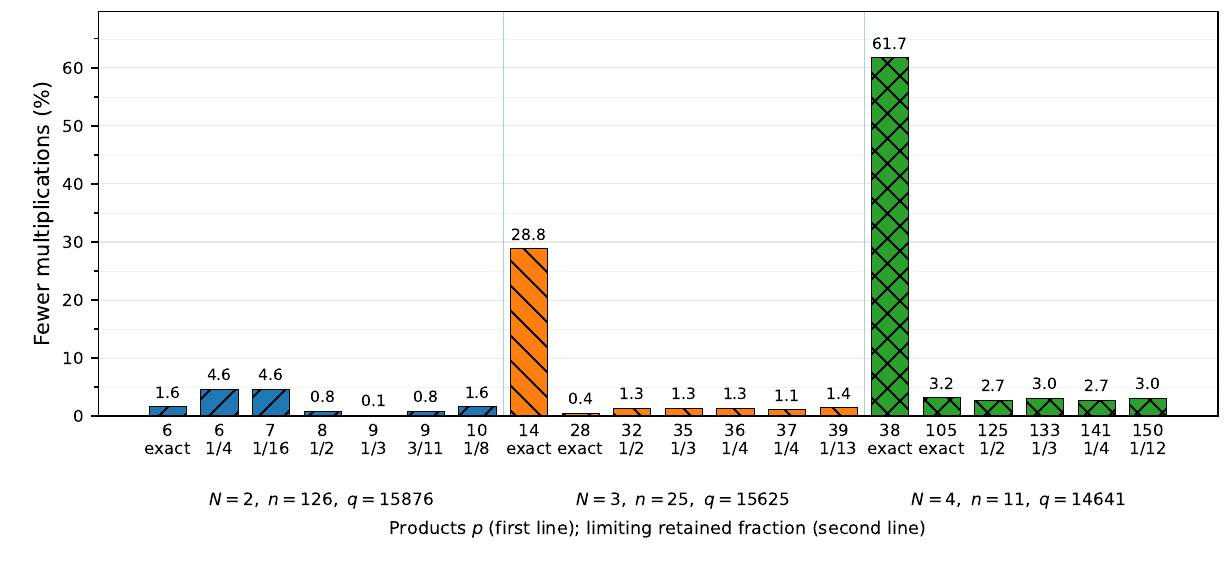}
\caption{Percentage fewer multiplications than centering, padding, and
weighting the identical support with the identical grading. Each bar is
$100$ times \eqref{eq:weighted-paired-reduction}; zero means no saving.
The first line below a bar gives $p$, and the second gives exactness or the
attained limiting retained fraction. Within each order, $n$ is the largest
integer with $q=n^N\le16000$. The comparison changes neither the requested
products nor the matrix dimension, and involves no composition of different
constructions. These are multiplication counts, not execution times.}
\label{fig:embedding-savings}
\end{figure}

For the second comparison, Figure~\ref{fig:weighted-capacity} shows how
many independent products one weighted aggregate can accommodate at fixed
$N$, with multiplication count $(n+2)^{3N}$. Without parameter
scaling its support must be conflict-free. Allowing a border-rank grading
can admit more products, subject to the requested limiting retained
fraction. The capacity search therefore maximizes $p$, independently of
$n$, rather than minimizing the correction-based count per product.
Only mixed triples surviving weighting constrain this grading. These
weighted constructions need not use the same tripartitions or exponents as
the centered constructions in Section~\ref{sec:construction-parameters};
Table~\ref{tab:weighted-construction-index} identifies the complete data
for every level in Figure~\ref{fig:weighted-capacity}.

\begin{figure}[!htbp]
\centering
\includegraphics[width=\textwidth]{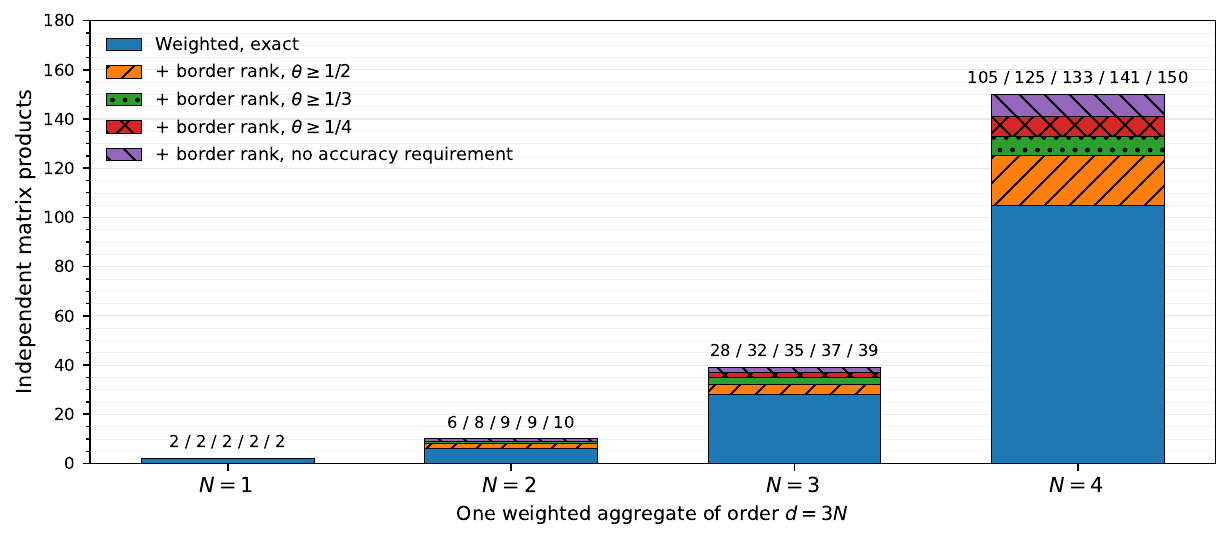}
\caption{Independent products accommodated by one centered, padded, and
weighted aggregate. The bottom segment is an exact conflict-free family. Successive
layers show the additional products admitted by allowing border-rank
gradings with retained fraction at least $1/2$, $1/3$, $1/4$, and no
accuracy requirement. Heights are cumulative product counts, not sums of
the capacities of separate constructions; the support may change between
levels. The five numbers above each bar give these cumulative counts in
order. ``No accuracy requirement'' means unrestricted border rank, not
exact computation or an assertion that the attained fraction is zero.
The cumulative bar heights are attained product counts, not upper bounds
or claims of optimality.}
\label{fig:weighted-capacity}
\end{figure}

For $N=4$, Figure~\ref{fig:weighted-capacity} shows $105$ exact products,
increasing to $125$, $133$, and $141$ at limiting retained fractions at
least $1/2$, $1/3$, and $1/4$, and to $150$ without an accuracy requirement.
These gains come from allowing parameter scaling to suppress mixed products
that an exact weighted construction must exclude.
All stronger-accuracy candidates remain eligible at weaker requirements,
so the cumulative capacities cannot decrease. Supports at successive levels
need not be nested. Search and verification use the procedures and budgets
in Section~\ref{sec:plot-models}.

\subsection{Comparison with prior matrix multiplication algorithms}
\label{sec:literature-comparison}

We now compare our finite constructions with selected prior single-product algorithms.
The baseline algorithms first reduce the multiplication count of a single
product and are then applied separately to each of the $p$ products. Strassen replaces eight $2\times2$ block products
by seven~\cite{strassen1969gaussian}; the exact Smirnov square base
combines three cyclic rectangular rank-$40$ constructions~\cite{smirnov2013bilinear}.
Pan's methods aggregate trilinear terms and cancel the unwanted
products~\cite{pan1982implicit}. Schwartz--Zwecher and LITA refine this
approach for a single square product~\cite{schwartz2025faster,khoruzhii2026lita}.
LITA has the lowest exact exponent among the bases in
Table~\ref{tab:plot-baselines}. Comparing against its padding-free
model tests whether aggregating independent matrix--matrix products improves on this
single-product recursion without benefiting from its padding overhead.

\begin{figure}[!htbp]
\centering
\includegraphics[width=\textwidth]{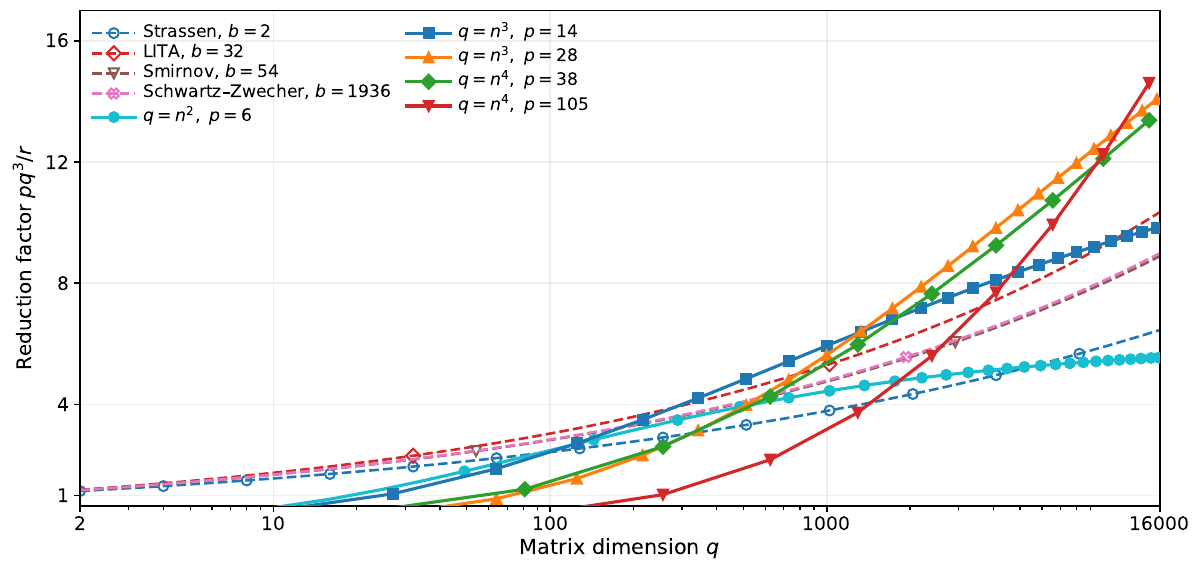}
\caption{Multiplication-count reduction factors for exact algorithms.
Each solid curve computes $p$ independent products of dimension $q=n^N$;
dashed curves apply the indicated prior base of dimension $b$ recursively.
Markers indicate dimensions $q=n^N$ with integer $n$, with the interpolation conventions of
Section~\ref{sec:plot-models}. The $p=14$ and $p=38$ families use centering
without corrections. The $p=6,28,105$ families include the corrections
counted in Table~\ref{tab:construction-parameters}. Larger corrected
families become advantageous as their correction terms decrease relative
to the leading aggregate.}
\label{fig:equalsize}
\label{fig:equalsize-exact}
\end{figure}

Figure~\ref{fig:equalsize-exact} shows that the largest support is not
uniformly preferable. For a correction-free family the reduction factor
is $p(n/(n+1))^{3N}$. At fixed $q$, increasing $N$ decreases $n$ and
increases this embedding overhead; corrections add a further cost.
The fourteen-product family is therefore competitive at smaller dimensions, while
the corrected twenty-eight-product family becomes preferable as its
lower-order terms decrease relative to the aggregate.
The same tradeoff delays the $105$-product family: it remains below the
$28$-product family at $q=4096$, but exceeds it at
$q=14641$.
Larger tensor powers become useful when the additional
outputs offset both embedding and correction costs.

The LITA, Strassen, and classical columns of
Table~\ref{tab:intro-speedups} distinguish aggregation beyond a strong
single-product algorithm from the total reduction relative to standard
multiplication. At $q=14641$, for exactly $105$ requested products, the smallest count
using the $N\le3$ families and the LITA model requires four
evaluations of the $28$-product construction and discards seven outputs. The $105$-product construction
reduces this count by a factor of approximately $1.123$, compared with
approximately $1.053$ for the count per output. Matching the number of requested independent matrix--matrix products can
therefore strengthen a relatively small per-product advantage.

\begin{figure}[!htbp]
\centering
\includegraphics[width=\textwidth]{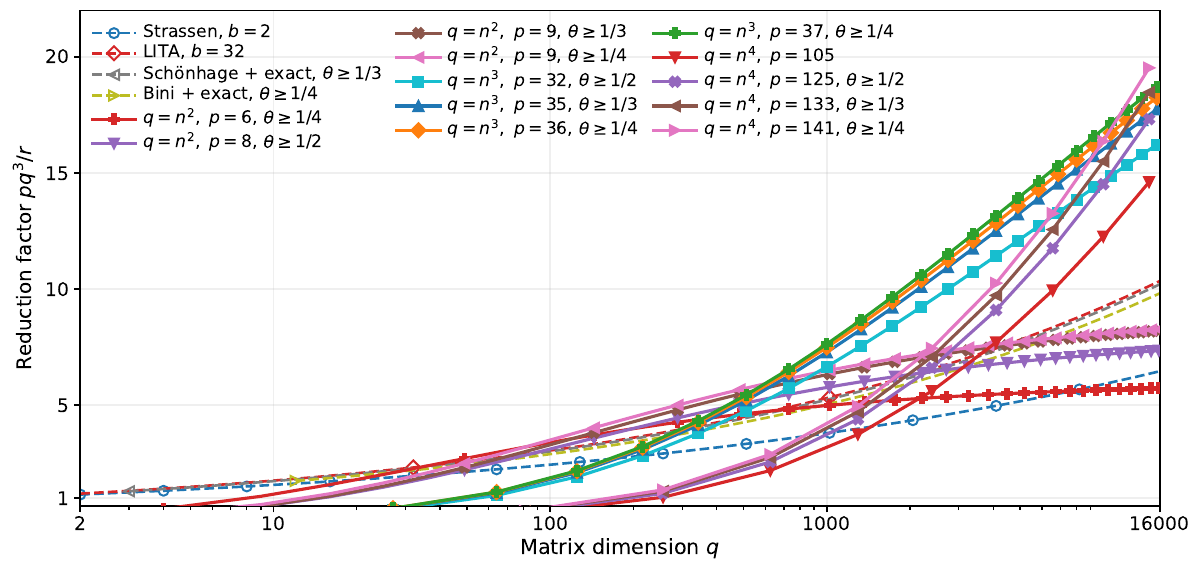}
\caption{Multiplication-count reduction factors with limiting fraction
of retained precision at least $1/4$. Labels distinguish the half-,
third-, and quarter-precision bounds; the exact $p=105$ family has
$\theta=1$. Solid curves are the centered constructions of this paper.
Dashed curves are exact prior algorithms or a single Sch\"onhage or cyclic
Bini approximation followed by exact algorithms for the remaining subproblems. These latter two constructions
have $\theta=1/3$ and $1/4$, respectively, as derived in
Appendix~\ref{sec:prior-stability}. The limiting fraction is relative to the classical error reference;
Section~\ref{sec:parameter-choice} gives the parameter rule and the
finite-precision comparison, including the approximation coefficients
and factor norms.
All counts use the conventions of Section~\ref{sec:plot-models}.}
\label{fig:equalsize-bounded}
\end{figure}

The accuracy-constrained comparison in Figure~\ref{fig:equalsize-bounded}
shows how parameter scaling changes the tradeoff. The smaller
$q=n^2$ constructions become useful before the larger families can
amortize their embedding costs. At larger dimensions, the $p=32,35,37$
and then $p=125,133,141$ families give the largest reductions at their
respective precision thresholds. As the polynomials in Section~\ref{sec:construction-parameters}
show, the benefit comes mainly from additional outputs rather than a
smaller aggregate: more mixed products can be suppressed with little
change to the correction budget. The half-, third-, and quarter-precision
choices therefore provide distinct product counts at comparable total cost.

In Figure~\ref{fig:equalsize-bounded}, the Bini and Sch\"onhage constructions use
one approximate step followed by exact subproblems to meet the accuracy
requirement. Neither improves on LITA alone under the count model used here.
Repeated approximate steps can reduce the count further but lose more
accuracy (Appendix~\ref{sec:prior-stability}). The new constructions instead
reduce the count per product by aggregating more independent matrix--matrix
products at a prescribed limiting retained fraction.

\begin{figure}[!htbp]
\centering
\includegraphics[width=\textwidth]{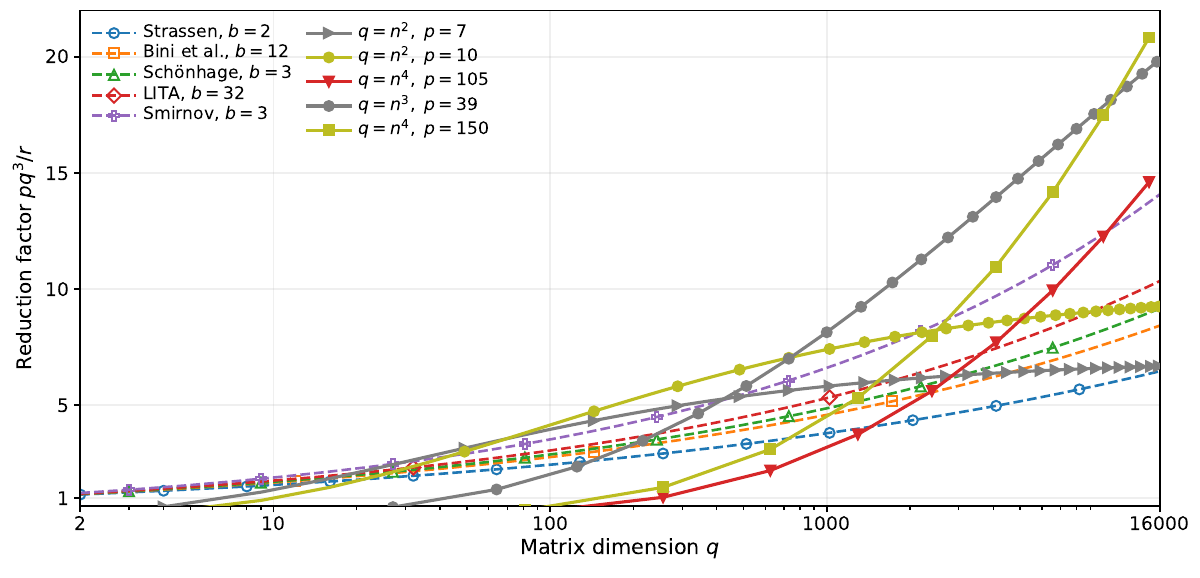}
\caption{Finite-size border-rank reduction factors without a precision
constraint. Prior algorithms are dashed and this paper's constructions
are solid. The $q=n^2$, $p=7,10$ families have limiting retained fractions
$1/16$ and $1/8$, respectively. The $p=39,150$ families have
limiting retained fractions $1/13$ and $1/12$. Smirnov's dashed curve
uses the rank-$20$ base of dimension three at unrestricted recursive
depth. The $p=105$ curve is exact. Construction parameters and multiplication-count comparisons are defined
in Section~\ref{sec:plot-models}, and Appendix~\ref{sec:prior-stability}
analyzes the parameter dependence of the prior approximations.}
\label{fig:equalsize-border}
\end{figure}

Without a precision constraint, the smaller supports in
Figure~\ref{fig:equalsize-border} can use much larger parameter exponents.
The seven-product family suppresses all remaining mixed products without
corrections; the ten-product family permits more outputs with a small
correction budget. Their improved counts at small dimensions come at the cost of a large
limiting loss of accuracy. This tradeoff is less favorable
at larger dimensions, where the $39$- and $150$-product families
overtake them at their respective accuracy levels.

Smirnov's rank-$20$ base gives the strongest prior unrestricted curve shown.
Its direct realization has limiting fraction $1/7$ for one level and
$1/(1+6k)$ for $k$ levels. By comparison, the $150$-product family reduces
this formal multiplication-count estimate by a factor of approximately $1.516$ at
$q=14641$, with $\theta=1/12$. At the stronger requirement
$\theta\ge1/4$, the $141$-product family gives a factor of approximately $1.421$.
Thus the new constructions appear promising in comparison with prior works in the multiplication count achieved both by exact constructions and border rank constructions under various accuracy constraints.

\paragraph{Unequal shapes.}
\label{sec:extensions}
For complete and incomplete tripartitions with unequal part sizes,
Tables~\ref{tab:mixed-complete} and~\ref{tab:bigcw} retain the matrix
shapes and product volumes. The complete ten- and thirteen-product
families have limiting retained fractions $1/2$ and $1/3$.
A twenty-two-product family gives an asymptotic-sum bound of approximately
$2.560147$ with limiting retained fraction at least $1/212$.
The incomplete $d=6,9$ families give bounds near $2.636$ and $2.639$.
These values follow by substituting the complete shape lists and rank
budgets into the asymptotic sum theorem.

\FloatBarrier
\begin{revision}
\section{Conclusion}
\label{sec:conclusion}

We provide explicit aggregation formulas for the simultaneous matrix products
obtained from fixed-shape restrictions of powers of the little
Coppersmith--Winograd tensor. Centering and weighting give a direct
realization of the $(n+2)^{3N}$ multiplication budget.
Explicit inclusion--exclusion corrections reduce
this budget, and further border-rank degenerations allow larger sets of
products to be aggregated. This connects the Coppersmith--Winograd approach
to Pan's trilinear aggregation and to bilinear algorithms for symmetric
tensor contraction.

These finite constructions improve on the multiplication counts of the
reference algorithms compared here for selected numbers and dimensions of
independent matrix products. The benefit is not a new best bound on
$\omega$, but explicit constructions for moderate matrix dimensions and
numbers of independent matrix--matrix products. 
Our error analysis quantifies error due to border-rank degenerations and finite-precision evaluation.
This analysis allows us to search for constructions with a prescribed limiting retained fraction of accuracy.

We expect the finite constructions, particularly the border-rank variants,
to be improvable. The searches are heuristic and use modest budgets, and their
objectives emphasize multiplication counts at prescribed limiting accuracy.
A practical formulation should also optimize finite-precision error constants
and implementation costs rather than rely on these asymptotic criteria alone.

An intriguing future direction is to consider the use of the introduced embeddings and the inclusion--exclusion principle to develop aggregations that allow extraction of non-independent matrix--matrix products.
Another open direction is to extend this approach to contractions of
tensor cliques, starting with the $K_4$ contractions studied
in~\cite{christandl2019graphtensors,christandl2019surgery}, and possibly to
other contraction graphs.

Further analysis of numerical stability should account for how the
embeddings and corrections are evaluated, alongside experimental
evaluation of both the exact and approximate algorithms. High-performance
implementations should examine encoding and decoding costs as well as
bilinear multiplication counts. Their communication cost, including data
movement within memory hierarchies and between processors, is another
important subject for study.

\section{Acknowledgments}

The work of Edgar Solomonik was supported by the U.S. Department of Energy,
Office of Science, Office of Advanced Scientific Computing Research, under
Award No.~DE-SC0023483, as part of Sparsitute: A Mathematical Institute for
Sparse Computations in Science and Engineering.

The initial idea of aggregation of tensor contraction was developed when both authors were participating as long-term visitors in the Fall 2025 program on Complexity and Linear Algebra at the Simons Institute for the Theory of Computing, University of California, Berkeley.
After the conclusion of that program, the first author used large language models to further advance and analyze the numerical algorithm, as well as establish connections to prior literature.
Gemini (Google) was initially used and helped establish the connection between Coppersmith and Winograd's construction and aggregation of tensor contractions.
The centering and weighting embeddings were developed with assistance from Gemini and ChatGPT (OpenAI).
The connection to Sierpinski's triangle was picked up from a ChatGPT exploration.
Both ChatGPT and Claude (Anthropic) helped develop and analyze centered corrections as well as border rank degenerations and their numerical stability.
ChatGPT was used to generate the final versions of the computational search algorithm and verification infrastructure, as well as to draft parts of the manuscript and edit text.
The authors iterated on, edited, and ultimately hand-wrote much of the manuscript; they also tested the verification infrastructure and assume responsibility for all contents of the manuscript and software artifacts.

\appendix

\section{Recursive Proof of the Equal-Size Asymptotic Sum Bound}
\label{apx:ast}
Assume that an exact rank-$r$ construction computes a direct sum of $m$
independent products $\langle n,n,n\rangle$, and that $m$ divides $r$.
Apply this construction to $n\times n$ block matrices whose blocks have size
$n^{k-1}\times n^{k-1}$.  The construction reduces the $m$ products of block
matrices to $r$ independent products of blocks.  Partition these $r$ block
products into $r/m$ groups of $m$ and compute each group recursively by the
same construction.  If $R_k$ denotes the resulting rank for $m$ independent
products of dimension $n^k$, then
\begin{equation}
R_1=r,
\qquad
R_k=\frac rm R_{k-1},
\end{equation}
and hence
\begin{equation}
R_k=m\left(\frac rm\right)^k.
\end{equation}
In particular, one product of dimension $n^k$ has rank at most $R_k$.  Taking
$k$th roots and passing to the limit gives
\begin{equation}
n^\omega\leq \frac rm,
\end{equation}
or equivalently $mn^\omega\leq r$.  If $m$ does not divide $r$, the same
limit follows by applying the divisible case to the $\ell$th tensor power
with rank bound $m^\ell\lceil(r/m)^\ell\rceil$, taking $\ell$th roots,
and letting $\ell\to\infty$.

\section{Prior matrix-multiplication algorithms}
\label{apx:baseline-details}
\subsection{Finite-base multiplication counts}
Table~\ref{tab:plot-baselines} lists the bases used in the comparison
catalogue. Their padding-free recursive estimates are defined in
\eqref{eq:prior-count-model}.

\begin{table}[htbp]
\centering\small
\begin{tabular}{lrrl}
\toprule
Construction & base dimension $b$ & count $r_b$ & type\\
\midrule

Strassen~\cite{strassen1969gaussian} & 2 & 7 & exact\\

Laderman~\cite{laderman1976noncommutative} & 3 & 23 & exact\\

Pan (1978)~\cite{pan1978strassen} & 70 & 143\,640 & exact\\

Pan (1982)~\cite{pan1982implicit} & 44 & 36\,133 & exact\\

Smirnov~\cite{smirnov2013bilinear} & 54 & 64\,000 & exact\\

Sedoglavic--Smirnov~\cite{sedoglavic2021five} & 12 & 1\,040 & exact\\

Schwartz--Zwecher~\cite{schwartz2025faster} & 44 & 36\,110 & exact\\

Schwartz--Zwecher~\cite{schwartz2025faster} & 1936 & 1\,303\,676\,064 & exact\\

Dumas--Pernet--Sedoglavic~\cite{dumas2025noncomplex} & 4 & 48 & exact\\

LITA~\cite{khoruzhii2026lita} & 32 & 14\,197 & exact\\

Bini et al.~\cite{bini1979n2} & 12 & 1\,000 & border rank\\

Smirnov~\cite{smirnov2013bilinear} & 3 & 20 & border rank\\

Sch\"onhage~\cite{schonhage1981partial} & 3 & 21 & border rank\\

\bottomrule
\end{tabular}
\caption{Prior square bases in the finite-size comparison.}
\label{tab:plot-baselines}
\end{table}
LITA uses the rational construction at the cited revision~\cite{khoruzhii2026lita}.
The size-$54$ Smirnov base is the cyclic tensor product of three
$\langle3,3,6\rangle$ rank-$40$ decompositions. The size-$12$ Bini base
is the cyclic tensor product of three $\langle3,2,2\rangle$ rank-$10$
approximations. The size-$1936$ Schwartz--Zwecher count is the improved
second-family bound, rather than the square of $36110$.

For rectangular products, Pan's aggregation identity generalizes to
\begin{align}
R\!\left(\langle m,k,n\rangle\oplus\langle k,n,m\rangle\right)
 &\le mkn+mk+kn+nm,\label{eq:intro-pan-pair-rank}\\
\underline R\!\left(\langle m,k,n\rangle\oplus\langle k,n,m\rangle\right)
 &\le mkn+mk+kn.\label{eq:intro-pan-pair-border}
\end{align}
The second bound follows by cyclically relabeling the dimensions in
\cite[Eq.~(3.1)]{pan1982implicit}.
The earlier exact one-product family used below is
\begin{equation}
P_{\rm Pan}(q)=\frac{q^3-4q}{3}+6q^2\qquad(q\text{ even})
\label{eq:pan-parametric-rank}
\end{equation}
\cite{pan1978strassen,pan2017neighbourhood}.

\subsection{Approximate steps with exact subproblems}
\label{sec:prior-stability}
The bilinear representations of Bini, Sch\"onhage, and Smirnov supplied
by Benson and Ballard~\cite{benson2015framework,bensonballard2015factors}
have first-order approximation error in $\lambda$.
If $A(\lambda)=\lambda^{-a}(A_0+O(\lambda))$ with $A_0\ne0$, then
$\|A(\lambda)\|_2=\lambda^{-a}(\|A_0\|_2+O(\lambda))$;
apply this to all three factors. The leading coefficients
$\sigma_0$ are as in Definition~\ref{def:leading-sigma}:
\[
\begin{array}{l|rrrr}
\text{formula} & r & h & \sigma_0 & \theta\\\hline
\text{Bini }\langle3,2,2\rangle &10&1&6+2\sqrt2&1/2\\
\text{Sch\"onhage }\langle3,3,3\rangle &21&2&5\sqrt7&1/3\\
\text{Smirnov }\langle3,3,3\rangle &20&6&2\sqrt{15(6+\sqrt5)}&1/7
\end{array}
\]
where $\sigma(\Xi(\lambda))=\sigma_0\lambda^{-h}(1+O(\lambda))$.
The constructions are due to~\cite{bini1979n2,schonhage1981partial,smirnov2013bilinear}.
In particular, the leading Smirnov factor norms are
$\sqrt{6+\sqrt5}\lambda^{-2}$, $\sqrt5\lambda^{-2}$, and
$\sqrt{12}\lambda^{-2}$, while its first residual coefficient is nonzero.

Spectral norms multiply under tensor products. Consequently, $k$ direct
levels have norm product $\sigma(\Xi(\lambda))^k$ and limiting fraction
$1/(1+kh)$. The cyclic product of three Bini rectangular steps gives
$b=12$, $r_b=1000$, $h=3$, and $\theta=1/4$.
One Sch\"onhage square step gives $\theta=1/3$.
An exact algorithm for the remaining subproblems contributes a $\lambda$-independent norm
product, preserving these limiting fractions. With its count modeled
by LITA's exponent $\omega_L=\log_{32}14197$, one approximate square
step has comparison curve
\begin{equation}
\xi\!\left(1,q,r_b(q/b)^{\omega_L}\right)=\frac{b^3}{r_b}(q/b)^{3-\omega_L},\qquad q\ge b.
\end{equation}
Its ratio to the LITA curve is $b^{\omega_L}/r_b$.
Changes of basis can alter the norm products at fixed
rank~\cite{ballard2016stability,ju2020convolution,dumas2025automated}.

\section{Further bounds for square-product constructions}
\label{apx:construction-bounds}
This appendix collects the weighted comparison, accuracy coefficients,
integer-dimension comparisons, and a shared-variable square-product construction.

\subsection{Weighted and centered comparison counts}
\label{apx:weighted-comparison}
For a complete fixed-shape family $\mathcal S$, apply
\eqref{eq:weighted-inclusion-exclusion} to the scaled blocks. By
\eqref{eq:general-correction-identity}, only the desired products and
$\mathcal B(\mathcal S)$ remain, giving
\begin{equation}
\Phi^{\rm wt}_{\mathcal S}(\lambda)
 =\phi_{\mathcal S}
   +\sum_{(i,j,k)\in\mathcal B(\mathcal S)}
      \lambda^{a_i+b_j+c_k}\phi_{ijk},
\qquad
R(\Phi^{\rm wt}_{\mathcal S}(\lambda))\le(n+2)^d.
\end{equation}
For a conflict-free family this is exact. Otherwise the positive powers
give a degeneration. Since the weighted embedding is independent of
$\lambda$, normalized exponents of height $h$ give approximation error of
order $\lambda$ and $\sigma(\Xi(\lambda))$ of
Definition~\ref{def:spectral-product} of order $\lambda^{-h}$.
For balanced parts, its exact ratio $\sigma(\Xi(\lambda))/\sigma_c$ is the upper-bound expression in
\eqref{eq:normalized-sigma-ratio} times $\lambda^{-h}$.
They retain limiting fraction at least $1/(1+h)$ by the balance in
Section~\ref{sec:stability}.
Thus weighting and explicit correction can be compared on the same
support with the same exponents.

The budget comparison follows from \eqref{eq:correction-vs-padding}.
The normalized saving is given by
\eqref{eq:weighted-paired-reduction}. Since weighting removes every
replicated term, its remaining mixed products form a subset of those
constrained by the corresponding corrected construction. With the same
grading it therefore has no smaller degree gap and no worse limiting
retained fraction. The capacity comparison optimizes over all saved
weighted supports at each requested fraction.

\subsection{Norm coefficients for the approximation families}
\label{apx:sigma-coefficients}
Table~\ref{tab:sigma-approximate} evaluates the leading coefficients in
\eqref{eq:leading-error-balance}. Definition~\ref{def:approximation-bound}
gives $\beta'(0)$ by grouping the remaining mixed products of degree one.
The spectral coefficient uses the blocks with the smallest exponent in
\eqref{eq:spectral-factor-norms}. Exact families satisfy the common
comparison~\eqref{eq:exact-sigma-ratio}.

\begin{table}[H]
\centering\small
\begin{tabular}{rrlcrr}
\toprule
$N$ & $p$ & Construction & $h$ & $\beta'(0)$ & $\sigma_0/\sigma_c$\\
\midrule
2 & 6 & \texttt{br6} & $3$ & $6(n+1)+3$ & $1.024$\\
2 & 7 & \texttt{br7} & $15$ & $8(n+1)+5$ & $1.024$\\
2 & 8 & \texttt{br8} & $1$ & $10(n+1)+2$ & $1.692$\\
2 & 9 & \texttt{br9} & $2$ & $8$ & $1.692$\\
2 & 9 & \texttt{br9q} & $\frac{8}{3}$ & $3(n+1)+6$ & $1.692$\\
2 & 10 & \texttt{br10} & $7$ & $3(n+1)+16$ & $1.692$\\
\addlinespace[2pt]
3 & 32 & \texttt{br32} & $1$ & $(n+1)^{3}+(n+1)^{2}+20$ & $4.310$\\
3 & 35 & \texttt{br35full} & $\frac{15}{8}$ & $32$ & $4.310$\\
3 & 35 & \texttt{br35} & $2$ & $15(n+1)^{2}+35$ & $4.310$\\
3 & 36 & \texttt{br36} & $3$ & $22(n+1)^{2}+40$ & $4.310$\\
3 & 37 & \texttt{br37} & $3$ & $15(n+1)^{2}+50$ & $4.310$\\
3 & 39 & \texttt{ub9} & $12$ & $(n+1)^{3}+(n+1)^{2}+4$ & $4.310$\\
\addlinespace[2pt]
4 & 125 & \texttt{br125} & $1$ & $(n+1)^{4}+21(n+1)^{3}+161$ & $14.828$\\
4 & 133 & \texttt{br133} & $2$ & $(n+1)^{4}+70(n+1)^{3}+163$ & $14.828$\\
4 & 141 & \texttt{br141} & $3$ & $(n+1)^{4}+36(n+1)^{3}+245$ & $14.828$\\
4 & 150 & \texttt{ub12} & $11$ & $10(n+1)^{3}+24$ & $14.702$\\
\bottomrule
\end{tabular}
\caption{Leading coefficients at $n=126,25,11$ for $N=2,3,4$, respectively.
The ratios $\sigma_0/\sigma_c$ are rounded to three decimal places.
Construction identifiers refer to the data in
Appendix~\ref{apx:construction-certificates}.}
\label{tab:sigma-approximate}
\end{table}

\subsection{Comparisons at integer dimensions}
The following comparisons use the same dimensions and number of independent
products, at dimensions where the comparison bases apply recursively without padding. For $n=12$, the nine-product construction \texttt{br9} has budget
$7{,}054{,}567$, compared with $9{,}000{,}000$ from two tensor levels
of the Bini square base for each product. The fourteen-product construction \texttt{ex14} has budget $13^9=10{,}604{,}499{,}373$, compared with $15{,}748{,}096{,}000$
from three tensor levels of the exact size-$12$ base for each product.

For a balanced family with $q=n^N$, the asymptotic-sum inequality is
\begin{equation}
p n^{N\omega}\le r(n),\qquad
\omega\le\omega_{\rm eff}(p,q,r(n)).
\label{eq:omega-asi}
\end{equation}
Appendix~\ref{apx:ast} gives a recursive argument for the exact case.

\subsection{A border-rank decomposition of the balanced little-CW cube}
\label{apx:N1br}
The balanced restriction of $\phi_{\mathrm{CW}}^{\otimes3}$ consists
of six matrix-multiplication forms sharing variable blocks; it is not a direct
sum of six independent products. The following decomposition reduces the
border-rank bound for this shared-variable tensor.
In a one-subscript entry of $X$, $Y$, or $Z$, the omitted subscript is summed over.
The restriction of the third tensor power of the little CW tensor to balanced tripartitions admits the following border-rank decomposition,
{\small
\begin{align}
\psi_{\mathrm{CW}}(X,Y,Z) = &\lim_{\lambda\to 0}
 \sum_{i,j,k}
(X^{(\{0\})}_{jk} +X^{(\{1\})}_{ik} +X^{(\{2\})}_{ij})
(Y^{(\{0\})}_{jk} +Y^{(\{1\})}_{ik} +Y^{(\{2\})}_{ij})
(Z^{(\{0\})}_{jk} +Z^{(\{1\})}_{ik} +Z^{(\{2\})}_{ij}) \\
-&\lambda^{-1}\sum_{j,k}
(X^{(\{0\})}_{jk} +\lambda X^{(\{1\})}_{k} +\lambda X^{(\{2\})}_{j})
(Y^{(\{0\})}_{jk} +\lambda Y^{(\{1\})}_{k} +\lambda Y^{(\{2\})}_{j})
(Z^{(\{0\})}_{jk} +\lambda Z^{(\{1\})}_{k} +\lambda Z^{(\{2\})}_{j})  \\
 -&\lambda^{-1}\sum_{i,k}
(\lambda X^{(\{0\})}_{k} +X^{(\{1\})}_{ik} +\lambda X^{(\{2\})}_{i})
(\lambda Y^{(\{0\})}_{k} +Y^{(\{1\})}_{ik} +\lambda Y^{(\{2\})}_{i})
(\lambda Z^{(\{0\})}_{k} +Z^{(\{1\})}_{ik} +\lambda Z^{(\{2\})}_{i}) \\
-& \lambda^{-1}\sum_{i,j}
(\lambda X^{(\{0\})}_{j} +\lambda X^{(\{1\})}_{i} +X^{(\{2\})}_{ij})
(\lambda Y^{(\{0\})}_{j} +\lambda Y^{(\{1\})}_{i} +Y^{(\{2\})}_{ij})
(\lambda Z^{(\{0\})}_{j} +\lambda Z^{(\{1\})}_{i} +Z^{(\{2\})}_{ij}) \\
-&(n-\lambda^{-1})\bigg[\sum_{j,k}
X^{(\{0\})}_{jk} Y^{(\{0\})}_{jk}Z^{(\{0\})}_{jk}
+\sum_{i,k}
X^{(\{1\})}_{ik} Y^{(\{1\})}_{ik}Z^{(\{1\})}_{ik}
+\sum_{i,j}
X^{(\{2\})}_{ij} Y^{(\{2\})}_{ij}Z^{(\{2\})}_{ij}\bigg].
\end{align}}
For every fixed $\lambda\neq0$, the displayed expression is a sum of
$n^3+6n^2$ products of three linear forms.  To verify the limit, classify a
monomial by the three superscripts selected from $\{0,1,2\}$.  If the three
superscripts are distinct, each of the three $\lambda^{-1}$ sums contains at
least two factors multiplied by $\lambda$, and the monomial survives with
coefficient one.  If exactly two superscripts agree, the sum corresponding to
the repeated superscript contains exactly one factor multiplied by $\lambda$
and cancels the monomial from the first line; the other two sums vanish in the
limit.  If all three superscripts equal $s$, the first line contributes $n$
times the diagonal product of the $s$-block, the corresponding
$\lambda^{-1}$ sum contributes $-\lambda^{-1}$ times that product, and the
last line contributes $-(n-\lambda^{-1})$ times it.  These contributions sum
to zero.  Hence precisely the six monomials with distinct superscripts remain,
which are the balanced restriction of $\phi_{\mathrm{CW}}^{\otimes3}$.
Thus
\begin{equation}
\brankop(\psi_{\mathrm{CW}})\leq n^3+6n^2.
\end{equation}
The third tensor power of the decomposition of $\phi_{\mathrm{CW}}$, or the
trilinear aggregation with centering, padding, and weighting, instead gives
$(n+2)^3=n^3+6n^2+12n+8$.

\section{Unequal-shape complete-tripartition constructions}
\label{apx:unequal-cw}

\begin{table}[H]
\centering
\small
\resizebox{\textwidth}{!}{\begin{tabular}{c|l|c|c|c|c|c}
products & shape counts & $h\le$ & $\theta\ge$
& correction counts & optimizing $n$ & $\omega_{\rm eff}$\\
\hline
10 & $6(2,2,2)+4(1,2,3)$ & 1 & $1/2$ & $(6,3)$ & 17 & 2.706\\
13 & $3(2,2,2)+10(1,2,3)$ & 2 & $1/3$ & $(6,6)$ & 14 & 2.660\\
22 & $8(2,2,2)+10(1,2,3)+4(1,1,4)$ & 211 & $1/212$ & $(6,12)$ & 11 & 2.560
\end{tabular}}
\caption{Complete-tripartition families of equal volume and unequal shapes.
A shape $(r,s,t)$ abbreviates all products whose ordered exponent triple is a
permutation of that triple. The height column gives an upper bound,
and the accuracy column its corresponding lower bound. The final column gives the asymptotic-sum bound on $\omega$, rounded to three decimal places.}
\label{tab:mixed-complete}
\end{table}

The correction counts determine each budget by
\eqref{eq:partial-correction-budget}. In particular, for the 22-product
family every product has volume $n^6$, so
\begin{equation}
22n^{2\omega}\le \rho_{22}(n)
 =(n+1)^6+6(n+1)^5+12(n+1)^4.
\label{eq:asi-22}
\end{equation}
\label{eq:rho-22}
The minimizing integer $n=11$ yields $\omega\le2.560147\ldots$;
its limiting retained fraction is at least $1/212$.

The full tripartition lists and exponents are in the auxiliary directory
\path{anc/constructions/}, in the files \path{general10.txt},
\path{general13.txt}, and \path{general22.txt}, respectively.

\section{Incomplete tripartitions of the big CW tensor}
\label{apx:incomplete-cw}
\subsection{Constructions and multiplication counts}
For the incomplete families in Table~\ref{tab:bigcw}, $m_g$ counts products of volume $n^g$, and the exponent bound is determined by
\begin{equation}
\sum_g m_g n^{\omega g/3}\le(n+2)^d.
\label{eq:incomplete-asi}
\end{equation}

\begin{table}[H]
\centering
\small
\resizebox{\textwidth}{!}{\begin{tabular}{c|c|l|c|c}
$d$ & products & counts by covered size & optimizing $n$ & $\omega_{\rm eff}$\\
\hline
6 & 51 & $m_6=15,\ m_5=2,\ m_4=4,\ m_3=18,\ m_2=8,\ m_1=1,\ m_0=3$ & 14 & 2.636\\
9 & 68 & $m_9=57,\ m_8=11$ & 14 & 2.639
\end{tabular}}
\caption{Conflict-free families of incomplete tripartitions.  Products have unequal
volumes and often unequal shapes; $m_s$ counts products with $s$ covered
coordinates. The exponent follows from
\eqref{eq:incomplete-asi} and is rounded to three decimal places.}
\label{tab:bigcw}
\end{table}

The constructions of Table~\ref{tab:bigcw} use the big-CW
restriction of Section~\ref{sec:cw-contractions}. Their border-rank bounds follow
from the restriction identity and the counts by covered size. Complete data
are in \texttt{inc6p51} and \texttt{inc9p68},
respectively, in the directory indexed by Appendix~\ref{apx:construction-certificates}.
These and the unequal-shape constructions of Table~\ref{tab:mixed-complete}
are tabulated comparisons, not curves in the figures.

\subsection{Norms of the incomplete big-CW restrictions}
\label{apx:incomplete-sigma}
For the rational degeneration defining $\phi_{\rm BCW}$ in Section~\ref{sec:fmm-background},
distribute the absolute scalar coefficient of each rank-one term
equally among its three factors. Put its sign in the output factor.
For $0<\lambda<1/n$, set $t=1-n\lambda$. Up to column signs, each factor has columns
\begin{equation}
\lambda^{-2/3}(e_0+\lambda e_i)\ (1\le i\le n),\qquad
\lambda^{-1}\left(e_0+\lambda^2\sum_{i=1}^n e_i\right),\qquad
t^{1/3}\lambda^{-1}e_0+\lambda^2t^{-2/3}e_{n+1}.
\label{eq:bigcw-balanced-columns}
\end{equation}
Summing their outer products and restricting to the span of $e_0$,
$n^{-1/2}\sum_{i=1}^n e_i$, and $e_{n+1}$ gives
\begin{equation}
H(\lambda)=
\begin{pmatrix}
n\lambda^{-4/3}+\lambda^{-2}(1+t^{2/3}) &
 \sqrt n(\lambda^{-1/3}+1) & \lambda t^{-1/3}\\
\sqrt n(\lambda^{-1/3}+1) & \lambda^{2/3}+n\lambda^2 & 0\\
\lambda t^{-1/3} &0&\lambda^4t^{-4/3}
\end{pmatrix}.
\end{equation}
The remaining $n-1$ directions in the middle block have eigenvalue
$\lambda^{2/3}$. A retained variable block has a word
$\alpha\in\{0,1,2\}^d$, determined by its scalar, active, and barred
coordinates of the variable blocks in Lemma~\ref{lem:bcw-power}.
For each variable space, form the matrix indexed by its retained words,
\begin{equation}
G_{\alpha\beta}(\lambda)=\prod_{j=0}^{d-1}H_{\alpha_j\beta_j}(\lambda).
\end{equation}
The factor's squared spectral norm is $\lambda_{\max}(G)$.
Indeed, decomposing each middle coordinate into its sum and orthogonal
complement splits the Gram matrix into invariant subspaces.
Each other invariant block is entrywise bounded by a principal submatrix
of $G$: an orthogonal middle-coordinate direction replaces $H_{11}$ by
$\lambda^{2/3}\le H_{11}$. All entries are nonnegative, so monotonicity
of the spectral radius and the principal-submatrix bound give no larger
eigenvalue.
The product of the three square roots gives $\sigma(\Xi(\lambda))$ without
forming the tensor-power factors.

For the two incomplete families of Table~\ref{tab:bigcw}, this yields
\[
\begin{array}{rr|l}
d&p& \sigma(\Xi(\lambda))\text{ to leading order}\\\hline
6&51&64\sqrt2\lambda^{-34/3}\\
9&68&512\lambda^{-15}
\end{array}
\]
These values use the displayed rational big-CW realization and its
original parameter, with the scalar weights balanced as above.

\section{Construction data}
\label{apx:construction-certificates}
\label{apx:accuracy-certificates}
The indices below locate the supplied data for the numerical results.
Each construction has a readable file \texttt{<identifier>.txt} in
\path{anc/constructions/}, listing its tripartitions and exponents or the
corresponding incomplete tripartitions. The correction rule is given in the
file, with explicit correction sets when required. Table~\ref{tab:construction-index} locates all centered
constructions used in the figures or discussed in Section~\ref{sec:construction-parameters}.
Table~\ref{tab:weighted-construction-index} separately locates the weighted
constructions used in Figure~\ref{fig:weighted-capacity}; their supports and
gradings need not agree with centered constructions having the same $N,p$.
The machine-readable index is \path{anc/auxiliary/construction_usage.csv}.

\begingroup\small
\begin{longtable}{rrcll}
\caption{Data files and figure appearances of the centered square-product
constructions. The first twenty rows are the plotted families in
Table~\ref{tab:construction-parameters}; the last is the full-correction
variant discussed in Section~\ref{sec:construction-parameters}. A dash means
that the construction is not drawn in a figure.}
\label{tab:construction-index}\\
\toprule
$N$ & $p$ & $\theta$ & File identifier & Figures\\\midrule\endfirsthead
\toprule
$N$ & $p$ & $\theta$ & File identifier & Figures\\\midrule\endhead
\bottomrule\endfoot
2 & 6 & $1$ & \texttt{ex6} & \ref{fig:embedding-savings}, \ref{fig:equalsize-exact}\\
2 & 6 & $1/4$ & \texttt{br6} & \ref{fig:embedding-savings}, \ref{fig:equalsize-bounded}\\
2 & 7 & $1/16$ & \texttt{br7} & \ref{fig:embedding-savings}, \ref{fig:equalsize-border}\\
2 & 8 & $1/2$ & \texttt{br8} & \ref{fig:embedding-savings}, \ref{fig:equalsize-bounded}\\
2 & 9 & $1/3$ & \texttt{br9} & \ref{fig:embedding-savings}, \ref{fig:equalsize-bounded}\\
2 & 9 & $3/11$ & \texttt{br9q} & \ref{fig:embedding-savings}, \ref{fig:equalsize-bounded}\\
2 & 10 & $1/8$ & \texttt{br10} & \ref{fig:embedding-savings}, \ref{fig:equalsize-border}\\
3 & 14 & $1$ & \texttt{ex14} & \ref{fig:embedding-savings}, \ref{fig:equalsize-exact}\\
3 & 28 & $1$ & \texttt{ex28} & \ref{fig:embedding-savings}, \ref{fig:equalsize-exact}\\
3 & 32 & $1/2$ & \texttt{br32} & \ref{fig:embedding-savings}, \ref{fig:equalsize-bounded}\\
3 & 35 & $1/3$ & \texttt{br35} & \ref{fig:embedding-savings}, \ref{fig:equalsize-bounded}\\
3 & 36 & $1/4$ & \texttt{br36} & \ref{fig:embedding-savings}, \ref{fig:equalsize-bounded}\\
3 & 37 & $1/4$ & \texttt{br37} & \ref{fig:embedding-savings}, \ref{fig:equalsize-bounded}\\
3 & 39 & $1/13$ & \texttt{ub9} & \ref{fig:embedding-savings}, \ref{fig:equalsize-border}\\
4 & 38 & $1$ & \texttt{ex38} & \ref{fig:embedding-savings}, \ref{fig:equalsize-exact}\\
4 & 105 & $1$ & \texttt{ex105} & \ref{fig:embedding-savings}, \ref{fig:equalsize-exact}, \ref{fig:equalsize-bounded}, \ref{fig:equalsize-border}\\
4 & 125 & $1/2$ & \texttt{br125} & \ref{fig:embedding-savings}, \ref{fig:equalsize-bounded}\\
4 & 133 & $1/3$ & \texttt{br133} & \ref{fig:embedding-savings}, \ref{fig:equalsize-bounded}\\
4 & 141 & $1/4$ & \texttt{br141} & \ref{fig:embedding-savings}, \ref{fig:equalsize-bounded}\\
4 & 150 & $1/12$ & \texttt{ub12} & \ref{fig:embedding-savings}, \ref{fig:equalsize-border}\\
\midrule
3 & 35 & $8/23$ & \texttt{br35full} & --\\
\end{longtable}\endgroup

\begingroup\small
\begin{longtable}{rrll}
\caption{Data files for every cumulative level of Figure~\ref{fig:weighted-capacity}.
Each row uses the weighted multiplication budget $(n+2)^{3N}$. The accuracy
column gives the plotted requirement, not necessarily the larger fraction
attained by the stored grading. ``None'' places no accuracy restriction;
``all'' means that the same exact construction supplies every level.}
\label{tab:weighted-construction-index}\\
\toprule
$N$ & $p$ & Accuracy requirement & File identifier\\\midrule\endfirsthead
\toprule
$N$ & $p$ & Accuracy requirement & File identifier\\\midrule\endhead
\bottomrule\endfoot
1 & 2 & all & \texttt{w1e}\\
2 & 6 & exact & \texttt{w2e}\\
2 & 8 & $1/2$ & \texttt{w2h}\\
2 & 9 & $1/3$, $1/4$ & \texttt{w2t}\\
2 & 10 & none & \texttt{w2u}\\
3 & 28 & exact & \texttt{w3e}\\
3 & 32 & $1/2$ & \texttt{w3h}\\
3 & 35 & $1/3$ & \texttt{w3t}\\
3 & 37 & $1/4$ & \texttt{w3q}\\
3 & 39 & none & \texttt{w3u}\\
4 & 105 & exact & \texttt{w4e}\\
4 & 125 & $1/2$ & \texttt{w4h}\\
4 & 133 & $1/3$ & \texttt{w4t}\\
4 & 141 & $1/4$ & \texttt{w4q}\\
4 & 150 & none & \texttt{w4u}\\
\end{longtable}\endgroup

The unequal-shape constructions of Table~\ref{tab:mixed-complete} are
\texttt{general10}, \texttt{general13}, and \texttt{general22}; the incomplete
constructions of Table~\ref{tab:bigcw} are \texttt{inc6p51} and
\texttt{inc9p68}, in the displayed order. These constructions are discussed
in Section~\ref{sec:extensions} but are not drawn in figures.

\end{revision}

\bibliographystyle{abbrv}
\bibliography{main}

@article{solomonik2021fast,
  title={Fast bilinear algorithms for symmetric tensor contractions},
  author={Solomonik, Edgar and Demmel, James},
  journal={Computational Methods in Applied Mathematics},
  volume={21},
  number={1},
  pages={211--231},
  year={2021},
  publisher={De Gruyter}
}

@article{coppersmith1990matrix,
  title={Matrix multiplication via arithmetic progressions},
  author={Coppersmith, Don and Winograd, Shmuel},
  journal={Journal of Symbolic Computation},
  volume={9},
  number={3},
  pages={251--280},
  year={1990},
  publisher={Academic Press, Inc. Orlando, FL, USA}
}

@article{blaser2013fast,
  title={Fast matrix multiplication},
  author={Bl{\"a}ser, Markus},
  journal={Theory of Computing},
  pages={1--60},
  year={2013},
  publisher={Theory of Computing Exchange}
}

@article{schonhage1981partial,
  author = {Arnold Sch{\"o}nhage},
  title = {Partial and Total Matrix Multiplication},
  journal = {SIAM Journal on Computing},
  volume = {10},
  number = {3},
  pages = {434--455},
  year = {1981}
}

@article{salemspencer1942,
  title   = {On sets of integers which contain no three terms in arithmetical progression},
  author  = {Salem, Rapha{\"e}l and Spencer, Donald C.},
  journal = {Proceedings of the National Academy of Sciences of the United States of America},
  volume  = {28},
  number  = {12},
  pages   = {561--563},
  year    = {1942},
  doi     = {10.1073/pnas.28.12.561}
}

@article{conner2021:cc,
  author  = {Conner, Austin and Gesmundo, Fulvio and Landsberg, Joseph M. and Ventura, Emanuele},
  title   = {Rank and border rank of {K}ronecker powers of tensors and {S}trassen's laser method},
  journal = {computational complexity},
  volume  = {31},
  number  = {1},
  pages   = {1},
  year    = {2022},
  doi     = {10.1007/s00037-021-00217-y}
}

@article{lancaster1965helmert,
  author  = {Lancaster, H. O.},
  title   = {The {Helmert} Matrices},
  journal = {The American Mathematical Monthly},
  volume  = {72},
  number  = {1},
  pages   = {4--12},
  year    = {1965},
  doi     = {10.1080/00029890.1965.11970483}
}

@article{strassen1969gaussian,
   author  = {Strassen, Volker},
   title   = {Gaussian elimination is not optimal},
   journal = {Numerische Mathematik},
   volume  = {13},
   number  = {4},
   pages   = {354--356},
   year    = {1969}
 }

@inproceedings{pan1978strassen,
   author    = {Pan, V. Ya.},
   title     = {Strassen's algorithm is not optimal: Trilinear
                technique of aggregating, uniting and canceling for
                constructing fast algorithms for matrix operations},
   booktitle = {Proceedings of the 19th Annual Symposium on
                Foundations of Computer Science},
   pages     = {166--176},
   year      = {1978}
 }

@article{laderman1976noncommutative,
  author  = {Laderman, Julian D.},
  title   = {A noncommutative algorithm for multiplying $3\times3$ matrices using 23 multiplications},
  journal = {Bulletin of the American Mathematical Society},
  volume  = {82},
  number  = {1},
  pages   = {126--128},
  year    = {1976}
}

@article{bini1979n2,
  author  = {Bini, Dario and Capovani, Milvio and Romani, Francesco and Lotti, Grazia},
  title   = {$O(n^{2.7799})$ complexity for $n\times n$ approximate matrix multiplication},
  journal = {Information Processing Letters},
  volume  = {8},
  number  = {5},
  pages   = {234--235},
  year    = {1979}
}

@article{smirnov2013bilinear,
  author  = {Smirnov, A. V.},
  title   = {The bilinear complexity and practical algorithms for matrix multiplication},
  journal = {Computational Mathematics and Mathematical Physics},
  volume  = {53},
  number  = {12},
  pages   = {1781--1795},
  year    = {2013}
}

@article{christandl2019graphtensors,
  title         = {Asymptotic tensor rank of graph tensors: beyond matrix multiplication},
  author        = {Christandl, Matthias and Vrana, P{\'e}ter and Zuiddam, Jeroen},
  journal       = {computational complexity},
  volume        = {28},
  number        = {1},
  pages         = {57--111},
  year          = {2019},
  doi           = {10.1007/s00037-018-0172-8},
  eprint        = {1609.07476},
  archivePrefix = {arXiv}
}

@article{strassen1991degeneration,
  author  = {Strassen, Volker},
  title   = {Degeneration and complexity of bilinear maps: {S}ome asymptotic spectra},
  journal = {Journal f{\"u}r die reine und angewandte Mathematik},
  volume  = {413},
  pages   = {127--180},
  year    = {1991}
}

@article{christandl2019surgery,
  title         = {Tensor surgery and tensor rank},
  author        = {Christandl, Matthias and Zuiddam, Jeroen},
  journal       = {computational complexity},
  volume        = {28},
  number        = {1},
  pages         = {27--56},
  year          = {2019},
  doi           = {10.1007/s00037-018-0164-8},
  eprint        = {1606.04085},
  archivePrefix = {arXiv}
}

@article{abdelfattah2021batched,
  author  = {Ahmad Abdelfattah and Timothy Costa and Jack Dongarra and Mark Gates and Azzam Haidar and Sven Hammarling and Nicholas J. Higham and Jakub Kurzak and Piotr Luszczek and Stanimire Tomov and Mawussi Zounon},
  title   = {A Set of Batched Basic Linear Algebra Subprograms and {LAPACK} Routines},
  journal = {ACM Transactions on Mathematical Software},
  volume  = {47},
  number  = {3},
  pages   = {21:1--21:23},
  year    = {2021},
  doi     = {10.1145/3431921}
}

@article{rupniewski2022directsum,
  author        = {Filip Rupniewski},
  title         = {Tensor Rank of the Direct Sum of Two Copies of $2\times2$ Matrix Multiplication Tensor Is 14},
  journal       = {arXiv preprint arXiv:2209.11040},
  year          = {2022},
  eprint        = {2209.11040},
  archivePrefix = {arXiv}
}

@article{bini1980approximate,
  author  = {Dario Bini and Grazia Lotti and Francesco Romani},
  title   = {Approximate Solutions for the Bilinear Form Computational Problem},
  journal = {SIAM Journal on Computing},
  volume  = {9},
  number  = {4},
  pages   = {692--697},
  year    = {1980},
  doi     = {10.1137/0209053}
}

@article{demmel2007stable,
  author  = {James Demmel and Ioana Dumitriu and Olga Holtz and Robert Kleinberg},
  title   = {Fast Matrix Multiplication Is Stable},
  journal = {Numerische Mathematik},
  volume  = {106},
  number  = {2},
  pages   = {199--224},
  year    = {2007},
  doi     = {10.1007/s00211-007-0061-6}
}

@article{shitov2019directsum,
  author  = {Yaroslav Shitov},
  title   = {Counterexamples to {Strassen}'s Direct Sum Conjecture},
  journal = {Acta Mathematica},
  volume  = {222},
  number  = {2},
  pages   = {363--379},
  year    = {2019},
  doi     = {10.4310/ACTA.2019.v222.n2.a3}
}

@article{pan1982implicit,
  author  = {Victor Ya. Pan},
  title   = {Trilinear Aggregating with Implicit Canceling for a New Acceleration of Matrix Multiplication},
  journal = {Computers \& Mathematics with Applications},
  volume  = {8},
  number  = {1},
  pages   = {23--34},
  year    = {1982},
  doi     = {10.1016/0898-1221(82)90037-2}
}

@misc{schwartz2025faster,
  author        = {Oded Schwartz and Eyal Zwecher},
  title         = {Towards Faster Feasible Matrix Multiplication by Trilinear Aggregation},
  year          = {2025},
  eprint        = {2508.01748},
  archivePrefix = {arXiv},
  primaryClass  = {cs.DS}
}

@inproceedings{benson2015framework,
  author    = {Austin R. Benson and Grey Ballard},
  title     = {A Framework for Practical Parallel Fast Matrix Multiplication},
  booktitle = {Proceedings of the 20th ACM SIGPLAN Symposium on Principles and Practice of Parallel Programming},
  pages     = {42--53},
  year      = {2015},
  doi       = {10.1145/2688500.2688513}
}

@misc{dumas2025automated,
  author        = {Jean-Guillaume Dumas and Cl\'ement Pernet and Alexandre Sedoglavic},
  title         = {Towards Automated Generation of Fast and Accurate Algorithms for Recursive Matrix Multiplication},
  year          = {2025},
  eprint        = {2506.19405},
  archivePrefix = {arXiv},
  primaryClass  = {cs.NA}
}

@article{ballard2016stability,
  author  = {Grey Ballard and Austin R. Benson and Alex Druinsky and Benjamin Lipshitz and Oded Schwartz},
  title   = {Improving the Numerical Stability of Fast Matrix Multiplication},
  journal = {SIAM Journal on Matrix Analysis and Applications},
  volume  = {37},
  number  = {4},
  pages   = {1382--1418},
  year    = {2016},
  doi     = {10.1137/15M1032168}
}

@article{pan2017neighbourhood,
  author  = {Victor Ya. Pan},
  title   = {Fast Matrix Multiplication and Its Algebraic Neighbourhood},
  journal = {Sbornik: Mathematics},
  volume  = {208},
  number  = {11},
  pages   = {1661--1704},
  year    = {2017},
  doi     = {10.1070/SM8833}
}

@article{borstnik2014sparse,
  author = {Urban Bor{\v s}tnik and Joost VandeVondele and Val{\'e}ry Weber and J{\"u}rg Hutter},
  title = {Sparse matrix multiplication: The distributed block-compressed sparse row library},
  journal = {Parallel Computing},
  volume = {40},
  number = {5--6},
  pages = {47--58},
  year = {2014},
  doi = {10.1016/j.parco.2014.03.012}
}

@misc{dumas2025noncomplex,
  author        = {Jean-Guillaume Dumas and Cl\'ement Pernet and Alexandre Sedoglavic},
  title         = {A non-commutative algorithm for multiplying $4\times4$ matrices using 48 non-complex multiplications},
  year          = {2025},
  eprint        = {2506.13242},
  archivePrefix = {arXiv},
  primaryClass  = {cs.SC},
  howpublished  = {arXiv:2506.13242},
}

@misc{khoruzhii2026lita,
  author = {Kirill Khoruzhii and Luzian Serafin and Patrick Gel{\ss} and Sebastian Pokutta},
  title  = {Local Improvements to Trilinear Aggregation},
  year   = {2026},
  howpublished = {\url{https://github.com/khoruzhii/lita}},
  note   = {Rational schemes and software, commit \texttt{c1dd9225df98}, accessed 21 September 2026; associated manuscript in preparation}
}

@misc{sedoglavic2021five,
  author        = {Alexandre Sedoglavic and Alexey V. Smirnov},
  title         = {The tensor rank of $5\times5$ matrices multiplication is bounded by 98 and its border rank by 89},
  year          = {2021},
  eprint        = {2101.12568},
  archivePrefix = {arXiv},
  primaryClass  = {cs.CC},
  howpublished  = {arXiv:2101.12568},
}

@misc{bensonballard2015factors,
  author = {Austin R. Benson and Grey Ballard},
  title = {Fast Matrix Multiplication: Algorithm Coefficient Files},
  year = {2015},
  howpublished = {\url{https://github.com/arbenson/fast-matmul}},
  note = {Coefficient files at commit \texttt{132358be7a6a}; Bini, Sch\"onhage, and Smirnov approximate formulas}
}

@article{bini1980stability,
  author = {Bini, Dario and Lotti, Grazia},
  title = {Stability of Fast Algorithms for Matrix Multiplication},
  journal = {Numerische Mathematik}, volume = {36}, number = {1},
  pages = {63--72}, year = {1980}, doi = {10.1007/BF01395989}
}

@article{dai2023nuclear,
  author = {Dai, Zhen and Lim, Lek-Heng},
  title = {Numerical Stability and Tensor Nuclear Norm},
  journal = {Numerische Mathematik}, volume = {155}, number = {3--4},
  pages = {345--376}, year = {2023}, doi = {10.1007/s00211-023-01377-5}
}

@article{ju2020convolution,
  author = {Caleb Ju and Edgar Solomonik},
  title = {Derivation and Analysis of Fast Bilinear Algorithms for Convolution},
  journal = {SIAM Review},
  volume = {62}, number = {4}, pages = {743--777},
  year = {2020}, doi = {10.1137/19M1301059}
}

@article{alder1981associative,
  author  = {Alder, A. and Strassen, Volker},
  title   = {On the algorithmic complexity of associative algebras},
  journal = {Theoretical Computer Science},
  volume  = {15},
  number  = {2},
  pages   = {201--211},
  year    = {1981},
  doi     = {10.1016/0304-3975(81)90070-0}
}

@article{flavi2025symmetric,
  author  = {Flavi, Cosimo and Jelisiejew, Joachim and Micha{\l}ek, Mateusz},
  title   = {Symmetric Powers: Structure, Smoothability, and Applications},
  journal = {International Mathematics Research Notices},
  volume  = {2025},
  number  = {18},
  year    = {2025},
  doi     = {10.1093/imrn/rnaf277}
}

@inproceedings{alman2018further,
  author    = {Alman, Josh and Vassilevska Williams, Virginia},
  title     = {Further Limitations of the Known Approaches for Matrix Multiplication},
  booktitle = {9th Innovations in Theoretical Computer Science Conference (ITCS 2018)},
  series    = {Leibniz International Proceedings in Informatics (LIPIcs)},
  volume    = {94},
  pages     = {25:1--25:15},
  editor    = {Karlin, Anna R.},
  publisher = {Schloss Dagstuhl -- Leibniz-Zentrum f{\"u}r Informatik},
  year      = {2018},
  doi       = {10.4230/LIPIcs.ITCS.2018.25}
}
\end{document}